\documentclass[11pt]{article}
\usepackage{mainsty}

\title{Power properties of the two-sample test based on the\\ nearest neighbors graph}
\author{Rahul Raphael Kanekar\footnote{Correspondence : \href{mailto:rkanekar@stanford.edu}{rkanekar@stanford.edu}}\\Department of Statistics, Stanford University}

\newcommand{\appendixB}{\Cref{appendix::consistency_and_clts}}
\newcommand{\appendixC}{\Cref{appendix::detection_thresholds}}
\newcommand{\appendixD}{\Cref{appendix::additional_simulations}}

\begin{document}

\maketitle

\begin{abstract}
    In this paper, we study the problem of testing the equality of two multivariate distributions. One class of tests used for this purpose utilizes geometric graphs constructed using inter-point distances. So far, the asymptotic theory of these tests applies only to graphs which fall under the stabilizing graphs framework of \citet{penroseyukich2003weaklaws}. We study the case of the $K$-nearest neighbors graph where $K=k_N$ increases with sample size, which does not fall under the stabilizing graphs framework. Our main result gives detection thresholds for this test in parametrized families when $k_N = o(N^{1/4})$, thus extending the family of graphs where the theoretical behavior is known. We propose a 2-sided version of the test that removes an exponent gap that plagues the 1-sided test. Our result also shows that increasing the number of nearest neighbors boosts the power of the test. This provides theoretical justification for using denser graphs in testing the equality of two distributions.
\end{abstract}

\section{Introduction}\label{section::introduction}

Given data from two sources, a natural question to ask is whether the two sources can be considered different in some sense. This is a basic question of comparison that arises in a variety of scientific settings. From the viewpoint of a statistician, this is the two sample problem and can be stated as follows.

Let $\{X_1,\dots,X_{N_1}\}$ and $\{Y_1,\dots,Y_{N_2}\}$ be i.i.d samples from the distributions $F$ and $G$, respectively. The two sample problem is to test the hypotheses
$$H_0: F=G \quad \text{ v/s } \quad H_1: F\neq G.$$

In this paper, we focus on two sample tests that are nonparametric---they do not assume that $F,G$ belong to some parametrized family of distributions---and distribution free---under the null $F=G,$ the test is valid for any distribution $F$. While parametric approaches such as $t$-tests, likelihood ratio tests, score tests require the model to be well specified in order to be powerful (see \cite{lehmann2005testing}), nonparametric tests are often more reliable under weaker assumptions. For univariate distributions $F,G$, a host of nonparametric tests such as the two sample Kolmogorov-Smirnov test \cite{smirnov1939estimation}, Mann-Whitney test \cite{mann1947test} and the Wald--Wolfowitz runs test \cite{wald1940test}. Univariate tests are often rank-based; they rely on the fact that univariate data can be ranked (see \cite{lehmann1975statistical} for rank-based methods). Since there is no natural extension of ranks to multivariate data, it is difficult to generalize these tests to multivariate data.

This paper focuses on a class of multivariate two sample tests known as \emph{graph-based two sample tests}. The idea behind graph-based two sample tests is to utilize inter-point distances to construct a graph that can summarize the data, in place of ranks. \citet{weiss1960twosampletests} explored this approach first, but the resulting test was not distribution-free. Following this, \citet{friedmanrafskytest} introduced a two sample test using the Euclidean Minimal Spanning Tree (MST) constructed from the pooled data. In addition to being nonparametric and distribution free, this test extends the Wald--Wolfowitz runs test to higher dimensions, bolstering the idea of using graphs instead of ranks in higher dimensions. By changing the underlying graph, numerous alternate tests have been proposed. \citet{schilling1986multivariatetwosampletest} and \citet{henze1988multivariate} studied the test based on the $K$-nearest neighbors graph. Later, \citet{rosenbaum2005exact2sampletest} provided a test based on the minimal bipartite matching (the Cross-Match test) which is distribution free in finite samples. \citet{biswas2014distribution} proposed a test based on the Hamiltonian Cycle, while \citet{ruth2014new} considers the regular minimum-weight spanning subgraph. 

Key factors in the appeal of graph-based two sample tests are that they are versatile and easy to implement. Like many nonparametric tests, they can be implemented as a permutation test. When the sample size makes permutation tests infeasible, one can instead use the asymptotic normality under the null \cite{friedmanrafskytest,henze1988multivariate,henze1999multivariate} to obtain a distribution free test. As shown in \citet{bhattacharya2020detectionthresholds}, this holds for a general family of graphs. Going beyond Euclidean data, \citet{chen2017new} proposed a graph-based test for high-dimensional and object data which has good power in practice. Finally, graph-based two sample tests have found applications in a number of contexts where the problem at hand is one of comparing two populations \cite{chen2019sequential,chen2023graph,chenzhang2015graphbasedchangepoint,hsiao2016mapping,mukherjee2022distribution,ZHANG2020107285,zhao2006improved}. 

\subsection{Related work}\label{section::related_two_sample_tests}

Apart from graph-based tests, a large class of two sample tests target some metric or measure of dissimilarity between distributions. In particular, tests based on the energy distance \cite{aslan2005new,baringhaus2004new,szekely2004testing,SZEKELY20131249} and the kernel maximum mean discrepancy (MMD) \cite{chatterjee2025boosting,gretton2006kernel,gretton2012kernel,ramdas2017wasserstein,song2024generalized} have received considerable attention in recent times. The test statistics are typically computable in quadratic or sub-quadratic time, and have been shown to have good power in practice. However, as shown for instance in \citet{szekely2004testing} and \citet{gretton2012kernel}, the limiting distributions of these statistics under the null can be non-Gaussian, intractable, or dependent on the null distribution. These factors hinder the goal of obtaining a distribution free tests. Using a permutation test remains an option, but can be computationally expensive for larger sample sizes. Having said this, we refer the reader to more recent examples including \citet{chatterjee2025martingale} and \citet{shekhar2022permutation} that have proposed variations of the kernel MMD statistic that are asymptotically standard Gaussian under the null.

Other approaches include that of \citet{huang2024kernel}, where the authors propose another kernel based measure of dissimilarity and use geometric graphs to estimate it, with the resulting test being distribution free. \citet{ghosal2022multivariate} defines a notion of multivariate ranks and uses it to generalize many univariate, distribution-free tests to higher dimensions. In addition to these, other solutions draw on a host of techniques including projection-averaging \cite{kim2020robust}, ball-divergence measure \cite{pan2018ball}, binary classifiers \cite{kim2021classification,lopez2017revisiting}, kernel mean embeddings \cite{chwialkowski2015fast,jitkrittum2016interpretable} and neural networks \cite{pmlr-v119-liu20m}.

\subsection{Graph-based two sample tests}

We now define a general graph-based test statistic, provide some examples and show how the statistic is used in testing the null. To this end, we define some terminology.

A graph functional $\Gscr$ is a function that for any finite $S\subset \R^d$ defines a graph $\Gscr(S)$ with vertex set $S.$ The edge set of the graph $\Gscr(S)$ is denoted by $E(\Gscr(S))$. For simplicity, we will assume that $\Gscr(S)$ has no self loops or multi-edges. The graph $\Gscr(S)$ can be directed or undirected.

\begin{definition}\label{defn::graph_based_test} Let $\Xscr_{N_1} := \{X_1,\dots,X_{N_1}\}$ and $\Yscr_{N_2}:= \{Y_1,\dots,Y_{N_2}\}$ be i.i.d samples of size $N_1$ and $N_2$ from densities $f,g$, respectively. Let $\Gscr$ be a graph functional. The two sample test statistic based on $\Gscr$ is given by
$$T(\Gscr(\Xscr_{N_1}\cup\Yscr_{N_2})) = \sum_{i=1}^{N_1}\sum_{j=1}^{N_2}\indicator\{(X_i,Y_j)\in E(\Gscr(\Xscr_{N_1}\cup \Yscr_{N_2}))\}.$$
\end{definition}

The statistic $T(\Gscr)$ counts the number of edges in the graph that go across samples. To ease notation, we will often denote the statistic by $T(\Gscr)$, when the samples are clear.

\begin{figure}
\centering
\captionsetup{width=.9\linewidth}
\includegraphics[width = 5cm]{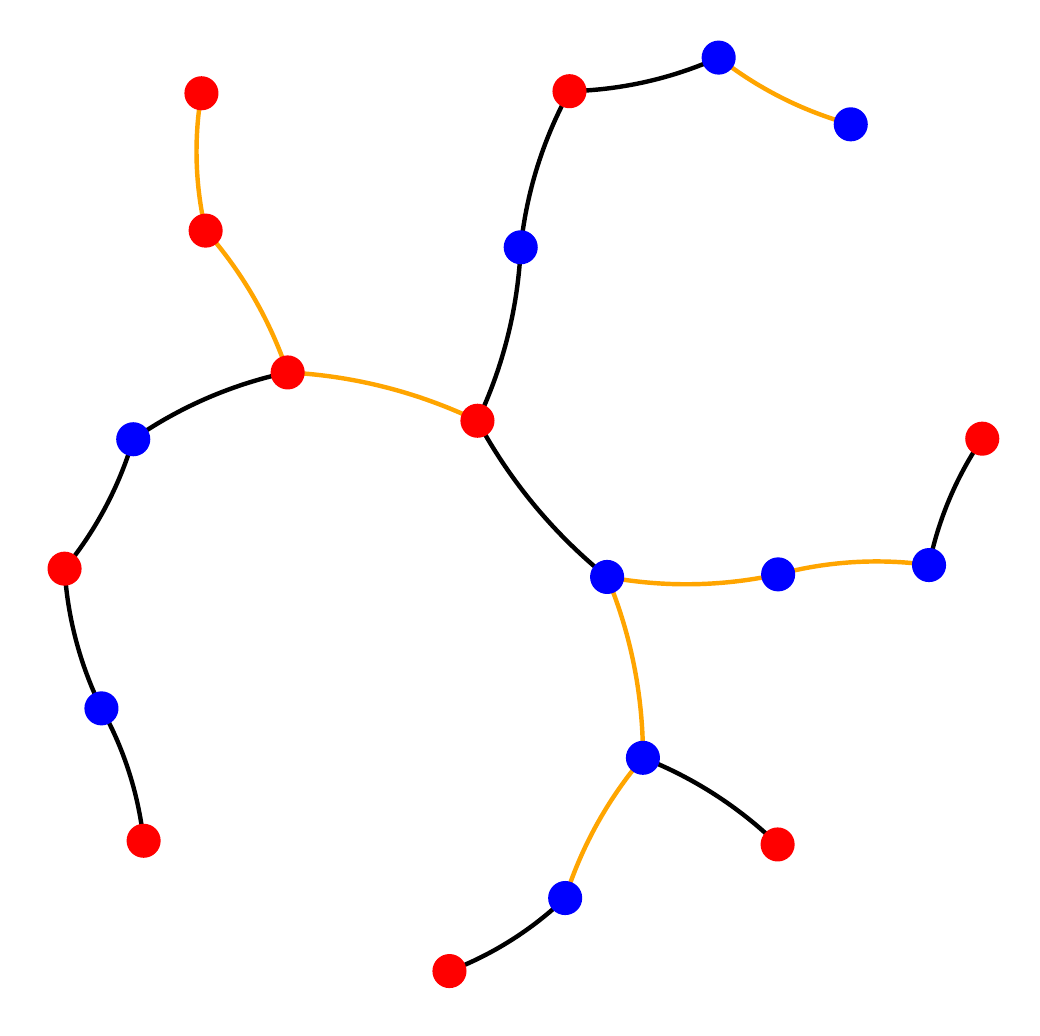}
\hspace{3cm}
\includegraphics[width = 5cm]{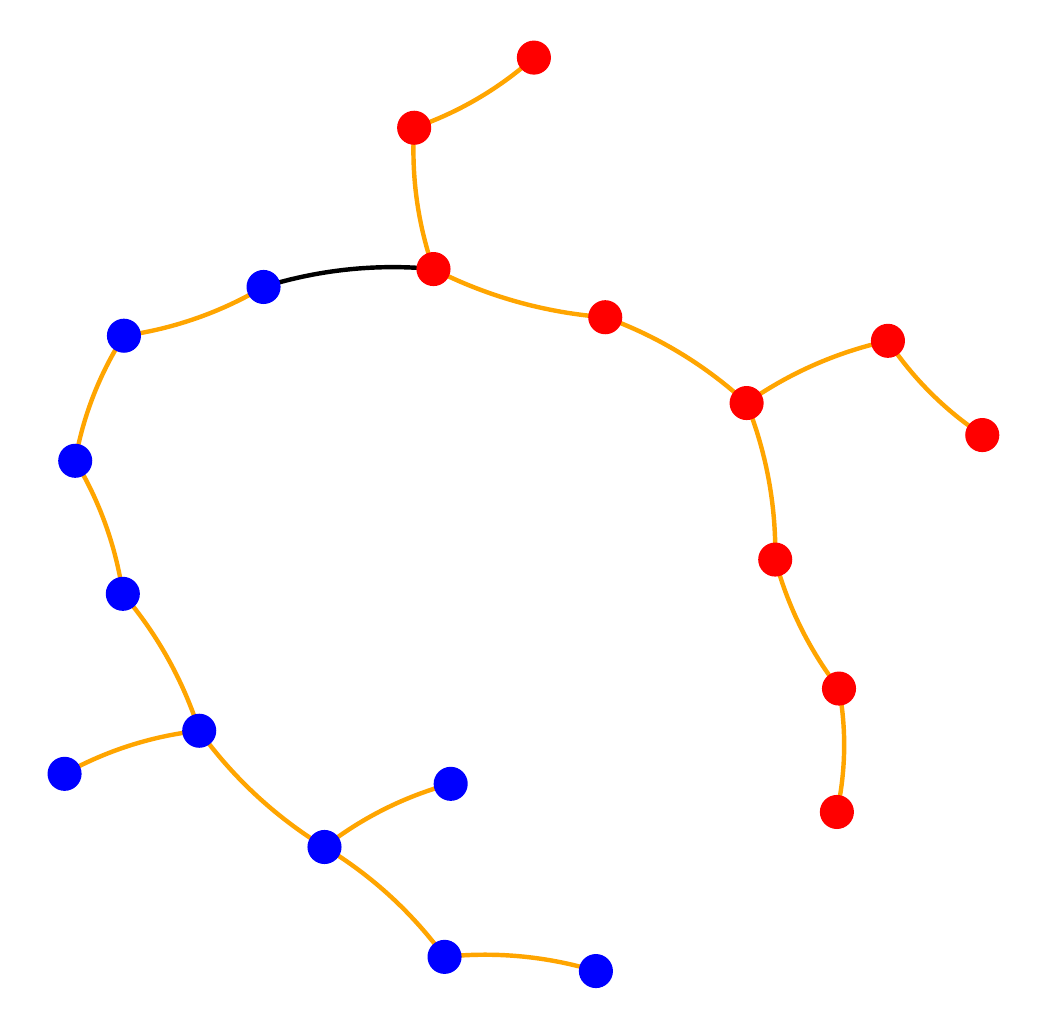} 
\caption{On the left is the undirected MST formed from 10 samples of $N(0,I_2)$(in red) and 10 samples of $N(0.2\cdot \indicator,I_2)$ (in blue). On the right is MST formed out of 10 samples each of $N(0,I_2)$ (in red) and $N(2\cdot \indicator,I_2)$ (in blue). The edges going across samples are colored black. Edges within samples are colored gold.}
\label{fig::mst_figures}
\end{figure}

\begin{example} (Friedman--Rafsky Test) The Friedman--Rafsky test results from taking $\Gscr$ to be the Euclidean minimal spanning tree. Given a finite set $S\subset \R^d,$ a spanning tree $\Tcal$ of $S$ is a connected, undirected graph with vertex set $S$ and no cycles. The length of a spanning tree is the sum of the lengths of all edges in the tree. A tree $\Tcal$ is called the Euclidean MST of $S$ if its length is at most the length of any other spanning tree $\Tcal'$ of $S.$ Thus, the Euclidean MST $\Tcal$ is a graph functional and yields a two sample test. Figure \ref{fig::mst_figures} gives examples of the Euclidean MST.

\end{example}

This test, introduced in \citet{friedmanrafskytest}, was one of the first graph-based tests proposed. One fact of particular interest is that when $d=1,$ the Friedman--Rafsky test gives exactly the Wald--Wolfowitz runs test. This is because, in dimension one the Euclidean MST is simply the line graph connecting adjacent points in the ranked data. While it was originally presented as a permutation test, the test statistic is also asymptotically normal (see \cite{bhattacharya2020detectionthresholds,friedmanrafskytest,henze1988multivariate}) which provides an alternate way of implementing the test.

\begin{example}\label{example::knn_test} ($K$-NN test) Given a finite set $S\subset \R^d$ and $K\in \N,$ the $K$-nearest neighbors graph $\Gscr_K(S)$ on $S$ is obtained by inserting the directed edge $(a,b)$ in $E(\Gscr_K(S))$ for $a,b \in S$, if and only if $b$ is one of the $K$ closest points to $a$ in Euclidean distance among the points in $S$. In this case, $\Gscr_K$ is a directed graph functional since each edge is from a point to one of its $K$-nearest neighbors. Figure \ref{fig::knn_figures} gives examples of the $K$-NN graph.

\begin{figure}[ht]
\centering
\captionsetup{width=.9\linewidth}
\includegraphics[width = 5cm]{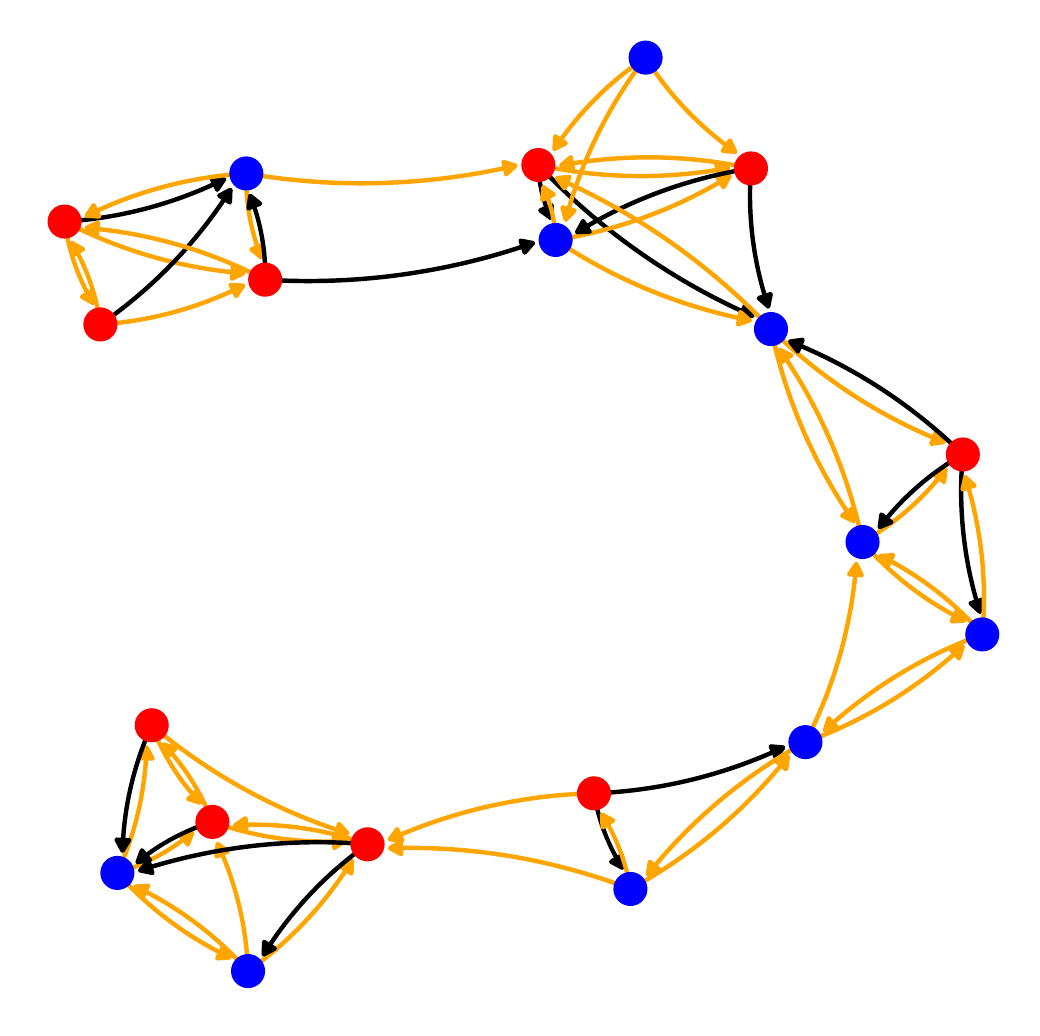}
\hspace{3cm}
\includegraphics[width = 5cm]{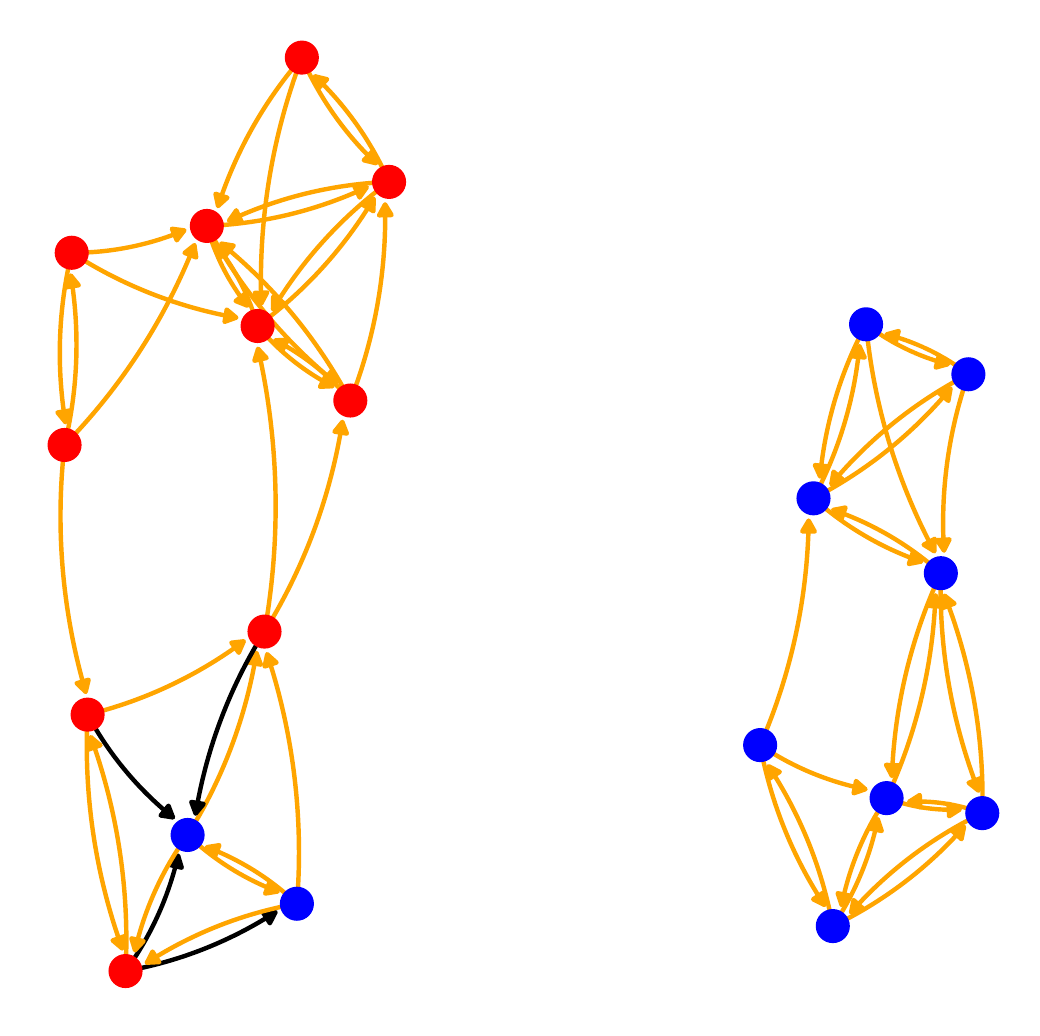}
\caption{On the left is the directed $3-$NN graph formed from 10 samples of $N(0,I_2)$(colored red) and 10 samples of $N(0.2,I_2)$ (colored blue). On the right is $3-$NN graph formed out of 10 samples each of $N(0,I_2)$(red) and $N(2,I_2)$(blue). The edges going from sample $1$ to sample $2$ are colored black. Edges within samples are colored gold.}
\label{fig::knn_figures}
\end{figure}
\end{example}

The test based on the $K$-nearest neighbors graph was introduced in \citet{henze1988multivariate} and will be our main object of interest.

Figures \ref{fig::mst_figures} and \ref{fig::knn_figures} highlight an interesting feature of graph-based two sample tests. Both figures contain data where the null (red points) in each case is $N(0,I_2)$ and the alternates (blue points) are $N(0.2 \cdot \indicator,I_2)$ and $N(2 \cdot \indicator,I_2)$, respectively. In both figures, there are much fewer cross-sample edges in the second case i.e. where the alternate is $N(2,I_2).$ This highlights an interesting principle namely, the more "different" two distributions are, the fewer cross-sample edges there are. Accordingly, the two sample test is often implemented as a 1-sided test where the null is rejected for
$$T(\Gscr(\Xscr_{N_1}\cup \Yscr_{N_2}))\leq \tau_\alpha,$$

where $\tau_\alpha$ is a threshold to be determined that will give a level-$\alpha$ test. We will elaborate on this further in Section \ref{section::consistency} when we look at the consistency of the test based on the $K$-nearest neighbors graph. We will also present a 2-sided version of the test and a key contribution of our work is a rigorous comparison of the power properties of the two tests.

\subsection{Poissonization}\label{section::poissonization}

In order to make the asymptotic behavior easier to analyze, we will be considering the Poissonized setting. This set up allows us to use the spatial independence of the Poisson process. Using the well known de-Poissonization techniques described in \citet{penrose2003random}, the results for the Poissonized setting can be used to derive analogous results for the un-Poissonized setting. See \Cref{remark:depoissonization} for more on de-Poissonization.

Let $f,g$ be densities on $\R^d$ and define $\phi_N(x):= \frac{N_1}{N}f(x) + \frac{N_2}{N}g(x)$ where $N_1+N_2 = N.$ Let $\Zcal_N:=\{Z_1,\dots,Z_{L_N}\}$ be the points sampled in the Poisson process with intensity function $N\phi_N = N_1 f + N_2 g.$ Here, the number of points $L_N$ is a Poisson random variable with parameter $N.$ For each point $z\in \Zcal_N,$ we assign the value $1$ or $2$ to the label $c_z$ with
\begin{equation}\label{eq::label_assignment}
    c_z =
    \begin{cases}
    1 \text{ with probability }\frac{N_1 f(x)}{N_1f(x) + N_2 g(x)},\\
    \\
    2 \text{ with probability } \frac{N_2 g(x)}{N_1f(x) + N_2g(x)}.
    \end{cases}
\end{equation}

The labels are assigned to all the points in $\Zcal_N$ independently. We will assume that $N_1,N_2$ are such that
$$\frac{N_1}{N} - p = o(N^{-\frac{1}{2}}), ~ \frac{N_2}{N} - q=o(N^{-\frac{1}{2}}),$$

where $q:= 1-p$. For the Poissonized setting we define the test statistic as
\begin{equation}\label{eq::poissonized_test_stat}
T(\Gscr_K(\Zcal_N)) = \sum_{x,y\in\Zcal_N}\psi(c_x,c_y)\indicator\{(x,y)\in E(\Gscr(\Zcal_N))\},
\end{equation}

\noindent where $\psi(c_x,c_y) = \indicator\{c_x=1,c_y=2\}.$ In Section \ref{subsection::clt_test_stat} we will show that if $\{k_N\}_N$ is a sequence of natural numbers such that $k_N = o(N^{1/4})$, then under $H_0$
$$\frac{1}{k_N\sqrt{N}}(T(\Gscr_{k_N}(\Zcal_N)) - \E_{H_0}(T(\Gscr_{k_N}(\Zcal_N)))) \to N(0,\sigma^2_0),$$

\noindent for some $\sigma^2_0>0.$ Hence, the test that rejects when
\begin{equation}\label{eq::poissonized_one_sided_test}
    \frac{1}{k_N\sqrt{N}}(T(\Gscr_{k_N}(\Zcal_N)) - \E_{H_0}(T(\Gscr_{k_N}(\Zcal_N))))\leq \sigma_0 z_\alpha,
\end{equation}

\noindent where $z_\alpha$ is the $\alpha$-quantile of the standard normal, is an asymptotically level-$\alpha$ test. Traditionally, this is the way the two sample graph-based test is implemented; as a 1-sided test. However, we will also be considering the 2-sided test i.e. the test that rejects when
\begin{equation}\label{eq::poissonized_two_sided_test}
\left|\frac{1}{k_N\sqrt{N}}(T(\Gscr_{k_N}(\Zcal_N)) - \E_{H_0}(T(\Gscr_{k_N}(\Zcal_N))))\right|\geq \sigma_0 z_{1-\alpha/2},
\end{equation}

where $z_{1-\alpha/2}$ is the $1-\alpha/2$ quantile of the standard normal.

\begin{remark}[Regarding de-Poissonization]
\label{remark:depoissonization} 

The Poissonized setting assumes that the sample size is $\pois(N)$, while the binomial setting takes the sample size to be some fixed $N$. De-Poissonization is a method by which asymptotic results in the Poissonized setting can be transferred to the fixed size binomial setting. In the context of geometric probability \citet{penrose2001central,penrose2005normal} define what are called stabilizing graph functionals, and provide results on de-Poissonization for this class (see, for example \cite[Lemma 4.2]{penrose2001central} and \cite[Theorem 2.3]{penrose2005normal}). For $k_N$ increasing with sample size, the $k_N$-NN graph is not stabilizing, and hence the results do not apply directly. A more general recipe for de-Poissonization is outlined in \citet[Section 2.5]{penrose2003random}. We believe that this can be used to de-Poissonize in the setting we consider. However, we do not have a proof at this time.
\end{remark}

\subsection{Summary of results}\label{section::summary_of_results}

The present work investigates the power properties of two sample tests based on the $K$-nearest neighbors (KNN) graph as defined in \Cref{example::knn_test}. Our primary focus is on the local asymptotic power of this test in parametric families, when the number of neighbors $K$ is allowed to grow with the sample size $N$. Since the parameter $K$ directly controls the edge density of the underlying graph, its influence on the power is an interesting theoretical question. From an empirical standpoint, the improvement in power obtained by increasing the edge density in graph-based tests has been noted before in the literature. \citet{friedmanrafskytest}, and \citet{chen2017new} consider a denser version of the MST, known as the $K$-MST. In both works, the authors noted an improvement in power with a higher value of $K$. In particular, the authors of \citet{chen2017new} suggested that the power improves because an increased graph density captures greater \emph{similarity information} between samples. More in line with our work, simulations in \citet{bhattacharya2019generalasymptoticframework} showed that allowing $K$ to grow with $N$ leads to improved power for the $K$-NN test.

In our work, we study the local asymptotic power of the $K$-NN test when $K=k_N = o(N^\frac{1}{4}).$ This fills an important gap in the literature by linking the growth of $K$ to the power of the two sample test. Our contributions are two-fold:

\begin{itemize}
\item Previous works on the local asymptotic power of graph-based statistics (see \cite{bhattacharya2020detectionthresholds,huang2024kernel}) have typically relied on the stabilizing graphs framework \cite{penrose2001central,penroseyukich2003weaklaws}. While the $K$-NN graph is stabilizing when $K$ is fixed, it is not stabilizing when $K$ is allowed to grow with the sample size. To go beyond the fixed $K$ case, we rely on a careful, hands-on analysis of the test statistic which allows us to precisely quantify the local dependence structure induced by the $K$-NN graph. A high-level description of these techniques is given in \Cref{section::technical_contributions}.

\item Using the above approach, we establish exact detection thresholds of the 1- and 2-sided tests defined in \eqref{eq::poissonized_one_sided_test} and \eqref{eq::poissonized_two_sided_test}. The detection thresholds depend intimately on the growth of $K$ with respect to $N$ and on the dimension $d$ of the data. We make this dependence explicit and highlight the consequences on the power of the 1- and 2-sided tests. This is described in \Cref{section::description_of_detection_thresholds} with the detailed results stated in \Cref{section::local_power_and_thresholds}.
\end{itemize}

\subsubsection{Technical contributions}\label{section::technical_contributions}

When deriving results for the case of growing $K,$ the main technical hurdle is that the underlying graph functional is no longer stabilizing. Recall that a graph functional $\Gcal$ associates a graph $\Gcal(S)$ given any finite set $S\subset \R^d.$ Given any point $x\in \R^d,$ we denote by $E(x,\Gcal(S))$ the set of edges incident on $x$ in the graph $\Gcal(S\cup {x}).$ A graph functional $\Gcal$ is said to be stabilizing if for any $S$ and $x,$ there exists an $R<\infty$ such that for any finite $\Acal \subset B(x,R)^c,$
$$E(x,\Gcal(S)) = E(x,\Gcal((S\cap B(x,R)) \cup \Acal)).$$

Intuitively, $\Gcal$ is said to be stabilizing if for any point $x$ in the graph, adding points far away enough does not affect the edges incident on $x$. Since edges are local, it is enough to focus on a small neighborhood of a typical point which in turn converges to a homogeneous Poisson process of intensity one after suitable spatial rescaling. Further, stabilizing graph functionals often have natural analogs on infinite sets. For instance, the fixed $K$-nearest neighbor graph can be defined on any locally finite point set, and the minimal spanning tree admits a version on infinite vertices (see \citet{aldous1992asymptotics}). As a result, the large sample limit of statistics calculated on a stabilizing graph functional can often be identified as expectations of the corresponding infinite volume graph functional evaluated on a homogeneous Poisson process. An example of this is \citet[Theorem 2.8]{penroseyukich2003weaklaws} We refer the reader to \citet{penrose2001central,penroseyukich2003weaklaws,penrose2005normal} for full details on stabilization.

In the case of $K$ growing, the value of $K$ depends on the size of the vertex set, and adding enough points to the vertex set eventually changes the neighbors of $x$. As a result, the graph functional is not stabilizing for growing $K$. At a more fundamental level, since the neighborhood of each vertex depends on the size of the vertex set there is no natural extension of this graph functional to infinite point sets. However, the local structure of the graph is still intact since each point $x$ only connects to the $K$ points closest to it. This allows us to restrict our attention to a ball of shrinking radius around $x$ for all the calculations involved in the proofs. When $K$ is fixed, the distance between nearest neighbors is of the order $N^{-\frac{1}{d}}$ regardless of the value of $K$. However, when $K = k_N$ grows with $N,$ the distance is of order $(k_NN^{-1})^{\frac{1}{d}}.$ By incorporating $k_N$ into the spatial scaling, it is still possible to approximate the neighborhood of each point by a homogeneous Poisson process. However, the limiting values have to be calculated explicitly since they cannot be expressed in terms of known graph functionals. The technical tools we use are provided in the supplementary material.

We have three main technical results.

\begin{enumerate}

\item In \Cref{section::consistency} we derive the weak limit of the statistic \eqref{eq::poissonized_test_stat}. In particular, we show that appropriately normalized, this converges in probability to the Henze-Penrose (HP) divergence. This same property was shown in \citet{bhattacharya2020detectionthresholds} for fixed $K$. This is the first of our results that show the continuity between the asymptotics of the statistic for fixed and growing $K$.

\item We derive the limiting distribution of the Poissonized test statistic under general alternatives and establish a central limit theorem for the appropriately standardized statistic (Theorem~\ref{thm::clt_general_stat}). We show that the asymptotic variance is of order $N k_N^2$ and that asymptotic normality holds for $k_N = o(N^{1/4})$. The proof uses Stein’s method for dependency graphs applied to a truncated version of the statistic, in the spirit of techniques developed for stabilizing graph functionals in \citet{penrose2005normal}. Since the stabilizing graph framework is not directly applicable in our setting, the argument relies on more hands-on bounds that quantify the local dependence structure of the nearest neighbor graph. These bounds allow us to establish convergence of the standardized statistic to normality in Wasserstein distance and suggest that stabilizing-graph methods may extend to a broader class of graph functionals.

\item We also provide a CLT for a conditional version of the statistic which relates to the implementation of the two sample test as a permutation test. Under the null, after conditioning on the locations of the points, the labels are exchangeable. Thus, instead of centering by the marginal expectation, we can center by the conditional expectation after conditioning on the locations of the points. The resulting statistic can be used to implement a conditional test. This statistic too is asymptotically normal after scaling by $k_N N^\frac{1}{2}$ (Theorem \ref{thm::clt_conditional_stat}). Using Stein's method for dependency graphs, one can show that this holds for any sequence $k_N\to \infty$ such that $k_N = o(N).$
\end{enumerate}

As an aside, we should note that the result for the asymptotic normality of the unconditional statistic holds for a smaller range of $k_N$ than that for the conditional statistic. We elaborate on this further in the discussion following \Cref{thm::clt_general_stat}.

\subsubsection{Power under local alternatives}\label{section::description_of_detection_thresholds}

Local asymptotic power is most naturally formulated by considering alternatives that approach the null as the sample size increases. In parametric families, where distributions can be represented by low-dimensional parameters, this corresponds to allowing the parameter under the alternative to converge to the null parameter at a certain rate with respect to the sample size. 

More rigorously, let $\{P_\theta\}_{\theta\in \Theta}$ be a family of distributions parametrized by elements of $\Theta\subset \R^p.$ Fix $\theta_1\in\Theta$ to be the null parameter and let $\{\theta_N\}_N$ be a sequence in $\Theta$ representing the local alternatives. The quantity of interest is the \emph{detection threshold}---the maximal rate at which $\theta_N$ may approach $\theta_1$ while the test continues to exhibit nontrivial asymptotic power. More precisely, the detection threshold of a two sample test is the sequence $\{\epsilon_N\}_N$ such that $\|\theta_N-\theta_1\|\gg \epsilon_N$ implies that the limiting power is $1$ and $\|\theta_N - \theta_1\|\ll \epsilon_N$ implies that the limiting power of the test is its level, $\alpha.$  

\begin{figure}
    \centering
    \captionsetup{width = 0.9\linewidth}
    \includegraphics[width=0.7\linewidth]{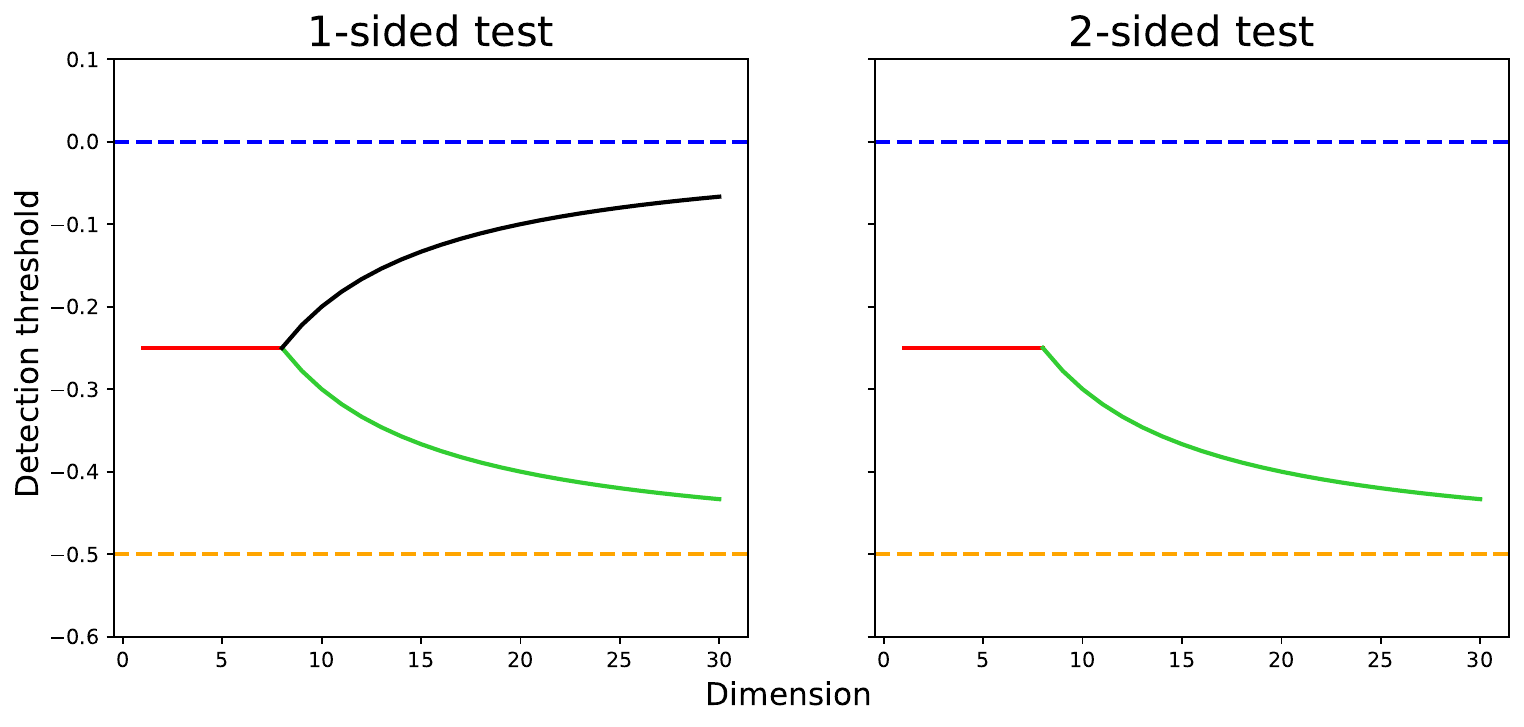}
    \caption{A comparison of the detection thresholds exponents of the 1- and 2-sided tests for a given $k_N = N^\gamma$. A lower value of the exponent represents a better detection threshold. The orange line represents the parametric rate and the blue line represents the case of fixed alternatives. In both cases, change in detection threshold occurs at the critical dimension $d=d_c(\gamma)$.}
    \label{fig::threshold_comparison}
\end{figure}

The graph-based two sample tests studied in the literature formulate the test as a 1-sided test. A key contribution of our work is to study the 2-sided version in \eqref{eq::rejection_region_for_2_sided_test} and carry out a comparison of the two, in the setting where $K = k_N = o(N^\frac{1}{4})$. To simplify matters, we will assume that $k_N = N^{\gamma}$ for some $0<\gamma<1/4.$
\begin{itemize}
\item We can define a critical dimension $d_c$ given by
\begin{equation}\label{eq::simplified_critical_dimension}
d_c(\gamma)= \lfloor 8(1-\gamma)\rfloor.    
\end{equation}
When $d\leq d_c(\gamma)$, the detection thresholds of the 1- and 2-sided tests are the same and equal $N^{-\frac{1}{4}}.$ This means, if $\|\theta_N - \theta_1\|\ll N^{-\frac{1}{4}},$ the limiting power of both tests is $\alpha$ and if $\|\theta_N - \theta_1\|\gg N^{-\frac{1}{4}},$ then the limiting power is 1. When $\theta_N - \theta_1 = hN^{-\frac{1}{4}}$ for some $h\in \R^p,$ we give an expression for the limiting power in terms of $h$ and the parametric family $\{P_\theta\}_\theta$.

\item When $d>d_c(\gamma),$ the situation is more delicate and the behavior of the two tests is drastically different. The limiting power of the 1-sided test can be described as follows.
\begin{itemize}
\item If $\|\theta_N - \theta_1\| \ll N^{-\frac{1}{2} + \frac{2(1-\gamma)}{d}},$ then the limiting power is $\alpha$.
\item If $\|\theta_N - \theta_1\| \gg N^{-\frac{2(1-\gamma)}{d}},$ then the limiting power is 1.
\item If $\theta_N - \theta_1 = h \epsilon_N$ for some unit vector $h$ and $N^{-\frac{1}{2} + \frac{2(1-\gamma)}{d}}\ll \epsilon_N \ll N^{-\frac{2(1-\gamma)}{d}},$ then the limiting power is 0 or 1 depending on $h.$
\end{itemize}

For $d>d_c(\gamma),$ there is an \emph{exponent gap} in the detection thresholds. Let $h$ be the unit vector in the direction of $\theta_N - \theta_1$. Depending on $h,$ the detection threshold is either $N^{-\frac{1}{2} + \frac{2(1-\gamma)}{d}}$ or $N^{-\frac{2(1-\gamma)}{d}}$. In the former case, the detection threshold represents a blessing of dimensionality since it improves with the dimension. On the other hand, the second case displays a curse of dimensionality since the detection threshold worsens with dimensions. The first frame of \Cref{fig::threshold_comparison} makes this clearer. The black and green curves represent the two rates. Depending on the direction $h,$ the detection threshold lies on either one of the curves. We also see the exponent gap widening with dimension.

\item The 2-sided test on the other hand demonstrates much more stable behavior with respect to the direction when $d>d_c(\gamma).$ The detection threshold is $N^{-\frac{1}{2} + \frac{2(1-\gamma)}{d}}$. The 2-sided test has limiting power $\alpha$ or 1 depending on whether $\theta_N - \theta_1$ converges to 0 faster or slower than $N^{-\frac{1}{2} + \frac{2(1-\gamma)}{d}}$, regardless of the direction in which $\theta_N$ approaches $\theta_1$. In particular, it does not demonstrate the exponent gap that plagues the 1-sided test. This is shown in the second frame of \Cref{fig::threshold_comparison}. The detection threshold exponent for the 2-sided test is given by a single curve as opposed to two curves for the 1-sided test. We also see the detection threshold improving with dimension after the critical point $d_c(\gamma)$, and it approaches the parametric rate of $N^{-\frac{1}{2}}.$
\end{itemize}
%

\subsection{Organization of the paper}

The rest of the paper is organized as follows. In Section \ref{section::consistency}, we derive the weak limit of the test statistic and use it to show that the 1- and 2-sided tests are consistent. In Section \ref{section::general_alternatives_distribution}, we derive the limiting distribution of the test statistic under general alternatives. Section \ref{section::local_power_and_thresholds} gives detailed results on the detection thresholds of the 1- and 2-sided tests as well as the limiting power at these thresholds. Section \ref{section::simulations} is dedicated to simulations on synthetic data used to validate the theoretical results, and a real data experiment. This section also contains a runtime analysis of the test.
\section{Consistency}\label{section::consistency}

In this section we show that the 1-sided and 2-sided tests are both consistent. For this, we require the notion of the \textit{Henze-Penrose dissimilarity measure} between two densities.

\begin{definition}\label{defn::HP_divergence}
Let $f,g$ be two densities on $\R^d.$ Let $p\in (0,1)$ and $q:=1-p.$ Then, the \textnormal{Henze-Penrose (HP) dissimilarity measure} between $f,g$ is defined as
$$\delta(f,g,p) = pq\int\frac{f(x)g(x)}{pf(x)+qg(x)}~dx.$$
\end{definition}

\noindent This belongs to the larger class of $f-$dissimilarities as defined in \citet{gyorfi1978fdissimilarity}. The following proposition shows that the HP dissimilarity is the limiting value of the statistic $T(\Zcal_N).$

\begin{proposition}\label{propn::weak_limit_of_stat} Let $f,g$ be two densities on $\R^d.$ Let $\{k_N\}_{N\geq 1}$ be a sequence of natural numbers such that $k_N = o(N)$. Then,
$$\frac{1}{Nk_N}T(\Gscr_{k_N}(\Zcal_N))\xrightarrow{p} \delta(f,g,p).$$
\end{proposition}

Under the null $f=g,$ $\delta(f,f,p) = pq.$ Furthermore, from \citet[Theorem 1 and Corollary 1]{gyorfi1978fdissimilarity} we have that $\delta(f,g,p)\leq pq$ with equality holding if and only if $f,g$ are equal everywhere except on a set of measure $0.$ These facts combined with the asymptotic normality of $T(\Gscr_{k_N})$ under the null (\Cref{thm::clt_general_stat} in the next section) show that the 1- and 2-sided tests are consistent. The convergence in \Cref{propn::weak_limit_of_stat} hinges on the fact that for $k_N\to \infty,$ the distance between a point and its $k_N$ nearest neighbors converges to 0 at the rate $(k_N N^{-1})^\frac{1}{d}$. As an estimator of the HP dissimilarity, the statistic $T(\Gscr_{k_N})$ is akin to nearest neighbor estimators of densities. Our proof also suggests that it should be possible to extend this convergence for other norms in $\R^d.$

The HP dissimilarity has been repeatedly obtained as the limiting value of graph-based statistics. Proposition 2.1 of \citet{bhattacharya2020detectionthresholds} shows that when $K$ is fixed,
$$\frac{1}{N}T(\Gscr_K(\Zcal_N)) \xrightarrow{p} K\cdot \delta(f,g,p),$$

In fact, the same result proves the above convergence to the HP dissimilarity holds for any stabilizing graph functional which also satisfies some additional degree conditions. Earlier, \citet{henze1999multivariate} proved the same result for the MST. In \citet{ariascastro2015consistency} the authors showed that the HP dissimilarity is also the weak limit for the statistic based on the minimum bipartite matching (MBM) graph, as defined in \citet{rosenbaum2005exact2sampletest}. It is unknown whether the MBM is stabilizing, which shows that stabilization is not a prerequisite.

These results accompanied by \Cref{propn::weak_limit_of_stat} suggest a general extension to \emph{dense} versions of a given base graph. Note that the $K$-NN graph can be constructed by iteratively constructing the $1$-NN graph. Initially, all edges are available to us, and we construct the usual $1$-NN graph. On the second iteration, we construct the $1$-NN graph restricted to those edges not occurring in the first iteration. We can repeat this process $K$ times, each time restricting the set of available edges to those that have not been selected in any previous iteration. Taking the union of these $K$ graphs gives the $K$-NN graph. This procedure can be used to produce a dense version of any given base graph. One can now ask a more general question---does a result similar to \Cref{propn::weak_limit_of_stat} hold if we consider the number of copies $K$ to be fixed or increasing with $N$, for a different base graph? By tailoring the approach for each specific base graph, it should be possible to show the HP dissimilarity as the limiting object in other cases. To handle a large family of graphs in one fell swoop, one can attempt to extend the definition of stabilization itself to incorporate the notion of varying densities.

The proof of Proposition \ref{propn::weak_limit_of_stat} is given in \appendixB.

\section{Distribution under general alternatives}\label{section::general_alternatives_distribution}

Recall the Poissonized set up from Section \ref{section::poissonization}. Let $f,g$ be densities on $\R^d.$ Define
\begin{align*}
    \phi_N(x) &= \frac{N_1}{N}f(x) + \frac{N_2}{N}g(x),\\
    \phi(x) &= pf(x) + qg(x).
\end{align*}
$\Zcal_N = \{Z_1,\dots,Z_{L_N}\}$ denotes the set of points sampled from a Poisson process on $\R^d$ with intensity function $N\phi_N(x) = N_1f(x) + N_2g(x)$ where $N_1+N_2 = N.$ Since $f,g$ are densities, we have that $L_N \sim \poi(N).$ To each point $z\in \Zcal_N,$ we assign the label $1$ or $2$ with probabilities proportional to $N_1f(z), N_2g(z).$

The normalized test statistic is
$$\Rcal(\Gscr_{k_N}(\Zcal_N)) = \frac{1}{k_N\sqrt{N}}(T(\Gscr_{k_N}(\Zcal_N)) - \E_{H_1}(T(\Gscr_{k_N}(\Zcal_N)))).$$

In this section, we will derive the asymptotic distribution of $\Rcal(\Gscr_{k_N}(\Zcal_N))$ as $N\to \infty$ when $k_N = o(N^{1/4}).$ To start, we prove the asymptotic normality of  $\Rcal(\Gscr_{k_N})$ after conditioning on an appropriate sigma algebra.

\subsection{CLT for a conditional statistic}\label{subsection::clt_conditional_stat}

To motivate the conditional CLT, it is useful to consider formulation of graph-based two sample tests as permutation tests. In the un-Poissonized setting, the data consists of $X_1,..,X_{N_1}\overset{\textnormal{i.i.d}}{\sim} F$ and $Y_1,\dots,Y_{N_2}\overset{\textnormal{i.i.d}}{\sim} G$ for two distribution $F,G$. The test statistic is then defined as
$$T(\Gscr_K) = \sum_{i=1}^{N_1}\sum_{j=1}^{N_2}\indicator\{(X_i,Y_j)\in \Gscr_K\},$$

where $\Gscr_K$ is the directed $K$-NN graph assembled from $\{X_1,..,X_{N_1},Y_1,\dots,Y_{N_2}\}.$ Under the null $F=G,$ an alternate way to sample the data is to first consider $N=N_1+N_2$ points, $Z_1,\dots,Z_N \overset{\textnormal{i.i.d}}{\sim} F.$ The labels of the points are $\{c_1,\dots,c_N\}$ such that $c_i=1$ for $i\leq N_1$ and $c_i=2$ otherwise. The test statistic is now defined as
$$T(\Gscr_K) = \sum_{i=1}^N\sum_{i=1}^N \indicator\{(Z_i,Z_j)\in \Gscr_K, c_i=1,c_j =2\}.$$

Sampling $B$ i.i.d uniform permutations $\{\tau_i\}_{i=1}^B$, we can define 
$$T^{(b)}(\Gscr_K) = \sum_{i=1}^N\sum_{j=1}^N \indicator\{(Z_i,Z_j)\in \Gscr_K, c_{\tau_b(i)}=1, c_{\tau_b(j)}=2\}.$$

Under the null, $T(\Gscr_K),T^{(1)}(\Gscr_K),\dots,T^{(B)}(\Gscr_K)$ are exchangeable due to the independence of locations and labels, which yields the permutation test. Thus, under the null the permutation test is obtained by conditioning on the locations of the points and resampling the labels.

Analogously, in the Poissonized setting we instead condition on $\Fcal_N = \sigma(\{Z_1,\dots,Z_{L_N}\})$, the sigma algebra containing the information on the locations of the points. The randomness now comes only from the labels, and we consider the conditional statistic
$$\Rcal_\cond(\Gscr_{k_N}(\Zcal_N)) = \frac{1}{\sqrt{N}k_N}(T(\Gscr_{k_N}(\Zcal_N)) - \E_{H_1}(T(\Gscr_{k_N}(\Zcal_N)| \Fcal_N)).$$

We now show the asymptotic normality of this statistic under general alternatives.

\begin{theorem}\label{thm::clt_conditional_stat}
Let $f,g$ be densities on $\R^d.$ Let $k_N = o(N).$ Define $\Fcal_N = \sigma(\Zcal_N,L_N),$ the sigma algebra generated by the points in the Poisson process. Then we have that
\begin{equation}\label{eqn::clt_for_conditional_stat}
\frac{\Rcal_\cond(\Gscr_{k_N}(\Zcal_N))}{\sqrt{\Var(\Rcal_\cond(\Gscr_{k_N}(\Zcal_N))\mid \Fcal_N)}} \xrightarrow{d} N(0,1).
\end{equation}

Furthermore,
\begin{equation}\label{eqn::asymptotic_conditional_variance}
    \frac{1}{Nk_N^2}\Var(\Rcal_\cond(\Gscr_{k_N}(\Zcal_N))|\Fcal_N)\xrightarrow{p} \sigma^2_\cond,
\end{equation}
where
$$\sigma^2_\cond = pq\int \frac{f(x)g(x)(pf(x)-qg(x))^2}{\phi^3(x)}~dx.$$
\end{theorem}

Under the null, the conditional variance in \eqref{eqn::clt_for_conditional_stat} is easily obtained from the adjacency structure of the graph $\Gcal_{k_N}.$ Indeed, after conditioning on $\Zcal_N,$ the statistic $T(\Gscr_{k_N})$ is a sum of dependent Bernoulli random variables, with the dependence structure encoded by $\Gcal_{k_N}.$ We elaborate on the exact conditional variance calculation in \Cref{subsection::runtime_analysis}. The above result gives a pivotal statistic that can be used to carry out the test when the sample size makes the permutation test computationally prohibitive. The proof used Stein's method for dependency graphs and is given in \appendixB.

\subsection{CLT for the test statistic}\label{subsection::clt_test_stat}

We now show the asymptotic normality in the unconditional case.

\begin{theorem}\label{thm::clt_general_stat}
Let $f,g$ be densities on $\R^d$ and let $k_N \to \infty$ with $k_N = o(N).$ Then,
$$\frac{1}{Nk_N^2}\Var(T(\Gscr_{k_N}(\Zcal_N))) \to \sigma^2$$
where
\begin{equation}\label{eq::general_unconditional_variance}
\sigma^2 = pq\bigintsss \frac{f(x)g(x)(pf(x)-qg(x))^2}{\phi(x)^3}~dx + p^2q^2\bigintsss \frac{f(x)^2g(x)^2}{\phi(x)^3}~dx.
\end{equation}

Furthermore, when $k_N = o(N^{1/4}),$
$$\Rcal(\Gscr_{k_N}(\Zcal_N)) \xrightarrow{d} N(0,\sigma^2)$$
\end{theorem}

\Cref{thm::clt_general_stat} relies on the fact that when $k_N = o(N),$ the $k_N$-nearest neighbors of a point all lie in a ball of shrinking radius around it. In particular, the probability of two points lying in well-separated regions being nearest neighbors is negligible. To use this in the proof, we divide the region into a grid of boxes shrinking at an appropriate rate. By the argument sketched above, we can then restrict ourselves to considering nearest neighbors within neighboring boxes. This allows us to prove asymptotic normality using Stein's method for a truncated version of $\Rcal(\Gscr(\Zcal_N))$. By using Slutsky's theorem, we recover the asymptotic normality result for the original statistic. Theorem \ref{thm::clt_general_stat} is proved in \appendixB.

The asymptotic normality in \Cref{thm::clt_conditional_stat} and \Cref{thm::clt_general_stat} is proved using Stein's method for dependency graphs. However, the range of $k_N$ considered is different for the two---respectively, $k_N = o(N)$ and $k_N = o(N^\frac{1}{4})$. We now give a heuristic of why this happens. The Stein's method approach requires bounding the Wasserstein distance to the standard normal. The Wasserstein bound involves powers of $k_N$ in both cases. These are unavoidable, since the dependence increases with $k_N$, which makes the normal approximation worse with increasing growth of $k_N$. In \Cref{thm::clt_conditional_stat}, after conditioning on $\Fcal_N,$ the randomness comes only from the labels of the points. In this case, the test statistic is a sum of Bernoulli random variables indexed by the edges of the graph. Only the random variables corresponding to the edges that share an end point can be correlated. Hence, the dependency graph in Stein's method is closely related to the $k_N$-NN graph obtained from the points in $\Zcal_N.$ This additional structure allows us to quantify the local dependence very well.

In contrast, the marginal CLT in \Cref{thm::clt_general_stat} has much greater dependence between points. This arises because, without any conditioning, potentially any two points could be nearest neighbors of each other. The truncation argument we use does restrict this dependence to nearby points, but the Wasserstein bound still contains additional powers of $k_N$ compared to the conditional case. As a result, we only obtain asymptotic normality for $k_N = o(N^\frac{1}{4})$. It is likely that there is a different argument which holds for the full range of $k_N = o(N)$. Having pointed this out, we should note that \Cref{thm::clt_conditional_stat} gives a valid way to test the null for $k_N = o(N)$. Thus, the restriction on the growth of $k_N$ does not pose a practical threat to the test. We elaborate on this further in \Cref{section::simulations}.
\section{Local power of the two sample test}\label{section::local_power_and_thresholds}

We now come to the local power of the $K$-NN test in a parametric family. We first state the assumptions we make on the family of distributions we consider. The assumptions we make are the same as the ones made by \citet{bhattacharya2020detectionthresholds} in proving the results for fixed $K.$

For a function $g(z_1,z_2):\R^d\times \R^m \to \R,$ we denote the gradient and Hessian with respect to $z_1$ for a fixed $z_2$ by $\nabla_{z_1}g(z_1,z_2)$ and $\Hnormal_{z_1}g(z_1,z_2)$ similarly denote by $\nabla_{z_2}g(z_1,z_2)$ and $\Hnormal_{z_2}g(z_1,z_2)$ the gradient and the Hessian of $g$ with respect to $z_2$ for a fixed $z_1.$ With the notation defined, we now state our assumptions.

\begin{assumption}\label{assume::assumptions_on_parametrized_family}
Let $\{P_\theta\}_{\theta\in \Theta}$ be a family of distributions parametrized by the elements of a convex set $\Theta\subset \R^m.$ We will assume the following properties for $\{P_\theta\}_\theta:$
\begin{enumerate}
    \item For all $\theta\in \Theta,$ the density $p(\cdot|\theta)$ has a common support $S$ such that $S$ is compact, convex with non-empty interior and $p(\cdot | \theta)$ are uniformly bounded above and below on $S.$
    \item The support $S$ satisfies $S = \overline{\textnormal{int}(S)}$ and $\partial S$ has Lebesgue measure zero.
    \item For all $\theta\in \Theta$ the functions $p(\cdot|\theta)$ and $\nabla_\theta p(\cdot|\theta)$ are three times continuously differentiable over $S$.
    \item $\E\left[\frac{h^T\nabla_{\theta_1}p(X|\theta)}{p(X|\theta)}\right]^2$ is finite and positive for all $\theta\in \Theta$ and $h\in \R^m,~ h\neq 0.$
    \item The function $p(x|\cdot)$ is three times continuously differentiable in $\Theta$ for all $x\in S.$
\end{enumerate}
\end{assumption}

We assume that all densities have the same compact support to simplify the proofs. However, we fully expect that one can circumvent it assuming that the tails of the distributions decay fast enough. Indeed, our simulations involve the spherical normal family which does not satisfy the compact support assumption. However, the behavior predicted by our results is observed, indicating that our results do hold in the more general setting as well. The assumptions on the differentiability and smoothness are required to analyze the difference in the null and alternate means under local alternative cases. Under looser assumptions on the differentiability, our methods could give upper and lower bounds on the detection thresholds, but it will be difficult to obtain exact power expressions.

To find the local power of the $K$-NN test, we need the asymptotic distribution of the statistic when the two densities $f,g$ are given by $f=p(\cdot|\theta_1), g = p(\cdot|\theta_N)$ where $\theta_1\in \Theta$ is fixed and $\theta_N\to \theta_1$. The proof of \Cref{thm::clt_general_stat}, which gives the asymptotic distribution under general alternatives, can be adapted to give the required result. This is summarized in the following lemma.

\begin{lemma}\label{lemma::local_alternative_clt} Let $f=p(\cdot|\theta_1), g=p(\cdot|\theta_N)$ such that $\theta_N\to \theta_1$ as $N\to \infty.$ Let $k_N = o(N^{1/4}).$ Then
$$\Rcal(\Gscr_{k_N}(\Zcal_N)) \xrightarrow{d} N(0,\sigma^2_0)$$
where 
\begin{equation}\label{eq::null_variance}
\sigma^2_0 = pq((p-q)^2+pq).
\end{equation}
\end{lemma}

The null variance in \eqref{eq::null_variance} is obtained by considering the general unconditional variance in \eqref{eq::general_unconditional_variance} for $f=g.$ Under the null, the expected value of the statistic is $Nk_N\frac{N_1N_2}{N^2}.$ Therefore, the 1-sided test rejects the null hypothesis when 
\begin{equation}\label{eq::rejection_region_for_1_sided_test}
    \frac{1}{k_N\sqrt{N}\sigma_0}\left(T(\Gscr_{k_N}(\Zcal_N)) - Nk_N\frac{N_1N_2}{N^2}\right) < z_\alpha
\end{equation}
and the 2-sided test rejects when
\begin{equation}\label{eq::rejection_region_for_2_sided_test}
    \frac{1}{k_N\sqrt{N}\sigma_0}\left|T(\Gscr_{k_N}(\Zcal_N)) - Nk_N\frac{N_1N_2}{N^2}\right| > z_{1-\alpha/2}
\end{equation}
where $z_{\alpha}, z_{1-\alpha/2}$ are the $\alpha$ and $1-\alpha/2$ quantiles, respectively, of the standard normal.

To state the theorems, we need some notation. For $\theta_1\in \Theta$ and $h\in \R^m,$ we define
\begin{align}
    a(\theta_1,h) &:= \frac{r^2}{2\sigma_0}\E\left[\frac{h^T\nabla_{\theta_1}p(X|\theta_1)}{p(X|\theta_1)}\right]^2 \label{eq::Hessian_term}\\
    b(\theta_1,h) &:= \frac{p^2q}{2(d+2)V_d^{\frac{2}{d}}\sigma_0} \bigintssss_S h^T\nabla_{\theta_1}\left(\frac{\text{tr}(\text{H}_xp(x|\theta_1))}{p(x|\theta_1)}\right)p^{\frac{d-2}{d}}(x|\theta_1) ~ dx \label{eq::Gradient_term},
\end{align}

\noindent where $\sigma^2_0$ is the null variance in \eqref{eq::null_variance}, $r=2pq$, $V_d$ denotes the volume of the $d$ dimensional unit ball and $\text{H}_x$ (as stated previously) denotes the Hessian with respect to $x$ at $x.$

In our results, we assume $k_N = N^{\gamma}$ for some $0<\gamma<1/4$.  Recall that the critical dimension $d_c(\gamma)$ is defined as
\begin{equation}\label{eq::critical_dimension_recalled}
d_c(\gamma) = \lfloor 8(1-\gamma)\rfloor.    
\end{equation}

The limiting power of both tests is different depending on whether $d\leq d_c(\gamma)$ or $d>d_c(\gamma)$. We will refer to these cases as \emph{below criticality} and \emph{above criticality}, respectively. We should point out that our methods work for any $k_N\to \infty$ satisfying $k_N = o(N^{\frac{1}{4}}).$ However, the definition of the critical dimension can be tedious. To better communicate the results, we stick to the case of $k_N = N^\gamma.$ A description of the critical dimension for a general growing sequence $\{k_N\}$ is provided in \Cref{subsection::proof_sketch}.

\subsection{Local power below criticality}

We now state the detection thresholds for $d\leq d_c(\gamma).$

\begin{theorem}\label{thm::power_below_criticality}

Let $\{P_\theta\}_{\theta\in \Theta}$ be a parametrized family satisfying \Cref{assume::assumptions_on_parametrized_family}. Let $\Zcal_N$ be the samples from the Poisson process with $f=p(\cdot|\theta_1),~g=p(\cdot|\theta_N)$ with labels assigned as in \eqref{eq::label_assignment}. Assume $k_N = N^{\gamma}$ with $0<\gamma<1/4$, and let $\theta_N \to \theta_1$ with $\epsilon_N := \theta_N - \theta_1.$ Finally, assume that $d\leq d_c(\gamma).$

Consider the 1- and 2-sided tests based on the $k_N$-NN graph with rejection regions as defined in \eqref{eq::rejection_region_for_1_sided_test} and \eqref{eq::rejection_region_for_2_sided_test}, respectively. The limiting power is given as follows.
\begin{enumerate}
\item If $\|\epsilon_N\| \ll N^{-\frac{1}{4}}$ then the limiting power of both tests is $\alpha.$
\item If $\|\epsilon_N\|\gg N^{-\frac{1}{4}}$ then the limiting power of both tests is 1.
\item If $\epsilon_N = hN^{-\frac{1}{4}}$ then:
\begin{itemize}
\item The limiting power of the 1-sided test is $\Phi(z_\alpha+a(h,\theta_1)).$
\item The limiting power of the 2-sided test is $\Phi(z_{\alpha/2}-a(h,\theta_1))+\Phi(z_{\alpha/2} + a(h,\theta_1))$ 
\end{itemize}
where $\Phi$ is the standard normal distribution function.
\end{enumerate}

\end{theorem}

The behavior of 1- and 2-sided tests in the case $d\leq d_c(\gamma)$ is fairly simple to describe and aligns with the results in many other settings. There is a single, sharp decay rate of $\epsilon_N := \theta_N - \theta_1$, in this case $N^{-\frac{1}{4}}$, which serves as the boundary between the regimes where the tests are powerless and powerful. If $\epsilon_N$ decays faster than $N^{-\frac{1}{4}},$ the tests are powerless. When it decays faster than $N^{-\frac{1}{4}},$ the tests are powerful. When the rate of decay is exactly $N^{-\frac{1}{4}},$ then our results describe the limiting power of two tests in terms of standard normal CDF. \Cref{thm::power_below_criticality} is proved in \appendixC.

\subsection{Local power above criticality}\label{sec::local_power_above_criticality}

We now describe the detection thresholds for $d>d_c(\gamma).$ In this regime, the limiting power is more subtle and to describe it, we will need some Terminology. Recall that we assume $k_N = N^{\gamma}$ for some $0<\gamma<1/4.$ When $d>d_c(\gamma),$ we have
\begin{equation}\label{eq::comparing_lower_and_upper_thresholds}
N^{-\frac{1}{2} + \frac{2(1-\gamma)}{d}}\ll N^{-\frac{1}{4}} \ll N^{-\frac{2(1-\gamma)}{d}}.
\end{equation}

We will refer to the rates $N^{-\frac{1}{2} + \frac{2(1-\gamma)}{d}}$ and $N^{-\frac{2(1-\gamma)}{d}},$ called the \emph{lower threshold} and \emph{upper threshold}, respectively. There are now three broad cases for the rate of decay of $\epsilon_N := \theta_N - \theta_1$, given as follows:
\begin{enumerate}
\item \textbf{Power below the lower threshold :} $\|\epsilon_N\| \ll N^{-\frac{1}{2}+\frac{2(1-\gamma)}{d}}.$
\item \textbf{Power in the exponent gap:} $N^{-\frac{1}{2}+\frac{2(1-\gamma)}{d}} \ll \|\epsilon_N\| \ll N^{-\frac{2(1-\gamma)}{d}}.$
\item \textbf{Power above the upper threshold:} $N^{-\frac{2(1-\gamma)}{d}} \ll \|\epsilon_N\|.$
\end{enumerate}

The next result gives the limiting power of the two tests in the above three regimes.

\begin{theorem}\label{thm::both_tests_above_criticality_broad_regimes}

Consider the 1- and 2-sided tests with rejection regions as defined in \eqref{eq::rejection_region_for_1_sided_test} and \eqref{eq::rejection_region_for_2_sided_test}, respectively. We assume the same conditions as \Cref{thm::power_below_criticality}, except now $d>d_c(\gamma)$ where $k_N = N^\gamma$. The limiting power is given as follows:

\begin{enumerate}
\item \textbf{Power below the lower threshold}: If $\|\epsilon_N\|\ll N^{-\frac{1}{2} + \frac{2(1-\gamma)}{d}},$ the limiting power of both tests is $\alpha.$
\item \textbf{Power in the exponent gap}: If $N^{-\frac{1}{2} + \frac{2(1-\gamma)}{d}} \ll \|\epsilon_N\|\ll N^{-\frac{2(1-\gamma)}{d}}$ and if $\frac{\epsilon_N}{\|\epsilon_N\|} = h,$ then: 
    \begin{itemize}
        \item The limiting power of the 1-sided test is $0$ or $1$ according to $b(h,\theta_1)$ being positive or negative, respectively.
        \item The limiting power of the 2-sided test is 1.
    \end{itemize}
\item \textbf{Power above the upper threshold}: If $N^{-\frac{2(1-\gamma)}{d}} \ll \|\epsilon_N\|$, then the limiting power of both tests is $1.$
\end{enumerate}
\end{theorem}

\begin{figure}[t]
    \centering
    \captionsetup{width = 0.9\linewidth}
    \includegraphics[width=0.7\linewidth]{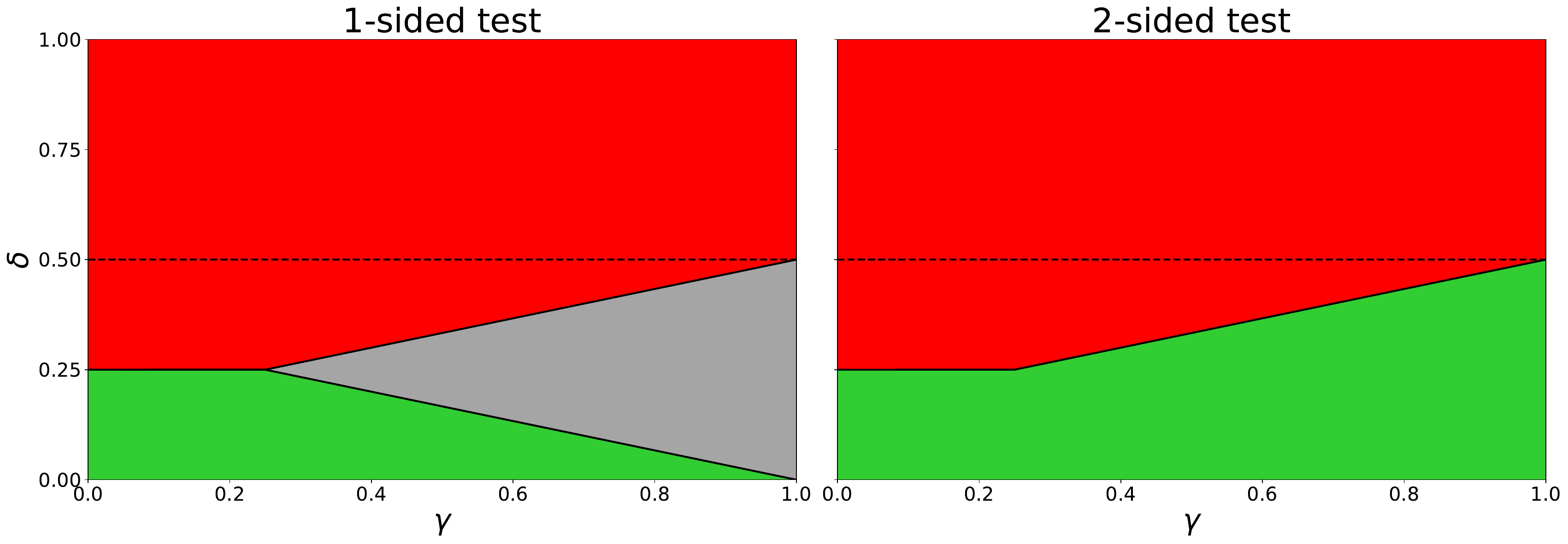}
    \caption{Limiting power comparison of the two tests for some $d\leq 8$. The X and Y-axes, respectively, represent $\gamma$ and $\delta$. Red denotes limiting power $\alpha$, green denotes limiting power $1$ and gray denotes limiting power 0 or 1. The slopes of the two lines change at $\gamma = \gamma_c(d).$ The dotted horizontal line represents the parametric rate $\delta = \frac{1}{2}$.}
    \label{fig::power_comparison_below_criticality}
\end{figure}

\Cref{thm::both_tests_above_criticality_broad_regimes} is proved in \appendixC. \Cref{thm::power_below_criticality,thm::both_tests_above_criticality_broad_regimes} are best illustrated when we take the number of neighbors to be $k_N = N^{\gamma}$ for $\gamma<1/4$ and the null and alternate parameter distance to be $\|\epsilon\| = N^{-\delta}$ for $\delta>0$. For a given value of $\gamma$ and ambient dimension $d,$ the two results then allow us to find the ranges of $\delta$ where the two tests are powerful and powerless. This provides an assessment of how sensitive the two tests are to differences in distributions.
\begin{itemize}
\item \textbf{Effect of $\gamma$:} For this part, we will assume that the ambient dimension $d$ is fixed. In \eqref{eq::critical_dimension_recalled}, we used $\gamma$ to define a critical dimension $d_c(\gamma)$ which serves as a phase transition point in the detection thresholds. Inverting this definition, we can instead define a critical growth rate of $k_N$ given by
\begin{equation}\label{eq::critical_gamma_definition}
\gamma_c(d) = 1 - \frac{d}{8}.
\end{equation}
From \Cref{thm::power_below_criticality}, if $k_N = N^\gamma$ with $\gamma\leq \gamma_c(d),$ then both tests have limiting power is $\alpha$ if $\delta>1/4$ and limiting power 1 if $\delta<1/4.$ When $\delta = 1/4,$ the limiting power is given in terms of the normal CDF according to the third point in \Cref{thm::power_below_criticality}. In this case, the detection threshold of both tests is $\delta = \frac{1}{4}$ regardless of the growth of $k_N.$

When $\gamma>\gamma_c(d),$ \Cref{thm::both_tests_above_criticality_broad_regimes} applies. The 2-sided test has limiting power $\alpha$ if $\delta>\frac{1}{2}-\frac{2(1-\gamma)}{d}$ and limiting power $1$ otherwise. The detection threshold of the 2-sided test in this case is $\delta = \frac{1}{2} - \frac{2(1-\gamma)}{d}$. Crucially, we see that as $\gamma\to 1,$ 
$$\frac{1}{2} - \frac{2(1-\gamma)}{d} \uparrow \frac{1}{2}.$$

Thus, the detection threshold of the 2-sided test actually improves and approaches the parametric rate of $\delta = \frac{1}{2}$ as we increase the growth of $k_N.$ The change in the detection threshold of the 2-sided test is summarized by the right-hand frame of \Cref{fig::power_comparison_below_criticality}.

The 1-sided test has a more subtle behavior when $\gamma>\gamma_c(d).$ The limiting power is $\alpha$ if $\delta> \frac{1}{2} - \frac{2(1-\gamma)}{d}.$ The limiting power is 1 if $\delta<\frac{2(1-\gamma)}{d}.$ When $\delta \in \left(\frac{2(1-\gamma)}{d},\frac{1}{2} - \frac{2(1-\gamma)}{d}\right)$, i.e. in the exponent gap, the limiting power is 0 or 1 depending on the direction in which $\theta_N$ approaches $\theta_1$. The effect of $\gamma$ on the detection threshold of the 1-sided test is shown in the left-hand frame of \Cref{fig::power_comparison_below_criticality}.

\item \textbf{Effect of dimension $d$:} From \eqref{eq::critical_gamma_definition}, we see that $\gamma_c(d)<0$ for $d\geq 9.$ Thus, when $d\geq 9,$ only \Cref{thm::both_tests_above_criticality_broad_regimes} applies regardless of the value of $\gamma$ and the detection thresholds of both tests are always different. In this case, a more representative plot of the limiting power is given in \Cref{fig::power_comparison_above_criticality}. Particularly, the exponent gap for the 1-sided test is present for all values of $\gamma$ as opposed to for $\gamma>\gamma_c(d)$ for $d\leq 8.$

On a similar note, \eqref{eq::critical_dimension_recalled} shows that $d_c(\gamma)\leq 8$ for any value of $\gamma.$ For $0<\gamma<1/4,$ the range covered by our results, the critical dimension can only be $d=7,8.$ Thus, for $\gamma<1/4$ and $d\leq 6,$ \Cref{thm::power_below_criticality} shows that the detection threshold of both tests is $\delta=1/4.$ 

By considering the definition of the $d_c(\gamma)$, we can trace the effect of the dimension on the detection threshold for a fixed value of $\gamma$. From \Cref{thm::both_tests_above_criticality_broad_regimes}, the 2-sided test, when $d>d_c(\gamma),$ the detection threshold is $\delta = \frac{1}{2}-\frac{2(1-\gamma)}{d}.$ On the other hand, the 1-sided test has a detection threshold of $\delta = \frac{1}{2}-\frac{2(1-\gamma)}{d}$ or $\delta = \frac{2(1-\gamma)}{d}$. We can observe that as $d\to \infty,$
$$\frac{1}{2} - \frac{2(1-\gamma)}{d} \uparrow \frac{1}{2},  \quad \frac{2(1-\gamma)}{d}\downarrow 0.$$

Thus, the 2-sided test progressively improves with increasing dimension and its detection threshold approaches the parametric rate of $\delta = \frac{1}{2}$. On the other hand, the exponent gap for the 1-sided test widens, and the test could be close to optimal ($\delta = 1/2$) or close to powerless ($\delta = 0$) depending on the direction in which the local alternative $\theta_N$ approaches the null $\theta_1$. This can be seen in \Cref{fig::threshold_comparison}. In \Cref{section::simulations}, we show that this effect of direction can be drastic even for moderate dimensions such as $d=25.$

It is interesting to see what happens when $\gamma=0.$ This corresponds to the case of $k_N$ being fixed. In this case simply plugging $\gamma=0$ into previous expressions gives the critical dimension $d_c(\gamma) = 8$. Substituting $\gamma = 0$ in \Cref{thm::power_below_criticality} gives that for $d\leq 8,$ the detection threshold of the 1-sided test is given as $\delta = \frac{1}{4}$. Taking $\gamma=0$ in \Cref{thm::both_tests_above_criticality_broad_regimes} gives that for $d\geq 9,$ the detection can be $\delta = \frac{1}{2} - \frac{2}{d}$ or $\delta = \frac{2}{d}$ depending on the direction. These align exactly with the rates obtained in \citet{bhattacharya2020detectionthresholds} for the case of $K$ fixed. However, their results do not follow from ours since we require $k_N \to \infty.$ 

\end{itemize}

\begin{figure}[t]
    \centering
    \captionsetup{width = 0.9\linewidth}
    \includegraphics[width=0.7\linewidth]{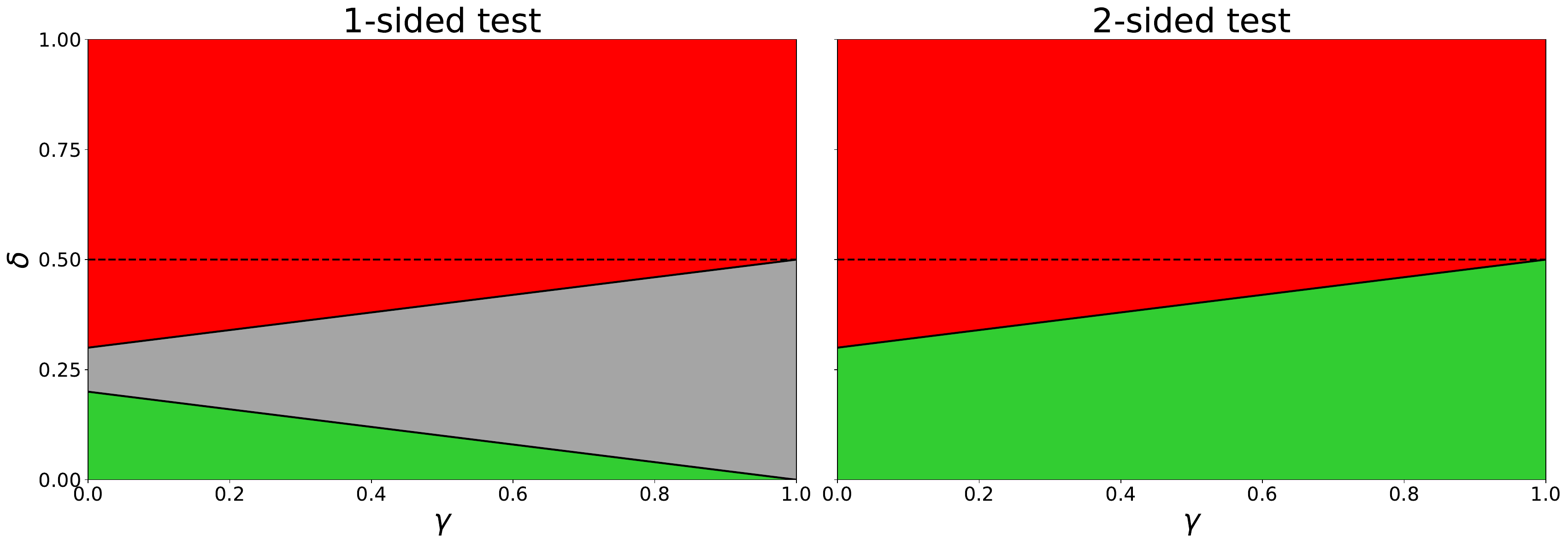}
    \caption{Limiting power comparison of the two tests for $d\geq 9$. The X and Y-axes, respectively, represent $\gamma$ and $\delta$. Red denotes limiting power $\alpha$, green denotes limiting power $1$ and gray denotes limiting power 0 or 1. The dotted horizontal line represents the parametric rate $\delta = \frac{1}{2}$.}
    \label{fig::power_comparison_above_criticality}
\end{figure}

We point out that when $b(h,\theta_1)=0,$ \Cref{thm::both_tests_above_criticality_broad_regimes} does not give the limiting power between the lower and upper thresholds. This corresponds to certain first-order degeneracy of the parametric family in the direction $h$, and the detection thresholds are different in this case. We elaborate on this point in the next section, where we sketch out the proof. Furthermore, our simulations cover the case of the equicorrelation family, where we do encounter a case of $b(h,\theta_1)=0.$ Our simulations indicate that even in this case, the 1- and 2-sided tests have different power properties with the 2-sided test performing better. The direction $h$ continues to play a role, albeit differently.

\Cref{thm::power_below_criticality,thm::both_tests_above_criticality_broad_regimes} almost completely describe the limiting power of the two tests. The only instance they do not cover is the limiting power at the exact lower and upper thresholds when $d>d_c(\gamma)$. The below result covers these cases.

\begin{theorem}\label{thm::both_tests_lower_upper_thresholds}

Consider the 1- and 2-sided tests with rejection regions as defined in \eqref{eq::rejection_region_for_1_sided_test} and \eqref{eq::rejection_region_for_2_sided_test}. Under the same assumptions and notation as \Cref{thm::both_tests_above_criticality_broad_regimes}, the limiting power at the lower and upper thresholds is given as follows:
\begin{itemize}
\item \textbf{Power at the lower threshold:} If $\epsilon_N = hN^{-\frac{1}{2}+\frac{2(1-\gamma)}{d}}$ with $b(h,\theta_1)\neq 0,$ then:
\begin{itemize}
    \item The 1-sided test has limiting power $\Phi(z_\alpha - b(h,\theta_1)).$
    \item The 2-sided test has limiting power $\Phi(z_\alpha - b(h,\theta_1)) + \Phi(z_\alpha + b(h,\theta_1)).$
\end{itemize}
\item \textbf{Power at the upper threshold:} If $\epsilon_N = hN^{-\frac{2(1-\gamma)}{d}}$ with $a(h,\theta_1)\neq b(h,\theta_1),$
\begin{itemize}
\item The 1-sided test has limiting power 0 or 1 if $a(h,\theta_1) - b(h,\theta_1)$ is positive or negative, respectively.
\item The 2-sided test has limiting power 1.
\end{itemize}
\end{itemize}

In all cases, $\Phi$ denotes the standard normal distribution function.
\end{theorem}

\Cref{thm::both_tests_lower_upper_thresholds} is proved in \appendixC.

\subsection{Proof sketch}\label{subsection::proof_sketch}

We briefly describe the proof strategy of \Cref{thm::power_below_criticality,thm::both_tests_above_criticality_broad_regimes,thm::both_tests_lower_upper_thresholds}. To test the null, we use the statistic
\begin{equation*}
\frac{1}{k_N\sqrt{N}}\left(T(\Gscr_{k_N}(\Zcal_N)) - \E_{H_0}(T(\Gscr_{k_N}(\Zcal_N)))\right)
\end{equation*}

Using the CLT under local alternatives given in \Cref{lemma::local_alternative_clt} the power analysis comes down to estimating the difference in null and alternate means,
$$\Delta_N := \frac{1}{k_N\sqrt{N}}(\mu_N(\theta_1,\theta_N)-\mu_N(\theta_1,\theta_1)),$$

where $\mu_N(\theta_1,\theta_2)$ denotes the expected value of $T(\Gscr(\Zcal_N))$ when $f=p(\cdot|\theta_1), g=p(\cdot|\theta_2).$

Suppose we take $\theta_N - \theta_1 = h\delta_N$ for some $h\in \R^p.$ Expanding the function $\mu_N(\theta_1,\cdot)$ around $\theta_1$ and analyzing the gradient and Hessian terms, we get that
\begin{equation}\label{eq::diff_of_means_heuristic_expansion}
\Delta_N \approx -a(h,\theta_1)N^{\frac{1}{2}}\delta_N^2 + b(h,\theta_1)N^{\frac{1}{2}}\left(\frac{k_N}{N}\right)^{\frac{2}{d}}\delta_N,
\end{equation}

where $a(h,\theta_1), b(h,\theta_1)$ are as defined in \eqref{eq::Hessian_term} and \eqref{eq::Gradient_term}, respectively. With the above expansion, we can now derive the limiting power for the 1- and 2-sided tests

We define the following, more general versions of the critical dimension.
\begin{align}
d_c^\lowernormal(k_N) := \liminf_{n\to \infty} \lfloor 8(1-\log_N k_N) \rfloor, \label{eq::lower_critical_dimension}\\
d_c^\uppernormal(k_N) := \limsup_{n\to \infty} \lfloor 8(1-\log_N k_N) \rfloor. \label{eq::upper_critical_dimension}
\end{align}

Note that for $k_N = N^\gamma,$ the critical dimensions above are the same and equal $d_c(\gamma)$ defined in \eqref{eq::critical_dimension_recalled}.

The 1-sided test has limiting power 0, $\alpha$ and 1 when $\Delta_N \to -\infty, 0, \infty$, respectively. The 2-sided test has limiting power $\alpha$ and $1$ when $|\Delta_N|\to 0,\infty$, respectively. Using \eqref{eq::diff_of_means_heuristic_expansion}, we can derive the limiting power. When, $d\leq d_c^\lowernormal(k_N),$ we have
\begin{itemize}
\item If $\delta_N\ll N^{-\frac{1}{4}},$ then $\Delta_N \to 0$ and the limiting power of both tests is $\alpha.$
\item If $\delta_N \gg N^{-\frac{1}{4}},$ then $\Delta_N \to -\infty$ and both tests have power $1.$
\end{itemize}

This is exactly the result of \Cref{thm::power_below_criticality}, with a more general version of the critical dimension. Similarly, we can derive a general version of \Cref{thm::both_tests_above_criticality_broad_regimes} by considering $d\geq d_c^\uppernormal(k_N).$ In this case, we have more general definitions of the lower and upper threshold. Specifically, we see:
\begin{itemize}
\item The lower threshold is given by $N^{-\frac{1}{2}}\left(\frac{N}{k_N}\right)^\frac{2}{d}.$ When $\delta_N$ decays faster than this rate, $\Delta_N\to 0$ and both tests have power $\alpha.$ When $k_N = N^\gamma,$ the lower threshold is equal to the lower threshold defined in \Cref{sec::local_power_above_criticality}.
\item The upper threshold is given by $\left(\frac{k_N}{N}\right)^\frac{2}{d}.$ When $\delta_N$ decays slower than this rate, the first term in \eqref{eq::diff_of_means_heuristic_expansion} dominates and $\Delta_N \to -\infty.$ For $k_N = N^\gamma,$ the upper threshold is exactly the upper threshold defined in \Cref{sec::local_power_above_criticality}.
\item The exponent gap is given by the rates $N^{-\frac{1}{2}}\left(\frac{N}{k_N}\right)$ and $\left(\frac{k_N}{N}\right)^\frac{2}{d}.$ In this gap, it is always the second term of \eqref{eq::diff_of_means_heuristic_expansion} that dominates, and $|\Delta_N|\to \infty$. Hence, the 2-sided test always has power $1$. The sign of $b(h,\theta_1)$ on the other hand determines the power of the 1-sided test. Depending on $b(h,\theta_1)$ being positive or negative, $\Delta_N\to \infty$ or $-\infty$, respectively, and the 1-sided test has power 0 or 1 accordingly.  
\end{itemize}

Similarly, at the exact thresholds $\delta_N = N^{-\frac{1}{2}}\left(\frac{N}{k_N}\right)^\frac{2}{d}$ and $\delta_N = \left(\frac{k_N}{N}\right)^\frac{2}{d}$, the limiting value of $\Delta_N$ can be found out from \eqref{eq::diff_of_means_heuristic_expansion}, and we get the general version of \Cref{thm::both_tests_lower_upper_thresholds}. The restrictions of $b(h,\theta_1)\neq 0$ and $a(h,\theta_1) - b(h,\theta_1)\neq 0$ in \Cref{thm::both_tests_above_criticality_broad_regimes,thm::both_tests_lower_upper_thresholds} come from the fact that we only consider the first two terms in the expansion \eqref{eq::diff_of_means_heuristic_expansion}. The cases of $b(h,\theta_1)=0$ and $a(h,\theta_1) - b(h,\theta_1)$ are not covered by the above expansion. However, by considering further terms, it should be possible to cover other cases as well.

\eqref{eq::diff_of_means_heuristic_expansion} provides intuition for the exponent gap of the 1-sided test. Recall that the test statistic counts the number of cross-group edges in the $K$-NN graph formed by pooling the two groups $\Xcal_N := {X_1,\dots,X_{N_1}}\sim P_{\theta_1}$ and $\Ycal_N := {Y_1,\dots,Y_{N_2}} \sim P_{\theta_N}$. The number of points in $\Ycal_N$ with nearest neighbors in $\Xcal_N$ depends on both cross-group distances $|X_i-Y_j|$ and within-group distances $|Y_i-Y_j|$. As $|\theta_N-\theta_1|$ increases, cross-group distances typically increase, corresponding to the first term of \eqref{eq::diff_of_means_heuristic_expansion}, where $-a(h,\theta_1)$ is always negative. In contrast, within-group distances may increase or decrease depending on whether $P_{\theta_N}$ is an overall expansion or shrinking of $P_{\theta_1}$. This effect is captured by the second term, where $b(h,\theta_1)$ quantifies the induced distributional expansion/shrinkage under a deformation of $P_{\theta_1}$ in the direction $h$. Thus, the difference in means reflects the competing effects of parameter separation and distributional expansion/shrinkage.

If $\theta_N\neq\theta_1$ are fixed, the shrinkage term in \eqref{eq::diff_of_means_heuristic_expansion} vanishes asymptotically, leaving parameter separation as the dominant effect, consistent with \Cref{propn::weak_limit_of_stat}. However, when $\theta_N\to\theta_1$, the shrinkage term dominates over a certain range of $|\theta_N-\theta_1|$, precisely the exponent gap of the 1-sided test. If $b(h,\theta_1)<0$, corresponding to overall shrinkage, within-group distances decrease relative to cross-group distances, reducing the number of cross-group edges and yielding high power for the 1-sided test. Conversely, if $b(h,\theta_1)>0$, corresponding to expansion, within-group distances increase relative to cross-group distances, producing more cross-group edges than under the null, and the 1-sided test has limiting power 0.
\section{Simulations}\label{section::simulations}

We now come to the performance of the test in practice. Increasing the value of $K$ with sample size creates a larger graph which comes with additional computational costs. We study the impact of this on execution time in \Cref{subsection::runtime_analysis}. We compare the runtime of the two sample test when using a permutation test, against using the asymptotic distribution.

More pertinent to the theoretical results presented earlier, we consider the empirical power of 1- and 2-sided tests on synthetic and real data. For the experiments on synthetic data in \Cref{subsection::power_comparison_simulation}, we consider multiple families of scale perturbations of Gaussian distributions. An important goal is to demonstrate that the power of the 1-sided $K$-NN test can be significantly affected by the direction in which the alternate approaches the null, while the 2-sided test is more stable. In the real data example in \Cref{section::real_data_example}, we apply the two sample tests to the task of distinguishing between two types of news articles.

\subsection{Runtime analysis}\label{subsection::runtime_analysis}

In \Cref{subsection::clt_conditional_stat} we describe the implementation of graph-based two sample tests as a permutation test. Alternately, \Cref{thm::clt_conditional_stat} proposes a pivotal statistic to test the null. We will compare the runtime of both algorithms. In both cases, the first step is to compute all pairwise distances to construct the $K$-NN graph. This operation has a time complexity of $O(N^2)$ for $K=o(N)$ and is common to both methods. We now discuss the method-specific overheads.

The discussion preceding \Cref{thm::clt_conditional_stat} shows that implementing the permutation test requires sampling $B$ uniform permutations of $\{1,\dots,N\}$ and recomputing the statistic $B$ times, which requires $O(NKB)$ steps. In \Cref{fig::runtime_plots}, the right-hand panel plots the run time of the permutation test for a range of values of $N,K$ with $B=300$ being fixed.

An alternative way is to use the conditional CLT given in \Cref{thm::clt_conditional_stat}, which requires computing the conditional variance under the null. To find the conditional variance, we refer again to the discussion preceding \Cref{thm::clt_conditional_stat}. Under the null, the data consist of $N_1 + N_2 = N$ i.i.d points $Z_1,\dots,Z_N$ with labels $c_1,\dots,c_N$, where $N_1, N_2$ are the group sizes. Since the labels are exchangeable under the null, the conditional distribution of the statistic $T(\Gscr_K)$ is given as
$$T(\Gscr_K)\mid \{Z_1,\dots,Z_N\} \sim \sum_{e\in \Gscr_K} U_e,$$

where $\{U_e\}_{e\in\Gscr_K}$ is a collection of Bernoulli random variables indexed by the edges of $\Gscr_K,$ the $K$-NN graph formed from the sample at hand. By the exchangeability of the labels, the covariance structure of the $U_e$'s can be easily obtained.

\begin{itemize}
    \item $U_e \sim \ber\left(\frac{N_1N_2}{N(N-1)}\right)$ for all $e\in \Gscr_K$.
    \item $\Cov(U_e,U_f) = \frac{N_1N_2(N_2-1)}{N(N-1)(N-2)} - \frac{N_1^2N_2^2}{N^2(N-1)^2}$ if $e=(i,j), f=(i,k)$ for $i,j,k$ distinct. 
    \item $\Cov(U_e,U_f) = \frac{N_1N_2(N_1-1)}{N(N-1)(N-2)} - \frac{N_1^2N_2^2}{N^2(N-1)^2}$ if $e=(i,j), f=(k,j)$ for $i,j,k$ distinct. 
    \item $\Cov(U_e,U_f) = - \frac{N_1^2N_2^2}{N^2(N-1)^2}$ if $e=(i,j), f=(j,k)$ for $i,j,k$ distinct. 
    \item $\Cov(U_e,U_f) = -\frac{N_1^2 N_2^2}{N^2(N-1)^2}$ if $e = (i,j), f = (j,i)$ for $i\neq j.$
    \item $\Cov(U_e,U_f) = \frac{N_1N_2(N_1-1)(N_2-1)}{N(N-1)(N-2)(N-3)} - \frac{N_1^2N_2^2}{N^2(N-1)^2}$ for $e=(i,j), f=(k,l)$ for $i,j,k,l$ distinct.
\end{itemize}

\begin{figure}[t]
    \centering
    \captionsetup{width = 0.9\linewidth}
    \includegraphics[width=0.7\linewidth]{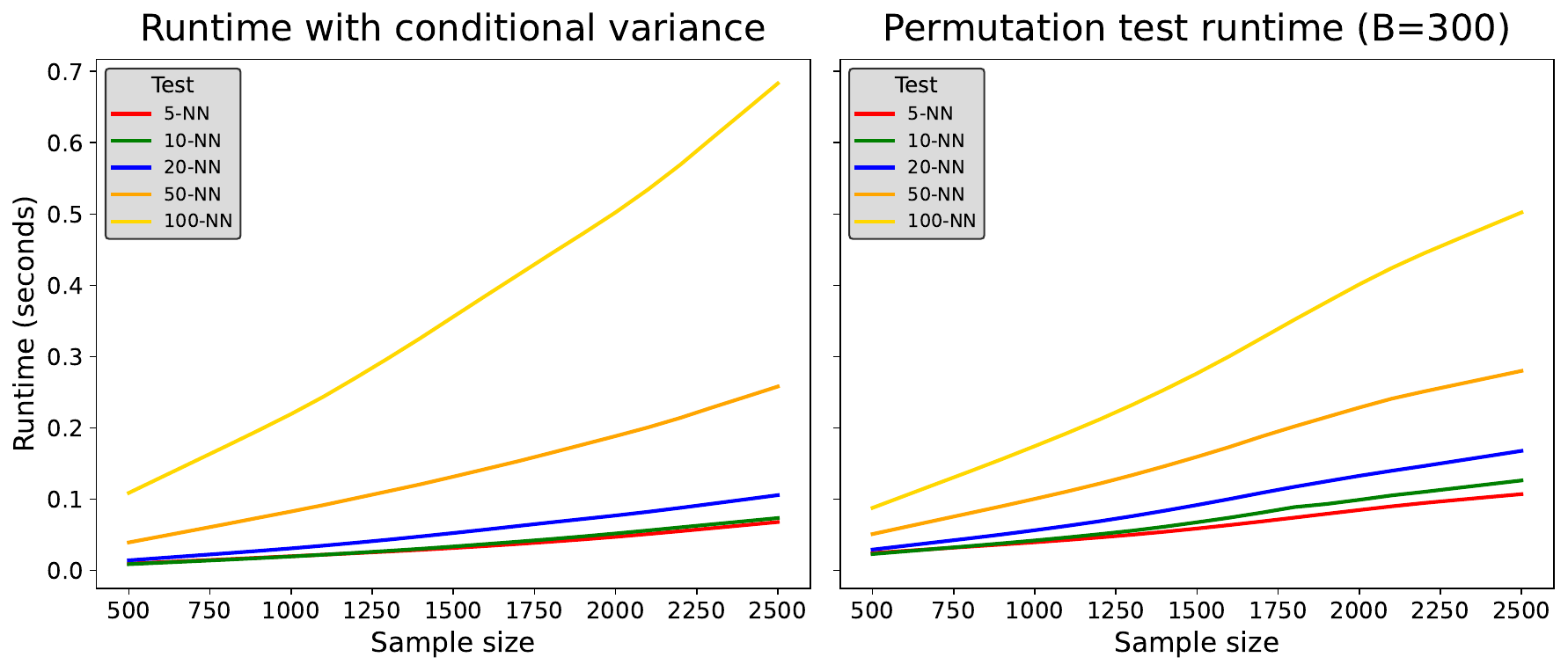}
    \caption{Runtime comparison with $K = 5,10,20,50,100$}
    \label{fig::runtime_plots}
\end{figure}

Using the above, the conditional variance $\Var\left(T(\Gscr_K)\mid Z_1,\dots,Z_N \right)$ under the null follows from the adjacency structure of the graph $\Gscr_K.$ The conditional variance computation requires counting the number of ordered pairs of edges falling in each of the categories above. Using the specific structure of the $K$-NN graph, some calculations can be simplified. For instance, the number of pairs in the second group is immediately seen to be $NK(K-1).$ However, other types of pairs are non-trivial to compute and the time complexity is $O(NK^2).$ The left-hand panel of \Cref{fig::runtime_plots} plots the runtime of the test using the conditional variance for various values of $N,K$.

\subsection{Power comparisons on synthetic data}\label{subsection::power_comparison_simulation}

We now consider power of the 1- and 2-sided tests on synthetic data in two examples.

\subsubsection*{Normal scale perturbations}

In the first set of simulations, we consider an $m$-parameter Gaussian family  with the following  parametrization. For $\theta\in \R^m_+$, $p(\cdot|\theta)$ denotes the density of $N(0,\Sigma_\theta)$ on $\R^d,$ where $\Sigma_\theta$ is the diagonal matrix with the first $m$ diagonal entries given by the components of $\theta$ and the rest equal to 1, with $m\leq d.$ Thus, the density is given by
$$p(x\mid \theta) = \frac{1}{(2\pi)^{\frac{d}{2}} (\theta_1\dots\theta_m)^\frac{1}{2}} \exp\left(-\sum_{i=1}^m\frac{x_i^2}{2\theta_i} + \sum_{j>m}\frac{x_i^2}{2}\right),$$

We take the null to be some parameter $\theta_1\in \R_+^m $, and the local alternative is given by $\theta_N = \theta_1 + hN^{b}$ for some choice of $h\in \R^m\setminus \{0\}$ and some negative exponent $b$. This represents changes in scale, where the local alternative is obtained by perturbing some of the variance components.

The family described above is supported over all of $\R^d$, while our results hold only for families with compact support. For this purpose, we can consider the densities $\{p_L(\cdot \mid \theta)\}_\theta$ which represent the above family truncated to $[-L,L]^d$. In this case, the truncated density is given by
$$p_L(x\mid \theta) = \left(Z_L(\theta)\right)^{-1} p(x\mid \theta)  \indicator\left\{x\in [-L,L]^d\right\},$$
for $x\in [-L,L]^d$, where $Z_L(\theta)$ is the normalizing constant. Following the second part of \Cref{thm::both_tests_above_criticality_broad_regimes} we see that the power in the exponent gap depends on the sign of the term $b(h,\theta_1)$ defined in \eqref{eq::Gradient_term}. The same equation shows that,
$$b(\theta_1,h) \propto \bigintssss_S h^T\nabla_{\theta_1}\left(\frac{\text{tr}(\text{H}_xp(x|\theta_1))}{p(x|\theta_1)}\right)p^{\frac{d-2}{d}}(x|\theta_1) ~ dx,$$
where $S$ denotes the support of the parametric family.  Hence, it suffices to determine the sign of the above integral. Since the density is available in this case, we can calculate this quantity. A long but simple calculation shows that for the truncated parametric family we consider,
\begin{align*}
b(h,\theta_1) &\propto \sum_{i=1}^m h_i(1-2\E(X_i^2)),
\end{align*}
where $X_i \sim N\left(0,\frac{d}{d-2}\theta_i\right) \mid_{[-L,L]}$. Thus, the sign of $b(h,\theta_1)$ is determined by the sign of $\sum h_i.$ The complete calculation is given in \appendixD.

In our simulations, we will consider the untruncated family. Furthermore, we will take $d=25$ and $m=10$, with the null being $\theta_1 = \indicator_m$. Using the above steps, we get that for the untruncated family
\begin{equation}\label{eq::gradient_term_sign_in_normal_scale_family}
b(h,\indicator_m) \propto -\sum_{i=1}^m h_i.
\end{equation}

\begin{figure}
    \centering
    \captionsetup{width = 0.9\linewidth}
    \includegraphics[width=0.7\linewidth]{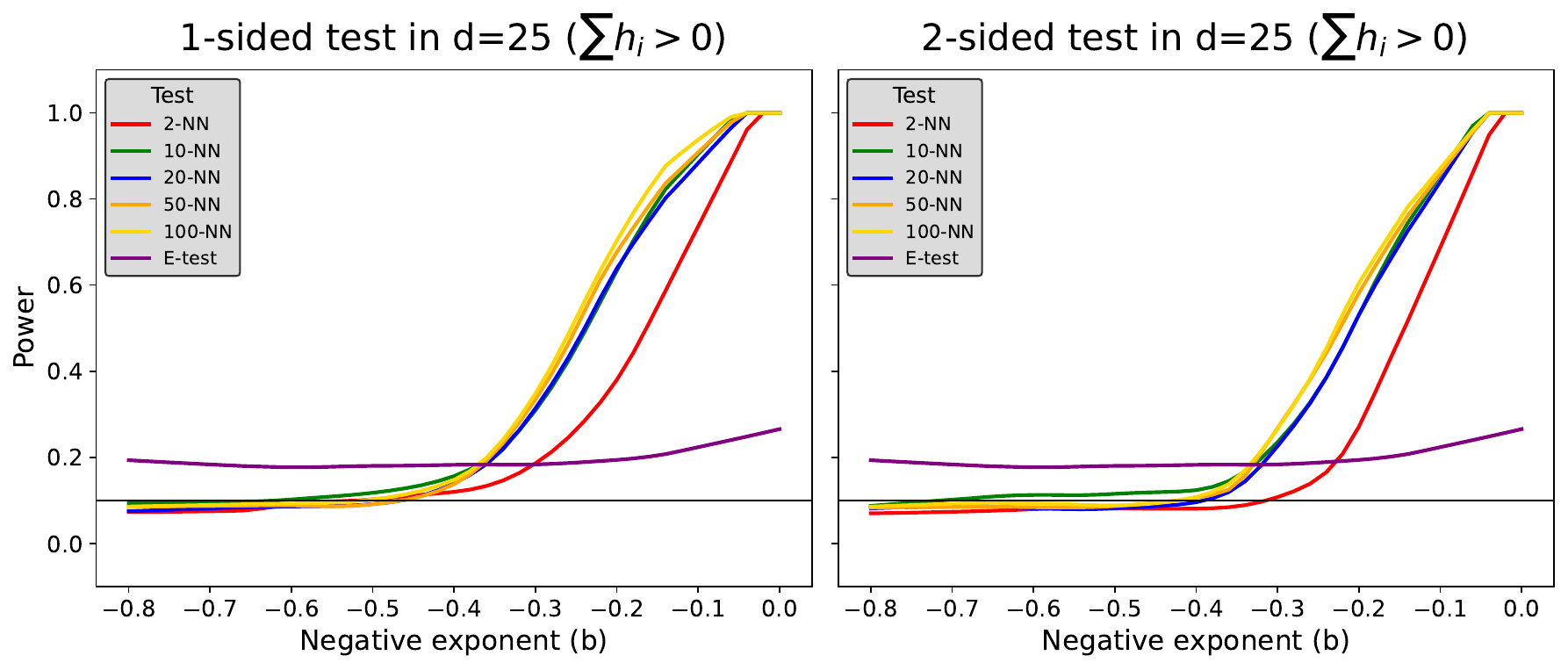}
    \caption{Power comparison of 1- and 2-sided test for normal scale perturbations with $\sum h_i>0$.}
    \label{fig::spiked_pos_sum_lineplots}
\end{figure}

Hence, $b(h,\theta_1)>0$ when $\sum_{i=1}^{10} h_i<0$ and vice versa. Given a growth rate of $k_N = N^\gamma$ for the number of neighbors, we see from \eqref{eq::critical_dimension_recalled} that the critical dimension $d_c(\gamma)$ satisfies $d_c(\gamma)\leq 8$ for all $\gamma.$ Therefore, for $d=25$, \Cref{thm::both_tests_above_criticality_broad_regimes} applies for all values of $\gamma$. With this, we can state the predicted power of the 1- and 2-sided tests. Recall that the null is $\theta_1$ and the local alternative is $\theta_N = \theta_1 + hN^{b}$ for some negative exponent $b\in [-1,0]$ while the number of neighbors is $k_N = N^\gamma$ for some $\gamma<1/4.$

\begin{itemize}
\item For $b<-\frac{1}{2} + \frac{2(1-\gamma)}{d}$, both tests have limiting power $\alpha.$
\vspace{2mm}
\item For $-\frac{1}{2} + \frac{2(1-\gamma)}{d}<b<-\frac{2(1-\gamma)}{d},$ when $\sum_{i=1}^{10}h_i$ is negative or positive the 1-sided test has limiting power 0 or 1, respectively. The 2-sided test has limiting power 1 in both cases.
\vspace{2mm}
\item For $b>-\frac{2(1-\gamma)}{d},$ the limiting power of both tests is 1.
\end{itemize}

\begin{figure}[t]
    \centering
    \captionsetup{width = 0.9\linewidth}
    \includegraphics[width=0.7\linewidth]{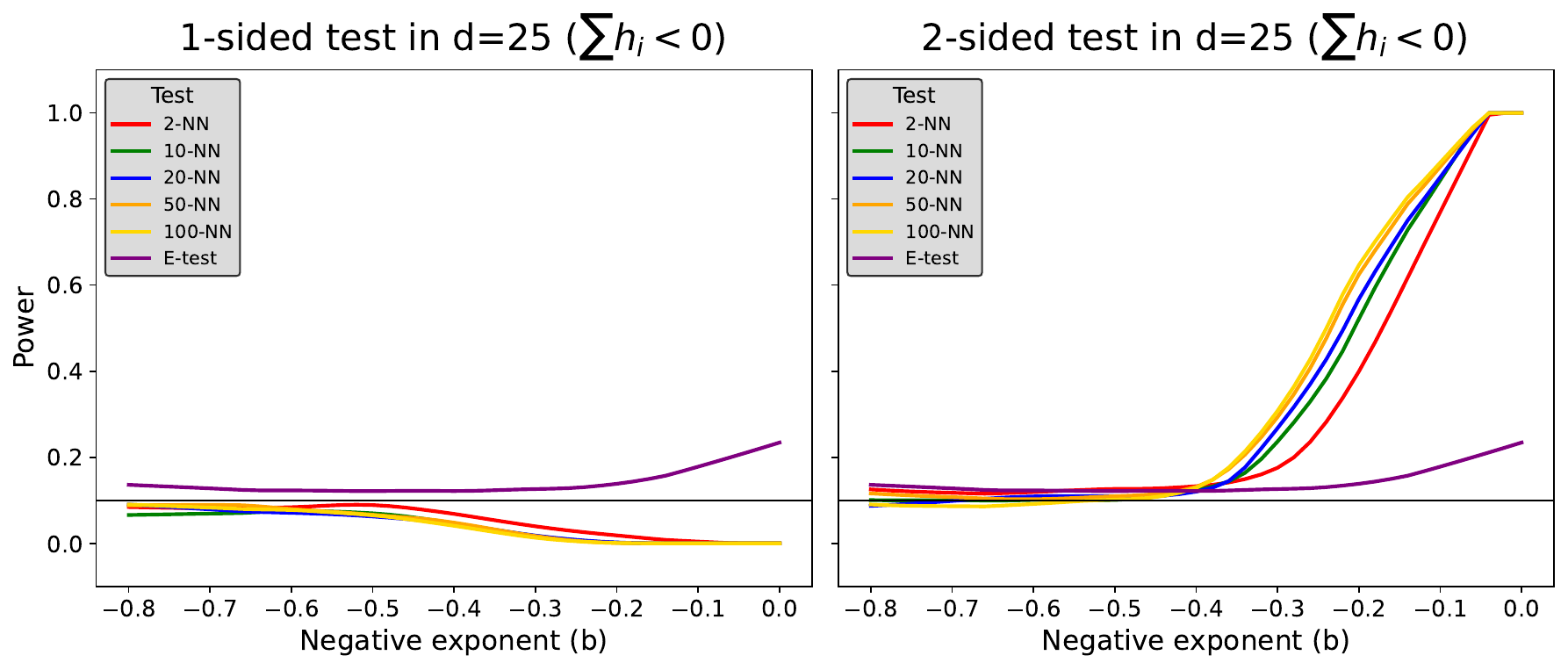}
    \caption{Power comparison of 1- and 2-sided tests for normal scale perturbations with $\sum h_i < 0$.}
    \label{fig::spiked_neg_sum_lineplots}
\end{figure}

In the simulation setup, we take the group sizes to be $N_1 = 12000, N_2 = 8000$ giving a total sample size of $N=20000.$ The null is $\theta_1 = \indicator_{10}$ which corresponds to the standard normal in $d=25$ dimensions. To simulate local alternatives, we take $\theta_N = \theta_1 + hN^{b}$ for $b$ taking a range of values in $[-1,0].$ We consider the $K$-NN test for $K=2,10,20,50,100$. For each choice of $b$ and $K$, we consider the empirical power across 100 trials. To demonstrate the change in power according to the direction, we consider a vector $h\in R^{10}$ with $\sum h_i >0$ and find the empirical power when the direction is given by $h$ and $-h.$ The null was tested at the level $0.1$.
\begin{itemize}
\item \Cref{fig::spiked_pos_sum_lineplots} plots the empirical power of the 1- and 2-sided tests in the normal scale model for $K=2,10,20,50,100.$ In \Cref{fig::spiked_pos_sum_lineplots}, we have considered a vector $h\in R^{10}$ with $\sum h_i > 0.$ In this case, the power of both tests is similar. This aligns with the results since in this case $b(h,\theta_1)<0$ as shown above, and both tests have power 1 for $b>-\frac{1}{2} + \frac{2(1-\gamma)}{d}.$ Moreover, the power generally increases for both tests with increasing values of $K.$ This also aligns with the detection thresholds derived in \Cref{thm::both_tests_above_criticality_broad_regimes}.

For comparison, we have also plotted the power of the energy distance test (\citet{szekely2004testing}). The simulation setup is the same as for the $K$-NN tests, but with sample sizes $N_1=600, N_2=400$ for computational reasons. Surprisingly, the energy distance test performs poorly in this setting.

\item \Cref{fig::spiked_neg_sum_lineplots} plots the power of the two tests when the deviation is given $-h$. The direction is now given by the negative of the vector $h$ considered for the simulations in \Cref{fig::spiked_pos_sum_lineplots}. Hence, we have $\sum(-h_i)<0$ which gives $b(h,\theta_1)>0$. In this case, we see that the 1-sided test is extremely poor and has power close to 0 even for the negative exponent $b$ close to 0. At first glance, this seems counterintuitive. The case $b=0$ corresponds to the case of fixed alternatives. Since the 1-sided test is consistent (\Cref{propn::weak_limit_of_stat}), it should actually have power close to 1 in this setting. This discrepancy can be explained by \Cref{thm::both_tests_above_criticality_broad_regimes}. Since in this case $b(h,\theta_1)>0,$  the result shows that the detection threshold for $k_N = N^\gamma$ is given by $N^{-\frac{2(1-\gamma)}{d}}.$ For $d=25$, the exponent is extremely close to 0 for all values of $\gamma$. As a result, in our simulations the 1-sided test has power close to 0 even for $b\approx 0.$

On the other hand, the power of the 2-sided test is almost unchanged, and we also see an improvement in the power with increasing $K$. Similarly, the energy distance test is also unchanged and performs poorly even after reversing the direction.

\end{itemize}

Further simulations exploring alternate settings are given in \appendixD.

\subsubsection*{Normal equicorrelation family}

We will now consider the family of distributions over $R^d$ given by $p(\cdot \mid \rho)$ given by $N(0,\Sigma_\rho)$ where $\Sigma_\rho$ is the $d-$dimensional matrix consisting of 1's on the diagonal and all other entries being $\rho$. This is the well known normal equicorrelation family.

Since $\Sigma_\rho = (1-\rho)I_d + \rho \indicator_d \indicator_d^T,$ there exists an orthogonal matrix $U$ such that
$$U \Sigma_\rho U^T = I_d + \rho\cdot \diag(d-1,-1,\dots,-1).$$
In addition, $U$ does not depend on $\rho$. Further, we use the Euclidean distance to calculate distances between points, which is unchanged under orthogonal transformations. Hence, the normal equicorrelation family is a subset of the larger family of diagonal perturbations of the normal covariance matrix we considered in the previous section.

We take the null to be $\rho_0 = 0$ which corresponds to the standard normal, and the alternate to be $\rho_N = hn^b$ for some negative exponent $b$ and some nonzero $h\in \R$. In terms of the normal scale family of the previous section, this corresponds to $\theta_1 = \indicator_d$, $\theta_N = \theta_1 + \rho h N^b$, where $h = (d-1)e_1 - \sum_{i=2}^d e_i.$ The number of variance parameters $m$ is now equal to the ambient dimension $d$.

Following the calculation in \eqref{eq::gradient_term_sign_in_normal_scale_family}, it is enough to determine the sign of $\sum h_i$ to determine the sign of $b(h,\theta_1)$ and hence find the detection thresholds. In this case $\sum h_i = 0.$ Therefore, $b(h,\theta_1)=0$ and the results of \Cref{thm::both_tests_above_criticality_broad_regimes} do not apply. However, even in this case the 1- and 2-sided tests can behave quite differently depending on the sign of $\rho$.

\begin{figure}[t]
    \centering
    \captionsetup{width = 0.9\linewidth}
    \includegraphics[width=0.7\linewidth]{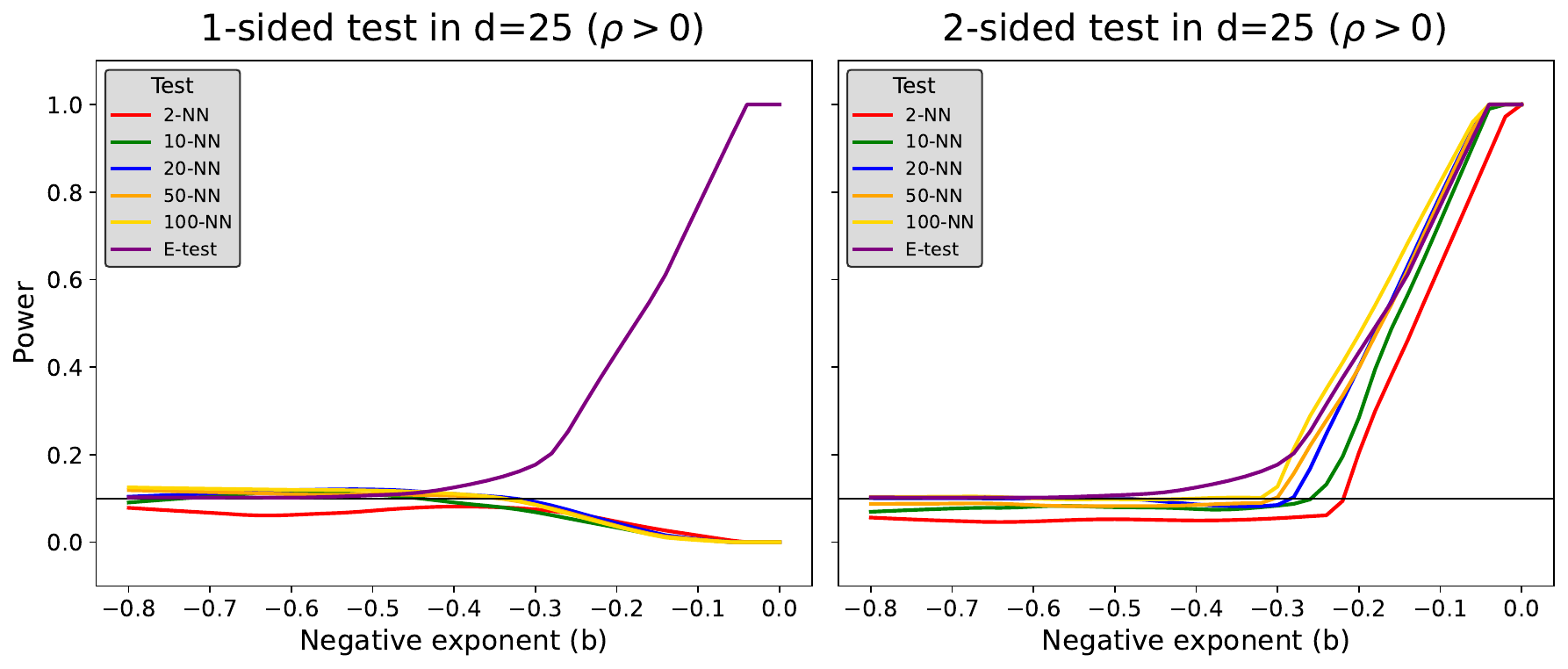}
    \caption{Power comparison of 1- and 2-sided tests in the normal equicorrelation family with $\rho>0.$ We take $d=25$, and group sizes $N_1 = 12000, N_2 = 8000$.}
    \label{fig::equicorrelation_family_lineplots}
\end{figure}

\Cref{fig::equicorrelation_family_lineplots} plots the power of the 1- and 2-sided tests for $\rho>0.$ There is again a stark difference between the two tests, the 1-sided test being quite poor and having power close to 0 even for $b\approx 0$. Although our results do not apply in this case since $b(h,\theta_1)=0,$ we still see that the 1-sided test can have extremely low power. The 2-sided test, on the other hand, performs well, as does the energy distance test. Furthermore, its power improves with increasing $K$.

In the case of $\rho<0,$ we chose not to include a plot since all tests consistently displayed low power in our simulations. This is expected since the equicorrelation family imposes the restriction $-\frac{1}{d-1}\leq \rho$ to ensure $\Sigma_\rho$ is non-negative definite. For $d=25,$ this $-\frac{1}{24}<\rho<0$ when considering $\rho<0.$ Even for large sample sizes, the tests fail to detect a difference between the two distributions.

\subsection{Real data experiment}\label{section::real_data_example}

To evaluate the proposed test on real data, we used the 20 Newsgroups dataset \citep{lang1995newsweeder}, consisting of approximately 18,000 news articles across 20 categories. Documents were embedded into a 384-dimensional Euclidean space using the pretrained \texttt{all-MiniLM-L6-v2} model from the \texttt{sentence-transformers} library \citep{reimers-2019-sentence-bert}. We restricted attention to the \texttt{rec.sport.baseball} and \texttt{rec.sport.hockey} categories (roughly 1,000 articles each). After computing embeddings for the full dataset, we applied PCA and retained enough components to explain 90\% of the variance, resulting in 200-dimensional embeddings for the subsequent analysis.

\begin{figure}[h]
\centering
\captionsetup{width = 0.9\linewidth}
\includegraphics[width=0.7\linewidth]{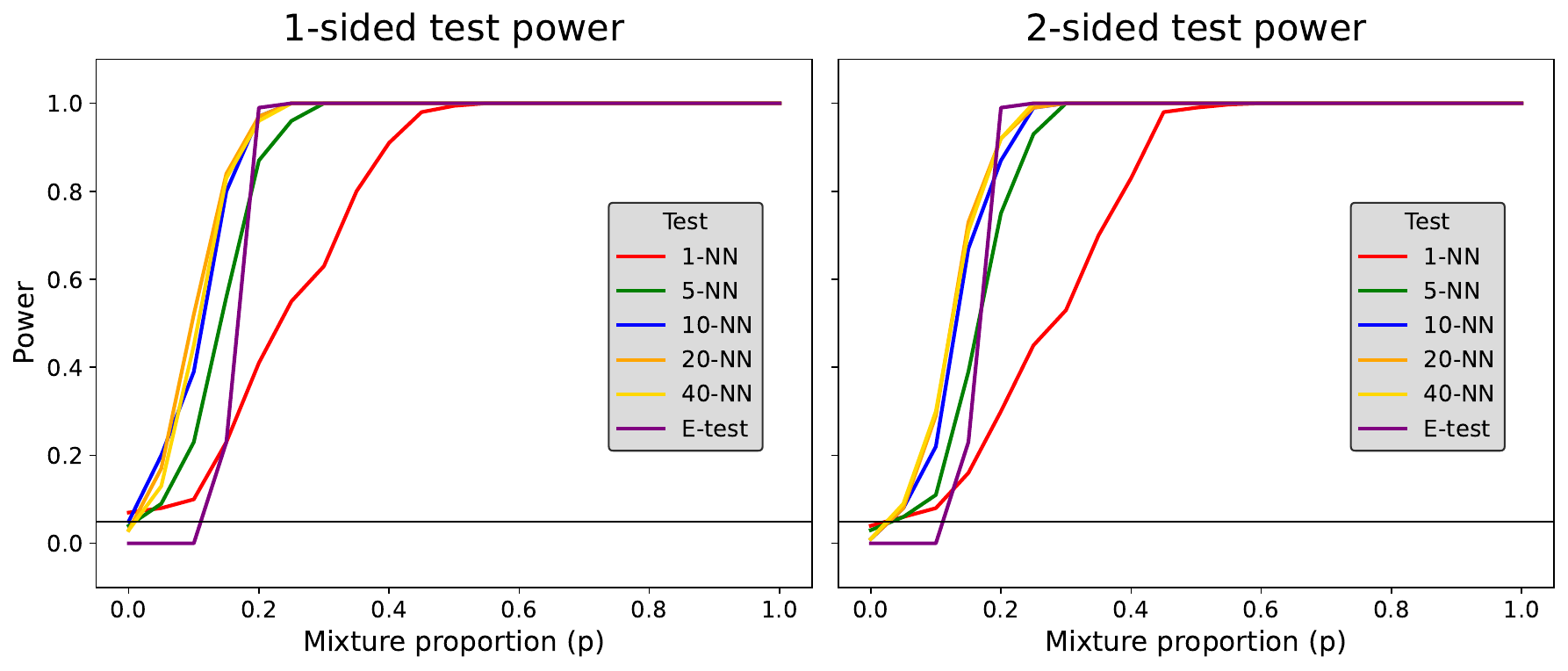}
\caption{Performance of the $K$-NN two sample tests when distinguishing sports articles.}
\label{fig::real_data_power}
\end{figure}

Let $P_\baseball$ and $P_\hockey$ denote the distributions of the baseball and hockey articles. We considered
$$H_0:F=G=P_\baseball
\quad \versus \quad
H_1:F=P_\baseball,;
G=(1-p)P_\baseball+pP_\hockey,$$
where $p\in[0,1]$ allows interpolation between the null $(p=0)$ and a fully distinct alternative $(p=1)$. For each $p$ and $K\in{1,5,10,20,40}$, we sampled 500 baseball articles for $F$ and 500 articles from the mixture for $G$, repeated the experiment 100 times, and estimated the empirical power. The results are shown in \Cref{fig::real_data_power}. Both tests exhibit similar performance, with power increasing as $K$ increases, while the energy distance test remains competitive.

Based on the synthetic and real data experiments, we recommend the 2-sided test in practice. Since our results address the case of growing $K$, there is also the question of choosing $K$ to boost the power. On one hand, all the simulations suggest that the power of the 2-sided test generally improves with increasing $K$. However, larger values of $K$ increase dependence in the graph which degrades the normal approximation of \Cref{thm::clt_conditional_stat}. Indeed, taking $K=N$ always gives the fully connected graph and hence a constant statistic. Our simulations suggest that $K\leq N/30$ provides a good compromise. The experiments on real and synthetic indicate that the power is generally increasing up to this value of $K$. Additional simulations in \appendixD show that the statistic is approximately normal in distribution in this range of $K$ while increasing $K$ beyond this point can impact the quality of the normal approximation. Developing a procedure that adaptively selects $K$ from multiple candidates remains an interesting direction for future work.

\bibliography{bibliography}

\appendix

\crefalias{section}{appendix}
\crefname{appendix}{appendix}{appendices}
\Crefname{appendix}{Appendix}{Appendices}

\section{Initial technical results}

\label{appendix::initial_technical_results}

This appendix is dedicated to some initial technical results that we will frequently use in our calculations. We begin with a standard bound on the lower tails of Poisson random variables which follows from the Chernoff bound for Poisson random variables.



\begin{lemma}\label{lemma::poisson_bound}

Let $X$ be a Poisson random variable with mean $\mu.$ Then, for $0\leq t\leq 1,$
$$\prob(X\leq (1-t)\mu)\leq \exp\left(-\frac{t^2\mu}{2}\right).$$

For all $t>0,$
$$\prob(X\geq (1+t)\mu)\leq \exp\left(-\frac{t^2\mu}{2}\right).$$
\end{lemma}

A corollary of the above lemma is the following concentration inequality for $\Gamma(M,1)$ random variables for $M\in \N.$

\begin{lemma}\label{corollary::gamma_concentration}
Let $X\sim \Gamma(M,1)$ for some $M\in \N.$ Then
$$\prob(X\geq M+u)\leq \exp\left(-\frac{(u+1)^2}{2(M+u)}\right).$$
\end{lemma}

\begin{proof}
Since $M\in \N,$ we know that the CDF $F_X$ of $X$ is
$$F_X(v) = 1-\sum_{k=0}^{M-1}\frac{v^k}{k!}e^{-v} = \prob(\poi(v)\geq M).$$ Hence,
\begin{align*}
    \prob(X\geq M+u) &= 1-\prob(\poi(M+u)\geq M)\\
    &= \prob(\poi(M+u)\leq M-1)\\
    &\leq \exp\left(-\frac{(u+1)^2}{2(M+u)}\right) \quad \ldots \quad \textnormal{from \Cref{lemma::poisson_bound}.}
\end{align*}

\end{proof}

The core set-up involves a Poisson process with intensity function $N_1 f + N_2 g$ for some densities $f,g.$ We will often be interested in the probability of two points being nearest neighbors. Specifically, we want to know how close the typical nearest neighbor of a point is to it. For this, we will need to integrate the intensity  function over balls of small radii. The following lemma will prove useful.

\begin{lemma}\label{lemma::ball_and_surface_integrals}

Let $S$ be an open set in $\R^d$ with $f$ a real valued, three times continuously differentiable function defined on $S$. Let $x\in S$ and $H_x f(x)$ denote the Hessian of $f$ at $x.$ Let $B(x,r)$ denote the ball of radius $r$ around $x$ such that $B(x,r)\subset S.$ Let $\partial B(x,r)$ denote its boundary. Then, as $r\to 0,$
\begin{align}
    \int_{B(x,r)} f(z)~dz &= f(x)V_d r^d + \frac{V_d \tr(\Hnormal_x f(x))}{2(d+2)} r^{d+2} + \delta_1(x,r), ,\label{eq::ball_integral_asymptotics}\\
    \int_{\partial B(x,r)} f(z)~dz &= f(x)dV_d r^{d-1} + \frac{V_d\tr(\Hnormal_x f(x))}{2} r^{d+1} + \delta_2(x,r) \label{eq::surface_integral_asymptotics},
\end{align}

where $|\delta_1(x,r)|\leq C_1(x)r^{d+3}$ and $|\delta_2(x,r)|\leq C_2(x) r^{d+2}$ for all $x,r$ as above, for some non-negative functions $C_1,C_2$ of $x.$\\

The functions $C_1,C_2$ can be taken as constants if $f$ has uniformly bounded third partial derivatives on $S.$ 
\end{lemma}

\begin{proof} WLOG we assume that $x=0.$ Expanding $f$ for $y$ near $0$ gives
\begin{equation}\label{eq::taylor_expansion}
f(y) = f(0) + \nabla f(0)^T (y-x) + \frac{1}{2}y^T \Hnormal f(x) y + O(\|y\|^3).
\end{equation}

The constant in the big-O term depends only on $x.$ If $f$ has bounded third derivatives, then the constant depends only on $f.$

We prove \eqref{eq::surface_integral_asymptotics} first. We change to spherical coordinates and parametrize $\partial B(x,r)$ by the variables $\psi_1,...,\psi_{d-1}$ with $\psi_{d-1} \in [0,2\pi)$ and $\psi_i\in [0,\pi)$ for $i\neq d-1.$ The change of variables is given by
\begin{align*}
    y_i &= r\sin(\psi_1)...\sin(\psi_{i-1})\cos(\psi_i) \text{ for }i\neq d,\\
    y_d &= r\sin(\psi_1)...\sin(\psi_d).
\end{align*}

Let $J$ denote the Jacobian for the change of coordinates. It can be shown that the determinant of the Jacobian $J$ at $\Psi = (\psi_1,...,\psi_{d-1})$ is given by 
$$|J(\Psi)| = r^{d-1}\sin^{d-2}(\psi_1)\sin^{d-3}(\psi_2)...\sin(\psi_1).$$

We now find the integral over $\partial B(0,r)$ by integrating every term in the Taylor expansion individually.
$$\int_{\partial B(x,r)}f(0)~dy = f(0)\int J(\Psi)~d \Psi = f(0)dV_d r^{d-1}.$$

To calculate the integral of the gradient term, we can notice that due to the parametrization given above,
$$\int_{\partial B(0,r)} y_i~dy = 0.$$

Hence, the gradient term integrates to $0.$ To find the Hessian term, we first notice that for $i\neq j,$
$$\int_{\partial B(0,r)}y_iy_j~dy = 0.$$

Further more,
\begin{align*}
\int_{\partial B(0,r)}y_i^2 ~dy &= \frac{1}{d}\int_{\partial B(0,r)} \|y\|_2^2 ~dy\\
&= \frac{1}{d}r^2\int |J(\Psi)|~d\Psi = V_dr^{d+1},
\end{align*}
where the first equality is due to symmetry. Hence, the integral of the Hessian term is given by
\begin{align*}
\frac{1}{2}\int_{\partial B(0,r)} y^T \Hnormal f(0) y ~dy &= \frac{1}{2} \int_{\partial B(0,r)} \sum h_{ii}y_i^2 ~dy\\
&= \frac{V_d \tr(\Hnormal f(0))}{2} r^{d+1}.
\end{align*}

Finally, the big-O term is bounded by $C_x r^3.$ Hence, by using the same ideas as above we get that the integral of the remainder term can be bounded by $O(r^{d+2}).$ This proves \eqref{eq::surface_integral_asymptotics}. Since
$$\int_{B(0,r)} f(z)~dz = \int_0^r\left(\int_{\partial B(0,u) }f(z)~dz\right)~du,$$

\eqref{eq::ball_integral_asymptotics} follows by integrating the individual expressions from 0 to $r.$
\end{proof}

A combination of the above results shows that the $k_N$-nearest neighbors of a point all lie at a distance of order smaller than $(k_N/N)^\frac{1}{d}$ from it. This is made precise in the following result.\\




\begin{lemma}\label{lemma::nearest_neighbor_distance_bound}

Let $k_N = o(N)$ and let $h$ be a density with support $S$ that satisfies the following: 
\begin{enumerate}[(a)]
    \item $S$ is compact, convex with non-empty interior satisfying $S = \overline{\textnormal{int}(S)}$ and $\partial S$ has Lebesgue measure zero.
    \item  $h$ is thrice continuously differentiable and uniformly bounded below on $S.$
\end{enumerate}

Let $\Gcal_N$ denote the Poisson process with intensity $Nh$, $\Gcal_N(x,r)$ denote the set of points in $\Gcal_N$ lying in $B(x,r)$ and $r_N(K) := \left(K \cdot \dfrac{\max\{k_N,(\log N)^2\}}{N}\right)^\frac{1}{d}.$ Let $\alpha>0$ be given.\\

Then, there exists a $K>0$ such that eventually for all large $N$ we have
\begin{equation*}
\sup_{x\in \textnormal{int}(S)}\prob(|\Gcal_N(x,r_N(K))|\leq k_N-1) \leq N^{-\alpha}
\end{equation*}
\end{lemma}

\begin{proof}  

Fix an $R>0$ and let $f(x) = \lambda(S\cap B(x,R))$ where $\lambda$ denotes the lebesgue measure. Note that $f$ is a continuous function on $S$. Clearly $f(x)>0$ for all $x\in \textnormal{int}(S).$ Furthermore, because $S = \overline{\textnormal{int}(S)},$ we also have that $f(x)>0$ for all $x\in \partial S.$ Since $S$ is compact the minimum value of $f$ is achieved and is positive. Let $M:= \min_{x\in S} f(x).$

By convexity of $S$
$$ x + \theta \cdot ((S\cap B(x,R)) - x)\subset S\cap B(x,R\theta).$$

Hence,
$$\lambda(S\cap B(x,R\theta))\geq \lambda(x + \theta\cdot ((S\cap B(x,R)))-x)\geq M\theta^d$$

for any $0\leq \theta\leq 1.$ Since $r_N(K)\to 0$ for any given $K,$ we have that for any given $K$ eventually for all large $N,$
$$\lambda(S\cap B(x,r_N(K)))\geq MR^{-d}K\frac{\max((\log N)^2,k_N)}{N}.$$

Note that for any $x\in S$ and any $r>0,$ $\Gcal_N(x,r)$ is $\poi(N\cdot h(B(x,r)))$ where $h(A) := \int_A h(z)~dz$ for any Borel set $A.$ If $C>0$ is such that $h(x)>C$ for all $x\in S$ then,
$$N\cdot h(B(x,r_N(K)))\geq C\cdot \lambda(S\cap B(x,r_N(K)))\geq CMR^{-d}K \cdot \max((\log N)^2,k_N).$$

Hence, for any fixed $K$ eventually we have the bound
$$\sup_{x\in \textnormal{int}(S)}\prob(|\Gcal_N(x,r_N(K))|\leq k_N-1) \leq \prob(\poi(CMR^{-d}K \cdot \max((\log N)^2,k_N))\leq k_N-1).$$

By taking $K$ large enough and using \Cref{lemma::poisson_bound}, the proof is complete.

\end{proof}

We now present a result that is sometimes called the Palm theory of Poisson processes. It allows us to write the expectation of certain functionals of Poisson processes in a cleaner manner.

\begin{lemma}\textnormal{(\citet[Theorem 1.6]{penrose2003random})} \label{lemma::palm theory identity}
Let $h$ be a density function, $\lambda>0$ be a constant and let $\Gcal_{\lambda h}$ denote the Poisson process with intensity function $\lambda h.$ Let $s>0$ be an integer and $u(\Ycal,\Xcal)$ be a bounded function defined on pairs of finite sets such that $\Ycal \subset \Xcal$ which is $0$ if $|\Ycal|\neq s.$ Then,
\begin{equation}\label{eq::palm_theory_identity}
\E\left(\sum_{\Ycal\subset \Gcal_{\lambda h}} u(\Ycal,\Gcal_{\lambda h})\right) = \frac{\lambda^s}{s!} \int \E\left(u(\{z_1,...,z_s\},\{z_1,...,z_s\}\cup \Gcal_{\lambda h})\right) \prod_{i=1}^s h(z_i) ~ dz_i
\end{equation}
\end{lemma}

The results given so far mostly deal with the out-neighbors of a given point. We conclude this first appendix by proving a result on the moments of the in-degree of a point in the Poisson process. We start by recalling the definition of a cone in $\R^d.$\\

For $x\in \R^d$ non-zero and $\theta>0,$ the cone $\Ccal(x,\theta)$ at $x$ of angle $\theta$ is defined as
$$\Ccal(x,\theta):= \{0\}\cup \left\{y:\arccos\left(\frac{x^Ty}{\|x\|\cdot \|y\|}\right)\leq \theta \right\}$$

We now state a result from \citet{biau2015lectures}.

\begin{lemma}\textnormal{(\citet[Lemma 20.5]{biau2015lectures})}\label{lemma::biau_devroye_cones} For $x\in \R^d$ and $\theta\leq \pi/6.$ Let $\Ccal(x,\theta)$ denote the cone around $x$ of angle $\theta.$ Then for any $y_1,y_2\in \Ccal(x,\theta)$ with $\|y_1\|\leq \|y_2\|,$ we have $\|y_1-y_2\|\leq \|y_2\|.$
\end{lemma}

Note that given an angle $\theta$,  $\R^d$ can be covered by a finite number of cones of angle $\theta$. Hence, using the above Lemma, we can see that given a set of points $A$ in $\R^d,$ a point $x\in A$ can be one of the $k$-nearest neighbors of at most $C_dk$ other points in $A$ where $C_d$ is some constant that depends only on $d.$ In particular, the in-degree of the any point in the $k_N$-NN graph is bounded by $C_dk_N.$ We now build on this and find the limiting first and second moments of the in-degree of a point in the Poisson process.

\begin{lemma}\label{lemma::limiting_moment_of_indegree} Let $h$ be a density on $\R^d$ bounded below on its support. Let $\Gcal_N$ denote the Poisson process on $\R^d$ with intensity function $Nh.$ Let $z$ be a point in the support of $h$ and let $\Gcal_N^z$ denote the set $\Gcal_N\cup\{z\}.$ Let $d^\downarrow_N(z)$ denote the in-degree of $z$ in the graph $\Gscr_{k_N}(\Gcal_N^z).$ Then the following hold:
\begin{align}
&\frac{\E(d^\downarrow_N(z))}{k_N} \to 1, \label{eq::limiting_first_moment_of_indegree}\\
&\frac{\E(d^\downarrow_N(z))^2}{k_N^2} \to 1. \label{eq::limiting_second_moment_of_indegree}
\end{align}
\end{lemma}

\begin{proof} WLOG assume $z=0.$ We assume that $k_N\gg (\log N)^2.$ The other case can be dealt with similarly but the calculations are a little more tedious. We first prove \eqref{eq::limiting_first_moment_of_indegree}. The proof of \eqref{eq::limiting_second_moment_of_indegree} is similar.\\

By the Palm Theory identity \eqref{eq::palm_theory_identity}, we get
\begin{align*}
\frac{E(d^\downarrow_N(0))}{k_N} &= \frac{N}{k_N}\int h(x)\cdot \prob\left((x,0)\in \Gscr_{k_N}(\Gcal_N^{x,0})\right)~dx
\end{align*}

%

The idea of the proof is to show that the probability in the above integral is almost $1$ if $\|x\|<\left(\frac{k_N}{N\cdot h(0)V_d}\right)^\frac{1}{d}$ and close to $0$ otherwise.\\

For any Borel set $A,$ let $h(A) := \int_{A}h(z)~dz$. Define $r_N(K) = \left(K\frac{\max((\log N)^2,k_N)}{N}\right)^\frac{1}{d}.$ Since we have assumed that $k_N\gg (\log N)^2,$ we can take $r_N(K) = \left(K\cdot \frac{k_N}{N}\right)^\frac{1}{d}.$\\

We can see that
$$\prob((z,0)\in \Gscr_{k_N}(\Gcal_N^{z,0})) = \prob(|\Gcal_N(z,\|z\|)|\leq k_N-1).$$

By \Cref{lemma::nearest_neighbor_distance_bound}, we can pick a $K$ such that the above probability is at most $N^{-4}$ for all $z$ such that $\|z\|\geq r_N(K).$ Hence, it suffices to find the limit of
\begin{align*}
\Ical_N &= \frac{N}{k_N}\int_{B(0,r_N(K))} h(x)\cdot \prob\left((x,0)\in \Gscr_{k_N}(\Gcal_N^{x,0})\right)~dx
\end{align*}

Note that $\Ical_N$ can be written as
\begin{align*}
\Ical_N &= \frac{N}{k_N}\int_{B(0,r_N(K))} h(x)\cdot \prob\left(\poi\left(N\cdot h(B(x,\|x\|))\right)\leq k_N-1\right)~dx.
\end{align*}

By continuity of $h,$ there exists a sequence $\epsilon_N\to 0$ such that for all $x\in B(0,2r_N(K)),$
$$(1-\epsilon_N)h(0)\leq h(x)\leq (1+\epsilon_N)h(0).$$

Let $\epsilon>0$ be given. Then, for all $x$ with $\|x\|\leq \left(\frac{k_N(1-\epsilon)}{N\cdot h(0)V_d}\right)^\frac{1}{d},$
\begin{align*}
N h(B(x,\|x\|)) &\leq N (1+\epsilon_N)h(0)\text{vol}(B(x,\|x\|))\\
&= N(1+\epsilon_N)h(0)V_d \frac{k_N(1-\epsilon)}{N\cdot h(0)V_d}\\
&= k_N(1-\epsilon)(1+\epsilon_N).
\end{align*}

Hence, for all large enough $N,$ 
$$N h(B(x,\|x\|)) \leq (1-\frac{\epsilon}{2})k_N,$$

for all $x$ with $\|x\|\leq \left(\frac{k_N(1-\epsilon)}{N\cdot h(0)V_d}\right)^\frac{1}{d}.$ Using \Cref{lemma::poisson_bound} we get
$$\prob\left(\poi\left(N\cdot h(B(x,\|x\|))\right)\leq k_N-1\right)\geq 1-k_N^{-2},$$

for all $x$ as above. Hence,
\begin{align*}
    &\liminf \frac{N}{k_N}\int_{B(0,r_N(K))} h(x)\cdot \prob\left(\poi\left(N h(B(x,\|x\|))\right)\leq k_N-1\right)~dx\\
    &\geq \liminf \frac{N}{k_N}\int_{B\left(0,\left(\frac{k_N(1-\epsilon)}{N\cdot h(0)V_d}\right)\right)^\frac{1}{d}} h(x)\cdot \prob\left(\poi\left(Nh(B(x,\|x\|))\right)\leq k_N-1\right)~dx\\
    &\geq \liminf \frac{N}{k_N}h(0)(1-\epsilon_N)(1-k_N^{-2})\text{vol}\left(B\left(0,\left(\frac{k_N(1-\epsilon)}{N\cdot h(0)V_d}\right)^\frac{1}{d}\right)\right)\\
    &= (1-\epsilon).
\end{align*}

To show an upper bound on the limsup, we can proceed as before and use \Cref{lemma::poisson_bound} to show that for all $x$ with $\left(\frac{k_N(1+\epsilon)}{N\cdot h(0)V_d}\right)\leq \|x\|\leq r_N(K),$
$$\prob\left(\poi\left(N\cdot h(B(x,\|x\|))\right)\leq k_N-1\right)\leq k_N^{-2}.$$

Bounding the above probability by $1$ for $\|x\|\leq \left(\frac{k_N(1+\epsilon)}{N\cdot h(0)V_d}\right)^\frac{1}{d}$ and by $k_N^{-2}$ otherwise and splitting the integral accordingly, we get
\begin{align*}
    &\limsup \frac{N}{k_N}\int_{B(0,r_N(K))} h(x)\cdot \prob\left(\poi\left(N \cdot h(B(x,\|x\|))\right)\leq k_N-1\right)~dx\\
    &\leq \limsup \left\{\frac{N}{k_N} h(0)(1+\epsilon_N)V_d \frac{k_N(1+\epsilon)}{N\cdot h(0)V_d} + k_N^{-2}\frac{N}{k_N}O\left(\frac{k_N}{N}\right)\right\}\\
    &= \limsup \left\{(1+\epsilon)(1+\epsilon_N) + O\left(k_N^{-2}\right)\right\}\\
    &= (1+\epsilon).
\end{align*}

Hence, we get that
\begin{align*}
    (1-\epsilon) &\leq \liminf \frac{N}{k_N}\int_{B(0,r_N(K))} h(x)\cdot \prob\left(\poi\left(N \int_{B(x,\|x\|)}h(z)~dz\right)\leq k_N-1\right)~dx\\
    &\leq \limsup \frac{N}{k_N}\int_{B(0,r_N(K))} h(x)\cdot \prob\left(\poi\left(N \int_{B(x,\|x\|)}h(z)~dz\right)\leq k_N-1\right)~dx\\
    &\leq (1+\epsilon).
\end{align*}

Since $\epsilon>0$ was arbitrary, we get the limit of the above integral as $1$ and as argued before, this gives us that
$$\lim \frac{\E(d^\downarrow_N(0))}{k_N}=1.$$

This completes the proof of \eqref{eq::limiting_first_moment_of_indegree}. Finding the limiting value of the second moment is similar so we only provide a brief sketch. Note that
$$d^\downarrow_N(0)^2 = \sum_{z\in \Gcal_N}\indicator\{(x,0)\in \Gscr_{k_N}(\Gcal_N^0)\} + \sum_{x,y\in \Gcal_N}\indicator\{(y,0),(z,0)\in \E(\Gscr_{k_N}(\Gcal_N^z))\}.$$

Using \Cref{lemma::biau_devroye_cones} and the discussion following it, we see that there exists a constant $C_d$ that depends only on $d$ such that
$$\sum_{x\in \Gcal_N}\indicator\{(x,0)\in \Gscr_{k_N}(\Gcal_N^0)\} \leq C_d k_N.$$

Hence, it is enough to find
$$\lim_{N\to\infty }\frac{1}{k_N^2}\E\left(\sum_{x,y\in \Gcal_N}\indicator\{(x,0),(y,0)\in \E(\Gscr_{k_N}(\Gcal_N^0))\}\right).$$

By the Palm Theory identity \eqref{eq::palm_theory_identity},
\begin{align*}
&\lim_{N\to\infty }\frac{1}{k_N^2}\E\left(\sum_{y,z\in \Gcal_N}\indicator\{(y,0),(z,0)\in \E(\Gscr_{k_N}(\Gcal_N^z))\}\right)\\
&= \frac{N^2}{k_N^2}\int h(x)h(y)\prob((x,0),(y,0)\in E(\Gscr_{k_N}(\Gcal_N^{x,y,0})))~dx~dy.
\end{align*}

Using the same arguments as before, we can restrict the integral to $x,y\in B(0,r_N(K)).$ For $x,y$ in this region, the probability in the above integral can be written as
$$\prob((x,0),(y,0)\in E(\Gscr_{k_N}(\Gcal_N^{x,y,0}))) = \prob(Z_1+Z_3,Z_2+Z_3\leq k_N-1)$$

where $Z_1,Z_2,Z_3$ are independent with
\begin{align*}
Z_1 &\sim \poi\left(N\cdot h(B(x,\|x\|)\setminus B(y,\|y\|))\right),\\
Z_2 &\sim \poi\left(N\cdot h(B(y,\|y\|)\setminus B(x,\|x\|))\right),\\ 
Z_3 &\sim \poi\left(N\cdot h(B(x,\|x\|)\cap B(y,\|y\|))\right)
\end{align*}
In particular, we see that
\begin{align*}
    Z_1+Z_3 &\sim \poi\left(N\cdot h(B(x,\|x\|))\right),\\
    Z_2 + Z_3 &\sim \poi\left(N\cdot h(B(y,\|y\|))\right)
\end{align*}

Using the same arguments as before, we get that $\prob(Z_1+Z_3\leq k_N-1) \geq 1-k_N^{-2}$ for $\|x\|\leq \left(\frac{k_N(1-\epsilon)}{N\cdot h(0)V_d}\right)^\frac{1}{d}$ and $\prob(Z_1+Z_3\leq k_N-1) \leq k_N^{-2}$ for $\|x\|\geq \left(\frac{k_N(1+\epsilon)}{N\cdot h(0)V_d}\right)^\frac{1}{d}$ and we have similar bounds on $\prob(Z_2+Z_3\leq k_N-1)$ in terms of $\|y\|.$ The rest of the proof is almost identical to the proof of the limiting first moment.

\end{proof}



\section{Consistency and asymptotic distributions}

\label{appendix::consistency_and_clts}

The main results of this appendix prove the consistency of the 2 sample test and that the test statistic has an asymptotically normal distribution under general alternatives. Specifically, we will prove \Cref{propn::weak_limit_of_stat}, and \Cref{thm::clt_conditional_stat,thm::clt_general_stat} in this section. For this, we recall some of the notation defined previously.

We work in the Poissonized setting where we sample the set of points $\Zcal_N:=\{Z_1,...,Z_{L_N}\}$ from a Poisson process $\Zcal_N$ with intensity function $N\phi_N(x)$ where $\phi_N(x):= \frac{N_1}{N}f(x) + \frac{N_2}{N}g(x)$ and $N_1+N_2 = N.$ We assume that $\frac{N_1}{N} - p = o(N^{-\frac{1}{2}}), \frac{N_2}{N}-q = o(N^{-\frac{1}{2}})$ and denote $\phi(x):=pf(x) + qg(x)$. The number of points in the process $\Zcal_N$ is denoted by $L_N$ which means $L_N\sim \poi(N).$ For each point $z\in \Zcal_N,$ we assign the label $c_z$ to $z$ with
\begin{equation}
    c_z =
    \begin{cases}
    1 \text{ with probability }\frac{N_1 f(x)}{N_1f(x) + N_2 g(x)},\\
    \\
    2 \text{ with probability } \frac{N_2 g(x)}{N_1f_n(x) + N_2g(x)}.
    \end{cases}
\end{equation}

The labels are assigned to all points in $\Zcal_N$ independent of all others. For a given $K,$ the test statistic is defined as
\begin{equation}
T(\Gscr_K(\Zcal_N)) = \sum_{x,y\in\Zcal_N}\psi(c_x,c_y)\indicator\{(x,y)\in E(\Gscr_K(\Zcal_N))\},
\end{equation}

\noindent where $\psi(c_x,c_y) = \indicator\{c_x=1,c_y=2\}.$ We will be considering the case where $K=k_N\to \infty.$ Hence, the statistic in our case will be denoted by $T(\Gscr_{k_N}(\Zcal_N)).$\\

For any function $h:\R^d\times \R^d \to [0,1]$ and $z\in \R^d$  we define 
\begin{align}
\kappa_N(h,z) &= \frac{1}{k_N}\sum_{w\in \Zcal_N}h(z,w)\indicator\{(z,w)\in \Gscr_{k_N}\},\label{eqn::kappa}
\end{align}

When $h$ is clear from context, we will simply denote this by $\kappa_N(z)$. Note that $\kappa_N$ is bounded by $1$ since a point can have at most $k_N$ out-neighbors in the $k_N$-NN graph.

Given a function $\omega : \R^d \times \R^d \times \R^d \to \R,$ we define $\tau^\uparrow_N(\omega,z), \tau^\downarrow(\omega,z)$ as
\begin{align}
    \tau_N^\uparrow(\omega,z) &= \frac{1}{2(k_N)^2}\sum_{w_1\neq w_2\in \Zcal_N} \omega(z,w_1,w_2)\indicator\{(z,w_1),(z,w_2)\in E(\Gscr_{k_N})\},\label{eqn::outneighbors_sum},\\
    \tau_N^\downarrow(\omega,z) &= \frac{1}{2(k_N)^2}\sum_{w_1\neq w_2\in \Zcal_N} \omega(z,w_1,w_2)\indicator\{(z,w_1),(z,w_2)\in E(\Gscr_{k_N})\}\label{eqn::inneighbors_sum},\\
    \tau_N^+(\omega,z) &= \frac{1}{(k_N)^2}\sum_{w_1\neq w_2\in \Zcal_N} \omega(z,w_1,w_2)\indicator\{(z,w_1),(w_2,z)\in E(\Gscr_{k_N})\}\label{eqn::oneinoneout_sum}.
\end{align}

Each of these sums refer to one of the `stars' that can be formed at a point $z.$ The first sum is over the outgoing stars, the second is over the incoming stars and the third is over the stars that have one incoming edge and one outgoing edge. We will see later that these terms come up in calculating the conditional variance of the statistic.

\subsection{Technical results}

We begin our results with some technical lemmas which will enable us to find the limiting value of certain objects easily. 

\begin{lemma}\label{lemma::convergence_to_limiting_value} Let $h:\R^d\times \R^d \to [0,1]$ be a uniformly continuous function. For any given $z,$
\begin{equation}\label{eq::pointwise_L1_convergence}
\kappa_N(h,z)\overset{L^1}{\to} h(z,z).
\end{equation}
Consequently,
\begin{equation}\label{eq::pointwise_convergence_in_mean}
\lim_{N\to \infty}\E(\kappa_N(h,z)) \to h(z,z).
\end{equation}

Furthermore, as $N\to \infty$
\begin{equation}\label{eq::L2_convergence_of_pointwise_sums}
\frac{1}{N}\sum_{z\in \Zcal_N}\kappa_N(h,z)\overset{L^2}{\rightarrow} \int_{\R^d} h(z,z)\phi(z)~dz.
\end{equation}
\end{lemma}

\begin{proof} For $K>0,$ define $r_N(K) = \left(K\dfrac{\max(k_N,(\log N)^2)}{N}\right)^\frac{1}{d}.$\\
\noindent We will prove \eqref{eq::pointwise_L1_convergence} first. We start by writing
\begin{align*}
    & ~ \Ebb \frac{1}{k_N}\sum_{w\in \Zcal_N}|h(z,w)-h(z,z)|\indicator\{(z,w)\in E(\Gscr_{k_N})\}\\
    &= \Ebb \frac{1}{k_N}\sum_{w\in \Zcal_N}|h(z,w)-h(z,z)|\indicator\{(z,w)\in E(\Gscr_{k_N}), w\in B(z,r_N(K))\}\\
    &+ \Ebb \frac{1}{k_N}\sum_{w\in\Zcal_N}|h(z,w)-h(z,z)|\indicator\{(z,w)\in E(\Gscr_{k_N}), w\not\in B(z,r_N(K))\}\\
    &=: E_1 + E_2.
\end{align*}

\noindent The above is true for any fixed $K.$ Consider $E_1$ first.
\begin{align*}
    E_1 &= \Ebb \frac{1}{k_N}\sum_{w\in \Zcal_N}|h(z,w)-h(z,z)|\indicator\{(z,w)\in E(\Gscr_{k_N}), w\in B(z,r_N(K))\}\\
    &\leq \Ebb \frac{1}{k_N}\sum_{w\in \Zcal_N}|h(z,w)-h(z,z)|\indicator\{w\in B(z,r_N(K))\}\\
    &= \frac{N}{k_N}\int_{B(z,(r_N(K))} |h(z,w)-h(z,z)|\phi_N(w) ~ dw \quad \ldots \quad \text{by \Cref{lemma::palm theory identity}}\\
    &\leq \frac{N}{k_N}\left(C\cdot r_N(K)\right)\cdot \text{vol}\left(B(z,r_N(K))\right) \quad \ldots \quad \textnormal{by uniform continuity of $h$}\\
    &\to 0.
\end{align*}

Now we come to $E_2$. We use the fact that $0\leq h\leq 1.$ Then, 
\begin{align*}
    E_2 &= \Ebb \frac{1}{k_N}\sum_{w\in\Zcal_N}|h(z,w)-h(z,z)|\indicator\{(z,w)\in E(\Gscr_{k_N}), w\not\in B(z,r_N(K))\}\\
    &\leq \Ebb \frac{1}{k_N}\sum_{w\in \Zcal_N}\indicator\{(z,w)\in E(\Gscr_{k_N}), w\not\in B(z,r_N(K))\}\\
    &\leq N\cdot \prob(B(z,(r_N(K)) \text{ contains less than $k_N$ points})\\
    &\to 0 \quad \ldots \quad \text{by \Cref{lemma::nearest_neighbor_distance_bound}, for all large enough $K$.}
\end{align*}

\noindent Since $E_1,E_2\to 0,$ we have
$$\Ebb \frac{1}{k_N}\sum_{w\in \Zcal_N}|h(z,w)-h(z,z)|\indicator\{(z,w)\in E(\Gscr_{k_N})\} \to 0.$$

\noindent Hence,
\begin{align*}
&\kappa_N(h,z) \overset{L^1}{\to} h(z,z),\\
&\E(\kappa_N(h,z))\to h(z,z).
\end{align*}

We now come to the proof of \eqref{eq::L2_convergence_of_pointwise_sums} i.e. the $L^2$ convergence. Using the Palm Theory identity in \Cref{lemma::palm theory identity}, and by \eqref{eq::pointwise_convergence_in_mean} and the DCT we see that
$$\E\left(\frac{1}{N}\sum_{z\in \Zcal_N}\kappa_N(z)\right) = \int \E(\kappa_N(Z)) \phi_N(z)~dz  \to \int h(z,z)\phi(z)~dz.$$

Hence, to prove $L^2$ convergence, we only need to show
$$\E\left(\frac{1}{N}\sum_{z\in \Zcal_N}\kappa_N(h,z)\right)^2 \to \left(\int h(z,z)\phi(z)~dz\right)^2.$$

Using the Palm Theory identity, we get that
\begin{align}
    & \E\left(\frac{1}{N}\sum_{z\in \Zcal_N}\kappa_N(h,z)\right)^2\\
    &= \frac{1}{N}\int \phi_N(z) \E\kappa^2_N(z) ~ dz + \int \phi_N(z_1)\phi_N(z_2)\E(\kappa_N(z_1)\kappa_N(z_2))~dz.
\end{align}

Since $h\in [0,1],$ we have that $\kappa_N(h,z)\leq 1$. Hence, the first term tends to $0$ almost surely. We want to prove that for any $z_1,z_2,$
$$\E(\kappa_N(z_1)\kappa_N(z_2)) \to  h(z_1,z_1)h(z_2,z_2).$$

\noindent By the DCT, this will show that
\begin{align*}
\int \phi_N(z_1)\phi_N(z_2)\E(\kappa_N(z_1)\kappa_N(z_2))~dz &\to \int \phi(z_1)\phi(z_2)h(z_1,z_1)h(z_2,z_2)~dz_1~dz_2\\
&= \left(\int_{\R^d} h(z,z)\phi(z) ~ dz\right)^2,
\end{align*}

which in turn will show $L^2$ convergence.

Define $A_K(z_1,z_2)$ to be the region
$$A_K(z_1,z_2) := B(z_1,(r_N(K))\times B(z_2,(r_N(K))$$

Then, we can split up the product $\kappa_N(z_1)\kappa_N(z_2)$ as
\begin{align*}
&\kappa_N(z_1)\kappa_N(z_2)\\
&= \frac{1}{(k_N)^2}\sum_{(w_1,w_2)\in \Zcal_N^2} h(z_1,w_1)h(z_2,w_2)\indicator\{(z_1,w_1),(z_2,w_2)\in E(\Gscr_{k_N})\}\\
&= \frac{1}{(k_N)^2}\sum_{(w_1,w_2)\in \Zcal_N^2} h(z_1,w_1)h(z_2,w_2)\indicator\{(w_1,w_2)\in A_K(z_1,z_2),(z_1,w_1),(z_2,w_2)\in E(\Gscr_{k_N})\}\\
&+ \frac{1}{(k_N)^2}\sum_{(w_1,w_2)\in \Zcal_N^2} h(z_1,w_1)h(z_2,w_2)\indicator\{(w_1,w_2)\in A_K(z_1,z_2)^c,(z_1,w_1),(z_2,w_2)\in E(\Gscr_{k_N})\}
\end{align*}

We will call the two terms $T_1$ and $T_2$ respectively. We will first prove that $\E(T_2)\to 0$ if we choose a large enough $K.$\\

\noindent Since $h\in [0,1]$ we have that
\begin{align*}
    T_2 &\leq \frac{d^{\uparrow}(z_1)d^\uparrow_K(z_2)}{(k_N)^2} + \frac{d^{\uparrow}(z_2)d^\uparrow_K(z_1)}{(k_N)^2}
\end{align*}

\noindent where $d^\uparrow(z)$ denote the total number of out-neighbors of $z$ and $d^\uparrow_K(z)$ denotes the number of out-neighbors of $z$ that lie outside the ball of radius $\left(\frac{r_N(K)}{N}\right)^{\frac{1}{d}}$ around $z$. Note that $d^\uparrow(z)\leq k_N$ since $\Gscr_{k_N}$ is the $k_N$-NN graph. Also, by \Cref{lemma::nearest_neighbor_distance_bound}, for any large enough $K,$
$$\E\left(\frac{d^\uparrow_K(z)}{k_N}\right)\to 0.$$

Hence, $\E(T_2)\to 0$ for a large enough $K.$ Now we need to find the limiting value of $\E(T_1).$ We will show that
$$\E\left(\frac{1}{(k_N)^2}\sum_{(w_1,w_2\in \Zcal_N^2)}\Delta_h \indicator\{(w_1,w_2)\in A_K(z_1,z_2),(z_1,w_1),(z_2,w_2)\in E(\Gscr_{k_N})\}\right)\to 0.$$

where $\Delta_h := |h(z_1,w_1)h(z_2,w_2)-h(z_1,z_1)h(z_2,z_2)|$. Note that quantity inside the expectation can be bounded as follows
\begin{align*}
    &\E\left(\frac{1}{(k_N)^2}\sum_{(w_1,w_2\in \Zcal_N^2)}\Delta_h \indicator\{(w_1,w_2)\in A_K(z_1,z_2),(z_1,w_1),(z_2,w_2)\in E(\Gscr_{k_N})\}\right)\\
    &\leq \E\left(\frac{1}{(k_N)^2}\sum_{(w_1,w_2)\in \Zcal_N^2}|h(z_1,w_1)h(z_2,w_2)-h(z_1,z_1)h(z_2,z_2)|\indicator\{(w_1,w_2)\in A_K(z_1,z_2)\}\right)\\
    &\leq \frac{N^2}{(k_N)^2}\int_{A_K(z_1,z_2)}|h(z_1,w_1)h(z_2,w_2)-h(z_1,z_1)h(z_2,z_2)| ~ dw_1 ~ dw_2 \quad \ldots \quad \text{by \Cref{lemma::palm theory identity}}\\
    &\leq C \frac{N^2}{k_N^2}r_N(K)(\text{vol}(B(0,r_N(K)))^2 \quad \ldots \quad \text{by uniform continuity of $h$}\\
    &\to 0.
\end{align*}

Hence,
\begin{align*}
&\lim \E(\kappa_N(z_1)\kappa_N(z_2))\\
&= \lim \frac{1}{k_N^2}\E\left(\sum_{(w_1,w_2)\in \Zcal_N^2} h(z_1,z_1)h(z_2,z_2)\indicator\{(w_1,w_2)\in A_K(z_1,z_2);(z_1,w_1),(z_2,w_2)\in \E(\Gscr_{k_N})\}\right).
\end{align*}

By the same argument as used to prove that $T_2\to 0,$ we get
$$\lim \frac{1}{k_N^2}\E\left(\sum_{(w_1,w_2)\in \Zcal_N^2} \indicator\{(w_1,w_2)\in A_K(z_1,z_2);(z_1,w_1),(z_2,w_2)\in \E(\Gscr_{k_N})\}\right) = 1$$
Hence,

\begin{align*}
\lim \E(\kappa_N(z_1)\kappa_N(z_2))&= h(z_1,z_1)h(z_2,z_2).
\end{align*}

This completes the proof.

\end{proof}

\begin{lemma}\label{lemma::convergence_of_cross_terms_in_conditional_variance}

Let $\omega:\R^d \times \R^d\times \R^d \to [0,1]$ be a uniformly continuous function. Then,
\begin{align*}
\sum_{z\in \Zcal_N}\tau_N^\uparrow(\omega,z) &\overset{L^2}{\rightarrow} \frac{1}{2}\int \omega(z,z,z)\phi(z)~dz,\\
\vspace{2mm}
\sum_{z\in\Zcal_N}\tau_N^\downarrow(\omega,z) &\overset{L^2}{\rightarrow} \frac{1}{2}\int \omega(z,z,z)\phi(z)~dz,\\
\vspace{2mm}
\sum_{z\in\Zcal_N}\tau_N^+(\omega,z) &\overset{L^2}{\rightarrow} \frac{1}{2}\int \omega(z,z,z)\phi(z)~dz.
\end{align*}
\end{lemma}

\begin{proof}

We will only prove the second statement. The other two can be shown in similar ways. The proof is very similar to that of \Cref{lemma::convergence_to_limiting_value} where we first show $L^1$ convergence of $\tau_N^\downarrow(z)$ for a fixed $z$ and then analyze the second moment of the sum to get $L^2$ convergence. We first show $L^1$ convergence. This entails showing that for every $z\in \R^d,$
$$\E\left(\left|\tau_N(\omega,z)-\frac{1}{2}\omega(z,z,z)\right|\right)\to 0.$$

Note that
\begin{align*}
\frac{\omega(z,z,z)}{2k_N^2}\E\left(\sum_{w_1\neq w_2}\indicator\{(w_1,z),(w_2,z)\in E(\Gscr_{k_N})\}\right)
&= \frac{\omega(z,z,z)}{2k_N^2}\E(d^\downarrow_N(z)(d^\downarrow_N(z)-1)),
\end{align*}

where $d_N^\downarrow(z)$ denotes the in-degree of $z$ in $\Gscr_{k_N}(\Zcal_N).$ Using the limits established in \Cref{lemma::limiting_moment_of_indegree} for the moments of the in-degree of a point, we get that
$$\frac{\omega(z,z,z)}{2}\E\left(\frac{d^\downarrow(z)(d^\downarrow(z)-1)}{k_N^2}\right) \to \frac{\omega(z,z,z)}{2}.$$

Hence, to show $L^1$ convergence, it suffices to show
$$\E \frac{1}{k_N^2}\sum_{w_1\neq w_2\in \Zcal_N}|\omega(z,w_1,w_2)-\omega(z,z,z)|\indicator\{(z,w_1),(z,w_2)\in E(\Gscr_{k_N})\} \to 0.$$

For $K>0,$ define $r_N(K) =\left( K\cdot \frac{\max((\log N)^2,k_N)}{N}\right)^\frac{1}{d}.$ Define 
$$B^2(z) = B(z,r_N(K))\times B(z,r_N(K)).$$

We start by writing
\begin{align*}
    &\E \frac{1}{k_N^2}\sum_{w_1\neq w_2\in \Zcal_N}|\omega(z,w_1,w_2)-\omega(z,z,z)|\indicator\{(w_1,z),(w_2,z)\in E(\Gscr_{k_N})\}\\
    =&\E \frac{1}{k_N^2}\sum_{w_1\neq w_2\in \Zcal_N}|\omega(z,w_1,w_2)-\omega(z,z,z)|\indicator\{(w_1,w_2)\in B^2(z), (w_1,z),(w_2,z)\in E(\Gscr_{k_N})\}\\
    +&\E \frac{1}{k_N^2}\sum_{w_1\neq w_2\in \Zcal_N}|\omega(z,w_1,w_2)-\omega(z,z,z)|\indicator\{(w_1,w_2)\in B^2(z)^c,(w_1,z),(w_2,z)\in E(\Gscr_{k_N})\}.
\end{align*}

Call the last two terms $T_1$ and $T_2$ respectively. We will show that both tend to $0.$ By \Cref{lemma::palm theory identity}
$$T_1 = \frac{N^2}{k_N^2}\int_{B^2(z)}|\omega(z,w_1,w_2)-\omega(z,z,z)|~\phi_N(w_1)\phi_N(w_2)~dw_1~dw_2$$

This can be bounded as follows.
\begin{align*}
    &\frac{N^2}{k_N^2}\int_{B^2(z)}|\omega(z,w_1,w_2)-\omega(z,z,z)|~\phi_N(w_1)\phi_N(w_2)~dw_1~dw_2\\
    &\leq C \cdot r_N(K) \cdot \frac{N^2}{k_N^2}\cdot (\text{vol}(B(z,r_N(K))))^2\quad \ldots \quad \text{by uniform continuity of $\omega$}\\
    &\to 0.
\end{align*}

Now we have to show that $T_2$ tends to $0.$\\

Let $d^\downarrow(z), d^\downarrow(K,z)$ denote the number of in-neighbors of $z$ and the number of in-neighbors of $z$ that lie outside the ball of radius $r_N(K)$ around $z$ respectively.

Since $\omega$ is bounded above by $1,$ we can bound $T_2$ as follows.
\begin{align*}
    T_2 &\leq \E \frac{1}{k_N^2}\sum_{w_1\neq w_2}\indicator\{(w_1,w_2)\in B^2(z)^c, (w_1,z),(w_2,z)\in E(\Gscr_{k_N})\}\\
    &\leq \E \frac{2L_N}{k_N} \sum_{w\in \Zcal_N}\indicator\{w\in B(z), (w,z)\in E(\Gscr_{k_N})\}\\
    &\leq \frac{N^2}{k_N^2}\int_{B(z)^c} \prob(\text{ $B(w,r_N(K))$ contains less than $k_N$ points.}) \phi(w)~dw.
\end{align*}

Since the densities are bounded below, by \Cref{lemma::nearest_neighbor_distance_bound} the point-wise probability decays faster than $N^{-2}$ uniformly in $w$, for some large enough $K.$ This shows that $T_2\to 0$ and hence, for any given $z$
\begin{align*}
&\tau_N^\downarrow(z) \overset{L^1}{\to} \frac{1}{2}\omega(z,z,z),\\
&\E(\tau_N^\downarrow(z)) \to \frac{1}{2}\omega(z,z,z).
\end{align*}

Having established the limits in expectation, the proof of $L^2$ convergence is similar to the $L^2$ convergence in \Cref{lemma::convergence_to_limiting_value}. We provide a brief sketch here.

We show that for $z_1, z_2$ being two distinct points in $\Zcal_N,$
$$\E\left(\tau_N^\downarrow(\omega,z_1) \tau_N^\downarrow(\omega,z_2)\right) \approx \E\left(\tau_N^\downarrow(\omega,z_1)\right)\cdot \E\left(\tau_N^\downarrow(\omega,z_2)\right)\approx \frac{\omega(z_1,z_1,z_1)\omega(z_2,z_2,z_2)}{4}.$$

This is done once again, as in the proof of \Cref{lemma::convergence_to_limiting_value}, by looking at the nearest neighbors of $z_1,z_2$ that are close and bounding the probability of having a nearest neighbor that is far away.

Additionally, the discussion following \Cref{lemma::biau_devroye_cones} shows that the in-degree in the $k_N$-NN graph is bounded by $C_dk_N.$ Hence, $\tau_N^\downarrow$ is bounded when $\omega$ is bounded. This combined with the point wise convergence and the DCT gives $L^2$ convergence.

\end{proof}

\subsection{Consistency and limiting variance}

Using the convergence theorems of the previous section, we can now show that $T(\Gscr_{k_N}(\Zcal_N))$ converges in probability to $\delta(f,g)$ where 
$$\delta(f,g) = pq\int_{\R^d} \frac{f(x)g(x)}{pf(x)+qg(x)}~dx.$$

Since this is a divergence between probability distributions, we will have that the test is consistent.

To analyze the variance of $T(\Gscr_{k_N})$, we condition on  $\Fcal_N,$ the sigma algebra generated by $\Zcal_N.$ We will then deal with the variance of the conditional expectation and the conditional variance separately. This is the aim of \Cref{subsection::limit_of_conditional_variance} and \Cref{section::limiting_variance_of_conditional_expectation}.

After conditioning on $\Fcal_N,$ which is the sigma algebra containing information on the location of all points of $\Zcal_N,$ the only randomness is in the labels $c_z$ for $z\in \Zcal_N$. Since the labels are assigned with probability proportional to the density, the conditional expectation can be written down comfortably. We define $h_N(x,y)$ as
\begin{equation}\label{eq::defn_of_h_N}
h_N(x,y) = \frac{N_1N_2}{N^2}\frac{f(x)g(x)}{(\frac{N_1}{N}f(x)+\frac{N_2}{N}g(x))(\frac{N_1}{N}f(y)+\frac{N_2}{N}g(y))}.
\end{equation}

Then the conditional expectation can be written as
$$\E(T(\Gscr_{k_N})|\Fcal_N) = \sum_{x,y\in \Zcal_N}h_N(x,y)\indicator\{(x,y)\in E(\Gscr_{k_N})\}.$$

We are now ready to prove Proposition \Cref{propn::weak_limit_of_stat}.

\subsection{Proof of \texorpdfstring{\Cref{propn::weak_limit_of_stat}}{Proposition \textcolor{blue}{2.1}} }

Note that $h_N\in [0,1]$ where $h_N$ is defined in \eqref{eq::defn_of_h_N}. We also know that $f,g$ are uniformly continuous and that $\frac{N_1}{N}-p = o(N^{-\frac{1}{2}}), \frac{N_2}{N}-q = o(N^{-\frac{1}{2}})$. Hence, we get that
$$\frac{1}{Nk_N}\E(T(\Gscr_{k_N})| \Fcal_N)= \frac{1}{Nk_N}\sum_{x,y\in \Zcal_N}h(x,y)\indicator\{(x,y)\in E(\Gscr_{k_N})\} + o(1),$$

where
\begin{equation}\label{eq::defn_of_h}
h(x,y) = pq\frac{f(x)g(y)}{(pf(x)+qg(y))(pf(y)+qg(y))}.
\end{equation}

With $\kappa_N(h,z)$ as defined in \eqref{eqn::kappa}, we get that
$$\frac{1}{Nk_N}\E(T(\Gscr_{k_N})| \Fcal_N) = \frac{1}{N}\sum_{z\in \Zcal_N}\kappa_N(h,z) + o(1).$$

By \Cref{lemma::convergence_to_limiting_value}, we know the $L^2$ limit, of $\frac{1}{N}\sum_{z\in \Zcal_N}\kappa_N(h,z).$ From the definition of $h$ in \eqref{eq::defn_of_h}, we get
\begin{equation}\label{eqn::l2_limit_of_conditional_expectation}
\frac{1}{Nk_N}\E(T(\Gscr_{k_N})|\Fcal_N)\overset{p}{\to} pq\int_{\R^d}h(x,x)\phi(x)~dx = \delta(f,g).
\end{equation}

With this, we only need to show that
$$\frac{1}{Nk_N}(T(\Gscr_{k_N})-\E(T(\Gscr_{k_N})|\Fcal_N))\overset{p}{\to} 0,$$
in order to show consistency. This follows from \eqref{eq::limiting_conditional_variance} in the next section which shows
$$\frac{1}{Nk_N^2}\Var(T(\Gscr_{k_N})|\Fcal_N) \overset{L^2}{\to} \sigma^2,$$

for some $\sigma^2>0$. This proves \Cref{propn::weak_limit_of_stat}.

\subsection{Limit of the conditional variance}\label{subsection::limit_of_conditional_variance}

For $x,y\in \Zcal_N,$ let
\begin{equation}\label{eq::defn_of_V_xy}
    V_{x,y} := \psi(c_x,c_y)\indicator\{(x,y)\in E(\Gscr_{k_N})\}.
\end{equation}
Note that conditional on $\Fcal_N,$ $V_{x,y} \sim \ber(h_N(x,y)).$ Hence, conditional on $\Fcal_N$ the statistic $T(\Gscr_{k_N})$ is the sum of the Bernoullis $\{V_{x,y}\}_{(x,y)\in E(\Gscr_{k_N})}.$ $V_{x,y}$ and $V_{w,z}$ are conditionally independent if and only if the two edges $(x,y), (w,z)$ do not share an endpoint.

Using this, we see that the conditional variance can be written as
\begin{align*}
    \Var(T(\Gscr_{k_N})|\Fcal_N) &= \sum_{(x,y)\in E(\Gscr_{k_N})} \Var(V_{x,y}|\Fcal_N)\\
    &+ \sum_{(x,y)\in E(\Gscr_{k_N})}\quad \sum_{z\neq y:(x,z)\in E(\Gscr_{k_N})} \Cov(V_{x,y},V_{x,z}|\Fcal_N)\\
    &+ \sum_{(y,x)\in E(\Gscr_{k_N})}\quad \sum_{z\neq y:(z,x)\in E(\Gscr_{k_N})} \Cov(V_{y,x},V_{z,x}|\Fcal_N)\\
    &+ 2\sum_{(y,x)\in E(\Gscr_{k_N})}\quad \sum_{z:(x,z)\in E(\Gscr_{k_N})} \Cov(V_{y,x},V_{x,z}|\Fcal_N)\\
    &+ \sum_{(x,y),(y,x)\in E(\Gscr_{k_N})} \Cov(V_{x,y},V_{y,x}|\Fcal_N).
\end{align*}

We will first deal with the sum of the variances. Note that for any $(x,y)\in E(\Gscr_{k_N}),$ 
$$\Var(V_{x,y}|\Fcal_N) = h_N(x,y)(1-h_N(x,y)) = h(x,y)(1-h(x,y)) + o(N^{-\frac{1}{2}}).$$
Hence, using the same ideas as in \Cref{lemma::convergence_to_limiting_value}, we get
$$\frac{1}{Nk_N}\sum_{(x,y)}\Var(V_{x,y}|\Fcal_N)\overset{L^2}{\to} \int h(x,x)(1-h(x,x))\phi(x)~dx.$$

This gives us that
$$\frac{1}{Nk_N^2}\sum_{(x,y)}\Var(V_{x,y}|\Fcal_N)\overset{L^2}{\to} 0.$$

We now come to each of the sums of the covariances. Define the functions $\omega_N^\uparrow, \omega_N^\downarrow, \omega^\uparrow$ and $\omega^\downarrow$ as
\begin{align*}
    \omega_N^\uparrow(x,y,z) &= \frac{N_1N_2^2}{N^3} \frac{f(x)g(y)g(z)}{(\frac{N_1}{N}f(x)+\frac{N_2}{N}g(x))(\frac{N_1}{N}f(y)+\frac{N_2}{N}g(y))(\frac{N_1}{N}f(z)+\frac{N_2}{N}g(z))}\\
    \vspace{5mm}\\
    \omega^\uparrow(x,y,z) &= \frac{pq^2f(x)g(y)g(z)}{(pf(x)+qg(x))(pf(y)+qg(y))(pf(z)+qg(z))}\\
    \vspace{5mm}\\
    \omega_N^\downarrow(x,y,z) &= \frac{N_1^2N_2}{N^3} \frac{g(x)f(y)f(z)}{(\frac{N_1}{N}f(x)+\frac{N_2}{N}g(x))(\frac{N_1}{N}f(y)+\frac{N_2}{N}g(y))(\frac{N_1}{N}f(z)+\frac{N_2}{N}g(z))}\\
    \vspace{5mm}\\
    \omega^\downarrow(x,y,z) &= \frac{p^2qg(x)f(y)f(z)}{(pf(x)+qg(x))(pf(y)+qg(y))(pf(z)+qg(z))}.
\end{align*}

For pairs of edges of the form $(x,y),(x,z)$ with $y\neq z$, we have
$$\Cov(V_{x,y},V_{x,z}|\Fcal_N) = \omega_N^{\uparrow}(x,y,z) - h_N(x,y)h_N(x,z).$$

By uniform continuity of the densities, we get
\begin{align*}
&\frac{1}{Nk_N^2}\sum_{(x,y)\in E(\Gscr_{k_N})}\sum_{z\neq y:(x,z)\in E(\Gscr_{k_N})} \Cov(V_{x,y},V_{x,z}|\Fcal_N)\\
&= \frac{2}{Nk_N^2}\left(\sum_{z\in \Zcal_N}\tau_N^\uparrow(\omega^\uparrow,z) - \sum_{z\in \Zcal_N} \tau_N^\uparrow(h^\uparrow,z)\right) + o(1),
\end{align*}

where $h^\uparrow(x,y,z) = h(x,y)h(x,z).$ By \Cref{lemma::convergence_of_cross_terms_in_conditional_variance}, we get 
$$\frac{1}{Nk_N^2}\sum_{(x,y),(x,z)}\Cov(V_{x,y},V_{x,z}|\Fcal_N) \overset{L^2}{\to} \int(\omega^\uparrow(x,x,x)-h^2(x,x))\phi(x)~dx.$$

Similarly, we can write 
\begin{align*}
&\frac{1}{Nk_N^2}\sum_{(y,x)\in \E(\Gscr_{k_N})} \quad \sum_{z\neq y, (z,x)\in E(\Gscr_{k_N})}\Cov(V_{z,x},V_{y,x})\\
&= \frac{1}{Nk_N^2}\sum_{(y,x),(z,x),y\neq z}(\omega^\downarrow(x,y,z)-h(y,x)h(z,x)) + o(1)\\
&= \frac{2}{k_N^2}\left(\sum_{z\in\Zcal_N}\tau_N^\downarrow(z,\omega^\downarrow) - \sum_{z\in\Zcal_N}\tau_N^\downarrow(h^\downarrow,z)\right)\\
&\overset{L^2}{\to} \int (\omega^\downarrow(x,x,x)-h^2(x,x))\phi(x)~dx \quad \ldots \quad \text{by \Cref{lemma::convergence_of_cross_terms_in_conditional_variance}},
\end{align*}

where $h^\downarrow(x,y,z) = h(y,x)h(z,x).$ Hence,
\begin{equation*}
   \frac{1}{Nk_N^2}\sum_{(y,x)\in \E(\Gscr_{k_N})} \quad \sum_{z\neq y, (z,x)\in E(\Gscr_{k_N})}\Cov(V_{z,x},V_{y,x}) \overset{L^2}{\to} \int (\omega^\downarrow(x,x,x) - h^2(x,x))~\phi(x)~dx.
\end{equation*}

The third sum of covariances can be written as
\begin{align*}
\frac{2}{Nk_N^2}\sum_{(y,x),(x,z)\in E(\Gscr_{k_N})}\Cov(V_{y,x},V_{x,z}|\Fcal_N) &= -\frac{2}{Nk_N^2}\sum_{(y,x),(x,z)\in \E(\Gscr_{k_N})} h_N(y,x)h_N(x,z)\\ 
&= -\frac{2}{Nk_N^2}\sum_{(y,x),(x,z)\in E(\Gscr_{k_N})}h(y,x)h(x,z) + o(1)\\
&= -\frac{2}{N}\sum_{z\in\Zcal_N}\tau_N^+(h^+,z) + o(1)\\
&\overset{L^2}{\to} -2\int h^2(x,x)\phi(x)~dx \quad \ldots \quad \text{by \Cref{lemma::convergence_of_cross_terms_in_conditional_variance}},
\end{align*}

where $h^+(x,y,z) = h(y,x)h(x,z).$ Hence,
$$\frac{2}{Nk_N^2}\sum_{(y,x),(x,z)}\Cov(V_{y,x},V_{x,z}|\Fcal_N) \overset{L^2}{\to} -2\int h^2(x,x)\phi(x)~dx$$

Finally, coming to the fourth sum we see that
\begin{equation*}
    \frac{1}{Nk_N^2}\sum_{(x,y),(y,x)\in E(\Gscr_{k_N})} \Cov(V_{x,y},V_{y,x}|\Fcal_N) \leq \frac{|\Zcal_N|k_N}{Nk_N^2}\overset{L^2}{\to} 0.
\end{equation*}

Put together, we get
\begin{equation}\label{eq::limiting_conditional_variance}
\begin{split}
    \frac{1}{Nk_N^2}\Var(T(\Gscr_{k_N})|\Fcal_N) &\overset{L^2}{\to} \int\left(\omega^\uparrow(x,x,x) + \omega^\downarrow(x,x,x) - 4h^2(x,x)\right)\phi(x)~dx\\
    &= pq\bigintsss \frac{f(x)g(x)(pf(x)-qg(x))^2}{\phi(x)^3}~dx =: \sigma^2_\cond.
\end{split}
\end{equation}

This gives us the $L^2$ limit of the conditional variance.


\subsection{Proof of \texorpdfstring{\Cref{thm::clt_conditional_stat}}{Theorem 3.1}}

Recall that the conditional statistic $\Rcal_\cond$ is defined as
$$\Rcal_\cond(\Gscr_{k_N}(\Zcal_N)) = \frac{1}{\sqrt{N}k_N}(T(\Gscr_{k_N}(\Zcal_N)) - \E_{H_1}(T(\Gscr_{k_N}(\Zcal_N)| \Fcal_N)).$$

From \eqref{eq::limiting_conditional_variance} we get the limiting variance of $\Rcal_\cond.$ Hence, to complete the proof of \Cref{thm::clt_conditional_stat} we only need to show asymptotic normality of $\Rcal_\cond.$\\

Recall the definition $V_{x,y}$ from \eqref{eq::defn_of_V_xy}. As in the previous section, we can take
$$T(\Gscr_{k_N}) = \sum_{x,y\in \Zcal_N}V_{x,y}.$$

As noted before, $V_{x,y}$ is a Bernoulli random variable conditional on $\Fcal_N.$ Let $G(\Zcal_N)$ denote the dependency graph of $\{V_{x,y}\}_{x,y}$ conditional on $\Fcal_N.$ $V_{x,y}$ and $V_{w,z}$ are conditionally independent if and only if the edges $(x,y)$ and $(w,z)$ do not share an endpoint in $\Gscr_{k_N}.$ Hence, using \Cref{lemma::biau_devroye_cones} and the discussion following it, we see that 
$$1 + \textnormal{deg}(G(\Zcal_N)) \leq M k_N$$
for some deterministic constant $M>0$ where $\textnormal{deg}(G(\Zcal_N))$ denotes the maximum among the degrees of the vertices of $G(\Zcal_N).$ We are now in a position to use the following theorem on Stein's method based on dependency graphs.

\begin{theorem}(\textnormal{\citet[Theorem 3.6]{rosssteinsmethod}})\label{thm::dependency_graph_clt} Let $G$ be a graph and let $\{X_i\}_{i\in V}$ be a collection of random variables indexed by the vertices of a graph $G.$ Suppose $\E(X_i)=0,$ $\sigma^2 := \Var\left(\sum X_i\right)$, $W:= \frac{\sum X_i}{\sigma}$ and $D:=1+\max(\textnormal{deg}(G))$. If $Z\sim N(0,1)$ then
\begin{equation*}
    \textnormal{Wass}(W,Z)\leq \frac{6}{\sqrt{\pi}\sigma^2}\sqrt{D^3\sum \E|X_i|^4} + \frac{D^2}{\sigma^3}\sum \E |X_i|^3, 
\end{equation*}

where $\textnormal{Wass}(W,Z)$ denotes the Wasserstein distance.
\end{theorem}

Let
$$W_N := \frac{T(\Gscr_{k_N}) - \E(T(\Gscr_{k_N})|\Fcal_N)}{\Var(T(\Gscr_{k_N})|\Fcal_N)}$$

and $W_N|_{\Fcal_N}$ denote the distribution of $\Fcal_N$ conditional on $\Fcal_N.$ Using \Cref{thm::dependency_graph_clt} and the upper bound established on $1+\textnormal{deg}(G(\Zcal_N))$  we get
\begin{align*}
\textnormal{Wass}(W_N|_{\Fcal_N},Z) &\leq \frac{6M^\frac{3}{2}}{\sqrt{\pi}}\frac{k_N^\frac{3}{2}\sqrt{L_Nk_N}}{\Var(T(\Gscr_{k_N})|\Fcal_N)} + \frac{k_N^2L_Nk_N}{\Var(T(\Gscr_{k_N}))^\frac{3}{2}}\\
&\leq \frac{6M^\frac{3}{2}}{\sqrt{\pi}}\frac{\sqrt{L_N}}{N}\left(\frac{\Var(T(\Gscr_{k_N})|\Fcal_N)}{Nk_N^2}\right)^{-1} + \frac{L_N}{N^\frac{3}{2}}\left(\frac{\Var(T(\Gscr_{k_N}))^\frac{3}{2}}{Nk_N^2}\right)^{-\frac{3}{2}}
\end{align*}

where $L_N = |\Zcal_N|.$ Since $L_N\sim \poi(N),$ we get $\frac{L_N}{\sqrt{N}} = o(1)$. Furthermore, \eqref{eq::limiting_conditional_variance} gives
$$\frac{\Var(T(\Gscr_{k_N}))^\frac{3}{2}}{Nk_N^2}\overset{p}{\to} \sigma^2_\cond.$$

Hence, we can marginalize over $\Zcal_N$ to get asymptotic normality of $W_N$ and hence of $\Rcal_\cond.$ This completes the proof of \Cref{thm::clt_conditional_stat}.

\subsection{Limiting variance of the conditional expectation}\label{section::limiting_variance_of_conditional_expectation}

The conditional expectation is given by
$$\E(T(\Gscr_{k_N})|\Fcal_N) = \sum_{x,y\in \Zcal_N}h_N(x,y)\indicator\{(x,y)\in E(\Gscr_{k_N})\}.$$

By uniform continuity of $f,g$ and since $\frac{N_1}{N}-p = o(N^{-\frac{1}{2}}),$ this can be written as
$$\E(T(\Gscr_{k_N})|\Fcal_N) = \sum_{x,y\in \Zcal_N}h(x,y)\indicator\{(x,y)\in E(\Gscr_{k_N})\} + o_p(N^{\frac{1}{2}}k_N)$$

We will now find it's asymptotic variance. In order to do this, we require some notation.

For $x,y\in \Zcal_N,$ let $J_{x,y}:= h(x,y)\indicator\{(x,y)\in E(\Gscr_{k_N})\}.$ Let $w(x) := \frac{pf(x)}{pf(x)+qg(y)}$ and let $v(y)=\frac{qg(y)}{pf(y)+qg(y)}.$ For $x\in \R^d,$ $\Hcal\subset \R^d$ a finite set and $A\subset \R^d$ a Borel set, let 
$$\xi_\Hcal^x(A) := w(x)\sum_{y\in A\cap \Hcal^x}\indicator\{(x,y)\in E(\Gscr_{k_N}(\Hcal^x))\}$$

where $\Hcal^x$ denotes the set $\Hcal$ with the point $x$ added to it. $\xi_\Hcal^x$ defines a measure on $\R^d.$ Let $\mu_N$ denote the sum of these measures across the Poisson process. That is, 
$$\mu_N := \sum_{x\in \Zcal_N}\xi^x_{\Zcal_N}.$$

We can integrate the function $v$ with respect to the measures $\xi^x_{\Zcal_N}$ and $\mu_N$ to get the quantities of interest to us. Specifically, we have
\begin{align*}
    \langle v,\xi^x_{\Zcal_N}\rangle &= \sum_{y\in \Zcal_N^x} h(x,y)\indicator\{(x,y)\in E(\Gscr_{k_N}(\Zcal_N^x))\},\\
    \\
    \langle v,\mu_N \rangle &= \sum_{(x,y)\in \Zcal_N} h(x,y)\indicator\{(x,y)\in E(\Gscr_{k_N}(\Zcal_N))\}.
\end{align*}

Note that $\langle v,\mu_N \rangle$ gives us exactly the conditional expectation. Writing the conditional expectation in this form allows us to use \citet{penrose2007gaussianlimits}, Lemma 4.2 which gives that
$$\Var(\E(T(\Gscr_{k_N})|\Fcal_N)) = Na_N + N b_N,$$

where
\begin{align*}
    a_N &:= \int \E(\langle v,\xi^x_{\Zcal_N} \rangle^2)\phi_N(x)~dx,\\
    b_N &:= \int (\E(\langle v,\xi^x_{\Zcal_N^{x_N(z)}} \rangle \langle v,\xi^{x_N(z)}_{\Zcal_N^x} \rangle)\\
    &\quad \quad - \E(\langle v,\xi^x_{\Zcal_N^{x_N(z)}} \rangle)\E( \langle v,\xi^{x_N(z)}_{\Zcal_N^x} \rangle))\phi_N(x)\phi_N(x_N(z)) ~dx ~ dz,
\end{align*}

where $x_N(z) = x + N^{-\frac{1}{d}}z.$\\

From the $L^2$ convergence in \Cref{lemma::convergence_to_limiting_value} and the DCT, we get
$$\frac{a_N}{k_N^2} \to \int h(x,x)^2\phi(x)~dx.$$

We will now show that
$$\frac{b_N}{k_N^2}\to 0.$$

This will give the scale of the limiting variance of the conditional expectation. To show the second limit, notice that for any $x,z$
\begin{align*}
\langle v,\xi^{x_N(z)}_{\Zcal_N^x} \rangle &= \sum_{y\in \Zcal_N^x}h(x_N(z),y)\indicator\{(x_N(z),y)\in E(\Gscr_{k_N}(\Zcal^x_{N}))\}\\
&= \sum_{y\in \Zcal_N^x}h(x,y)\indicator\{(x_N(z),y)\in E(\Gscr_{k_N}(\Zcal^x_{N}))\} + o(k_NN^{-\frac{1}{d}}).
\end{align*}

After writing it in this form, we can use arguments similar to \Cref{lemma::convergence_to_limiting_value}, to get that
$$\frac{\langle v,\xi^{x_N(z)}_{\Zcal_N^x} \rangle}{k_N} \overset{L^1}{\to} h(x,x).$$

Similarly, we also have
$$\frac{\langle v,\xi^{x}_{\Zcal_N^{x_N(z)}} \rangle}{k_N} \overset{L^1}{\to} h(x,x).$$

Hence, we have that the expectations converge which in turn gives us that
$$\E(\langle v,\xi^x_{\Zcal_N^{x_N(z)}} \rangle)\E( \langle v,\xi^{x_N(z)}_{\Zcal_N^x} \rangle))\to h(x,x)^2.$$

We now come to the term $\E(\langle v,\xi^x_{\Zcal_N^{x_N(z)}} \rangle \langle v,\xi^{x_N(z)}_{\Zcal_N^x} \rangle).$ We wish to show that
$$\frac{\E(\langle v,\xi^x_{\Zcal_N^{x_N(z)}} \rangle \langle v,\xi^{x_N(z)}_{\Zcal_N^x} \rangle)}{k_N^2}\to h(x,x)^2.$$

Note that both inner products inside the expectation are bounded above by $k_N.$ This is because they are both sums of at most $k_N$ many summands each of which lies in $[0,1].$ Specifically, we have that
\begin{align*}
    0&\leq \frac{1}{k_N}\sum_{y\in \Zcal_N^{x_N(z)}}h(x,y)\indicator\{(x,y)\in E(\Gscr_{k_N}(\Zcal_N^{x_N(z)}))\} \leq 1,\\
    0&\leq \frac{1}{k_N}\sum_{y\in \Zcal_N^x} h(x_N(z),y)\indicator\{(x_N(z),y)\in E(\Gscr_{k_N}(\Zcal_N^{x}))\} \leq 1.
\end{align*}

As a result,
$$ 0\leq \frac{\langle v,\xi^x_{\Zcal_N^{x_N(z)}} \rangle}{k_N}, \frac{\langle v,\xi^{x_N(z)}_{\Zcal_N^x} \rangle}{k_N} \leq 1.$$

Using the individual $L^1$ convergence and the boundedness, we get that
$$\frac{\langle v,\xi^x_{\Zcal_N^{x_N(z)}} \rangle \langle v,\xi^{x_N(z)}_{\Zcal_N^x} \rangle}{k_N^2} \overset{L^1}{\to} h(x,x)^2.$$

This shows convergence to the same quantity point wise. Once again, using the boundedness and the DCT, we have that $\frac{b_N}{k_N^2}\to 0.$ Altogether, this gives us
\begin{equation}\label{eq::limiting_variance_of_conditional_expectation}
\begin{split}
    \frac{1}{Nk_N^2}\Var(\E(T(\Gscr_{k_N})|\Fcal_N)) &= \frac{a_N}{k_N^2} + \frac{b_N}{k_N^2} \\
    &\to \int h(x,x)^2\phi(x)~dx\\
    &= p^2q^2\bigintsss \frac{f(x)^2g(x)^2}{\phi(x)^3}~dx.
\end{split}
\end{equation}

\subsection{Proof of \texorpdfstring{\Cref{thm::clt_general_stat}}{Theorem 3.2}}



Recall that the statistic $\Rcal$ was defined as
$$\Rcal(\Gscr_{k_N}(\Zcal_N)) = \frac{1}{k_N\sqrt{N}}(T(\Gscr_{k_N}(\Zcal_N)) - \E_{H_1}(T(\Gscr_{k_N}(\Zcal_N)))).$$

From \eqref{eq::limiting_conditional_variance} and \eqref{eq::limiting_variance_of_conditional_expectation} we get the limiting variance of $\Rcal.$ To prove \Cref{thm::clt_general_stat} we need to show asymptotic normality. We will show asymptotic normality of a slightly truncated statistic $T'(\Gscr_{k_N})$ which we will now define.

Define $r_N(K) = \left(K \frac{\max((\log N)^2,k_N)}{N}\right)^\frac{1}{d}$. For a given $K,$ let $\{D(i,N,K)\}_{i=1}^{M(N,K)}$ be a partition of the support $S$ of $f,g$ into $M(N,K)$ boxes of side length $r_N(K).$ For $1\leq i\leq M(N,K),$ let $N(i)$ be the set of indices such that $\{D(m,N,K):m\in N(i)\}$ is the set of boxes that share a side with $D(i,N,K).$ For $1\leq i \leq M(N,K),$ define
$$X(i,N,K) = \sum_{x,y\in \Zcal_N}\psi(c_x,c_y)\indicator\{x\in D(i,N,K);(x,y)\in \Gscr_{k_N}; \|x-y\|\leq r_N(K) \}.$$

Thus, $X(i,N,K)$ is the number of edges in the graph $\Gscr_{k_N}$ such that the label of the tail is $1,$ the label of the head is $2,$ the tail lies in the box $D(i,N,K)$ and the head lies either in $D(i,N,K)$ or one of it's neighboring boxes. Define 
$$T'(\Gscr_{k_N}) = \sum_{i=1}^{M(N,K)} X(i,N,K).$$

We now bound $\|T(\Gscr_{k_N}) - T'(\Gscr_{k_N})\|_2.$ Note that
\begin{align*}
|T(\Gscr_{k_N}) - T'(\Gscr_{k_N})| &\leq k_N\sum_{x\in \Zcal_N} \indicator\{|\Zcal_N\cap B(x,r_N(K))|\leq k_N-1\}.
\end{align*}

From \Cref{lemma::nearest_neighbor_distance_bound}, and the above bound, we get by choosing a large enough $K,$
$$\|T(\Gscr_{k_N}) - T'(\Gscr_{k_N})\|_2\leq N^{-3},$$

and in particular,
$$\lim_{N\to \infty} \frac{\Var(T'(\Gscr_{k_N}))}{Nk_N^2} = \lim_{N\to \infty} \frac{\Var(T(\Gscr_{k_N}))}{Nk_N^2}.$$

Furthermore, to show asymptotic normality of $\Rcal(\Gscr_{k_N}(\Zcal_N))$, by Slutsky's theorem it suffices to show asymptotic normality of $T'(\Gscr_{k_N}).$ Note that for each $i,$ $|D(i,N,K)\cap \Zcal_N|$ is a Poisson random variable with some mean $d(i,N,K).$ Since $f,g$ are bounded above, we get that there exists some universal constant $C$ such that
$$\max_i d(i,N,K)\leq CK\max((\log N)^2,k_N).$$

Using this, we get that

\begin{equation}\label{eq::moment_bounds_for_truncated_statistic}
\begin{split}
\E|X(i,N,K)-\E(X(i,N,K))|^4 &\leq C(K)\cdot  k_N^4 (\max((\log N)^2,k_N))^4,\\
\E|X(i,N,K)-\E(X(i,N,K))|^3 &\leq C(K)\cdot  k_N^3 (\max((\log N)^2,k_N))^3,
\end{split}
\end{equation}

where $C(K)$ is some constant that depends only on $K.$ We are now in a position to use \Cref{thm::dependency_graph_clt}. If $G$ denotes the dependency graph of $\{X(i,N,K)\}_{i=1}^{M(N,K)}$ then, the max degree is bounded since the edge counts in two boxes that are not neighboring or do not share a common neighbor are independent. Hence,
$$D := 1+ \max_{v\in G}(\deg(v))\leq C_d$$

for some constant $C$ that depends on the dimension $d.$ Finally, we also see that
$$M(N,K)\leq C(K)\frac{N}{k_N}$$

for some constant $C(K)$ depending only on $K.$ Hence, the bound coming from \Cref{thm::dependency_graph_clt} gives
\begin{align*}
\textnormal{Wass}\left(\frac{T'(\Gscr_{k_N} - \E(T'(\Gscr_{k_N}))}{\sqrt{\Var(T'(\Gscr_{k_N}))}},Z\right) &\leq \dfrac{6C(K)\sqrt{C_d^3Nk_N^3(\max((\log N)^2,k_N)^4}}{\sqrt{\pi} Nk_N^2}\left(\frac{\Var(T'(\Gscr_{k_N}))}{Nk_N^2}\right)^{-1}\\
&+ \frac{C_d^2C(K)Nk_N^2(\max((\log N)^2,k_N))^3}{N^\frac{3}{2}k_N^3}\left(\frac{\Var(T'(\Gscr_{k_N}))}{Nk_N^2}\right)^{-\frac{3}{2}}.
\end{align*}

For $k_N = o(N^{\frac{1}{4}})$ the above bound goes to $0$ which proves asymptotic normality of $T'(\Gscr_{k_N}).$ As stated before, we get asymptotic normality of $\Rcal(\Gscr_{k_N}(\Zcal_N))$ which proves \Cref{thm::clt_general_stat}.


\section{Detection thresholds}

\label{appendix::detection_thresholds}

Recall that when considering local alternatives in a parametrized family $\{p_\theta\}_{\theta\in \Theta}$, the null hypothesis for the 2-sample test is given by
$$H_0 : f=g=p_\thetaone,$$

for some $\thetaone \in \Theta.$ The alternate hypothesis is given by
$$H_1 : f = p_\thetaone, g=p_{\theta_2},$$

where $\theta_2 = \thetaone + \epsilon_N$ for some $\epsilon_N \to 0.$\\

The CLT's proved in the previous section can be generalized to show that

$$\frac{N^{-\frac{1}{2}}}{k_N}(T(\Gscr_{k_N}) - \E_{H_1}(T(\Gscr_{k_N}))) \to N(0,\sigma^2_0),$$

when $H_1$ is as given above and $\sigma^2_0$ denotes the null variance. Hence, to find the limiting power, it suffices to analyze the difference of means i.e.
\begin{equation}\label{eq::diff_of_means}
\frac{N^{-\frac{1}{2}}}{k_N}(\E_{H_1}(T(\Gscr_{k_N}))-\E(T(\Gscr_{k_N}))),
\end{equation}

as $\epsilon_N\to 0.$ Broadly, we need to characterize the conditions under which limiting value of \ref{eq::diff_of_means} is $0,$ finite and infinity. This will give the limiting power of the test. This Appendix is dedicated to this purpose.

We first define the following notation.
\begin{align*}
    \phi_N^{\theta_1,\theta_2}(x) &= \frac{N_1}{N}p_{\theta_1}(x) + \frac{N_2}{N}p_{\theta_2}(x),\\
    h_N^{\thetaone,\theta_2}(x,y) &= \frac{N_1N_2p_{\theta_1}(x)p_{\theta_2}(y)}{(N_1p_{\theta_1}(x) + N_2p_{\theta_2}(x))(N_1p_{\theta_1}(y) + N_2p_{\theta_2}(y))},\\
    \rho_K^{\thetaone,\theta_2}(x,y) &= \prob((x,y)\in \E(\Gscr_K(\Pcal_N^{x,y}))).
\end{align*}

By the Palm Theory identity \ref{eq::palm_theory_identity}, we can write $\E_{H_1}(T(\Gscr_{k_N}(\Zcal_N)))$ as
\begin{align*}
\E_{H_1}(T({\Gscr_{k_N}})) &= N^2\int h_N^{\thetaone,\theta_N}(x,y)~\rho_{k_N}^{\thetaone,\theta_2}
(x,y)~\phi_N(x)\phi_N(y)~dx~dy\\
&= \frac{N_1N_2}{N^2} N^2 \int p_\thetaone(x)p_{\theta_2}(y)\rho_N^{\theta_1,\theta_2}(x,y)~dx~dy.
\end{align*}

Since $\Gscr_{k_N}$ is the $k_N$-NN graph, we can write $\rho_{k_N}^{\theta_1,\theta_2}(x,y)$ as
\begin{equation}\label{eq::expression_for_rho_kn}
\rho_{k_N}^{\thetaone,\theta_2}(x,y) = \prob(\poi(\lambda_N^{\thetaone,\theta_2}(x,y))\leq k_N-1) = \sum_{k=1}^{k_N-1} \frac{\lambda_N^{\thetaone,\theta_2}(x,y)^k}{k!}\exp\left(-\lambda_N^{\thetaone,\theta_2}(x,y)\right),
\end{equation}
where
\begin{equation}\label{eq::lambda_N_defn}
    \lambda_N^{\thetaone,\theta_2}(x,y) = N_1 \int_{B(x,\|x-y\|)} p_\thetaone(z)~dz + N_2\int_{B(x,\|x-y\|)} p_{\theta_2}(z)~dz.
\end{equation}

The expectation when the two densities are $p_\thetaone,p_{\theta_2}$ can be written as
\begin{align*}
\E_{H_1}(T_{\Gscr_{k_N}}) &= \frac{N_1N_2}{N^2} N^2\int p_\thetaone(x)p_{\theta_2}(y)~\rho_{k_N}^{\thetaone,\theta_2}
(x,y)~\phi_N(x)\phi_N(y)~dx~dy\\
&= Nk_N \frac{N_1N_2}{N^2} \mu_N(\theta_1,\theta_2),
\end{align*}
where 
\begin{equation}\label{eq::mu_N_defn}
\mu_N(\theta_1,\theta_2) = \frac{N}{k_N} \int_{\|x-y\|\leq r_N(K)}p_\thetaone(x)p_{\theta_2}(y) \rho_{k_N}^{\thetaone,\theta_2}(x,y)~dx~dy.
\end{equation}

In order to find the limit of \ref{eq::diff_of_means}, we can expand $\mu_N(\theta_1,\theta_2)$ for $\theta_2 = \theta_1 + \epsilon_N.$ Doing a Taylor expansion in the second variable gives us
\begin{equation}\label{eq::taylor_expansion_of_difference_of_means}
\begin{split}
&\frac{N^{-\frac{1}{2}}}{k_N}(\E_{H_1}(T(\Gscr_{k_N}))-\E(T(\Gscr_{k_N})))\\
&= \sqrt{N}\frac{N_1N_2}{N^2}\left(\mu_N(\theta_1,\theta_N)-\mu_N(\thetaone,\thetaone)\right)\\
&=\frac{N_1N_2}{N^2}\left(\sqrt{N} \epsilon_N^T \nabla_\thetaone \mu_N(\thetaone,\thetaone) + \frac{\sqrt{N}}{2} \epsilon_N^T (\Hnormal \mu_N(\thetaone,\thetaone)) \epsilon_N \right) + \Rcal_N.
\end{split}
\end{equation}

Here the gradient and Hessian are with respect to only the second argument of $\mu_N.$ The following lemmas give the limiting values of the gradient, Hessian and remainder term respectively. The remainder term can be written as
$$\Rcal_N = \frac{N_1N_2\sqrt{N}}{3!N^2}\sum_{1\leq i,j,k\leq p}(\epsilon_N)_{ijk}\frac{\partial^3\mu_N(\theta_1,\theta)}{\partial \theta_{ijk}}\bigmid_{\theta\in (\theta_1,\theta_2)}.$$

In the above expression, $(\epsilon_N)_{ijk}$ denotes the product of the $i,j,k$ components of $\epsilon_N.$ The same notation extends to the partial derivatives with respect to $\theta$ and $(\theta_1,\theta_2)$ denotes the segment in $\R^p$ joining $\theta_1$ and $\theta_2.$\\

The following three lemmas give the limiting values of the gradient and hessian terms as well as the required bounds on the remainder term.


\begin{lemma}\textnormal{(Limit of the gradient term)}\label{lemma::gradient_limit}
    For $\epsilon_N = hN^{-\frac{1}{2}} \left(\frac{N}{k_N}\right)^\frac{2}{d}$ and under the assumptions of \Cref{thm::both_tests_above_criticality_broad_regimes},
    $$\sqrt{N} \epsilon_N \nabla_\thetaone \mu_N(\thetaone,\thetaone) \to \frac{p}{2(d+2)V_d^{\frac{2}{d}}} \bigintssss h^T\nabla_{\theta_1}\left(\frac{\tr(\Hnormal_x p(x|\theta_1))}{p_\thetaone(x)}\right)p_\thetaone^{\frac{d-2}{d}}(x) ~ dx$$
\end{lemma}

\begin{lemma}\textnormal{(Limit of the Hessian term)}\label{lemma::hessian_limit}
    For $\epsilon_N = hN^{-\frac{1}{4}}$ and under the assumptions of \Cref{thm::both_tests_above_criticality_broad_regimes},,
    $$ \sqrt{N} \epsilon_N^T \Hnormal \mu_N(\thetaone,\thetaone) \epsilon_N \to  -2pq \cdot \E\left[\frac{h^T\nabla_{\theta_1}p_\thetaone(x)}{p_\thetaone(x)}\right]^2$$
\end{lemma}

\begin{lemma}\textnormal{(Controlling the remainder term)}\label{lemma::controlling_remainder_term} For $\epsilon_N = h \cdot u_N$ for some $h\in \R^p\setminus\{0\}$ and $u_N\to 0,$ we have
$$\Rcal_N = O\left(N^\frac{1}{2}u_N^{3}\right)$$

under the assumptions of \Cref{thm::both_tests_above_criticality_broad_regimes}.

\end{lemma}

Before proving the above results, we show how \Cref{thm::power_below_criticality,thm::both_tests_above_criticality_broad_regimes,thm::both_tests_lower_upper_thresholds} follow from them.



\subsection{Proof of \texorpdfstring{\Cref{thm::power_below_criticality,thm::both_tests_above_criticality_broad_regimes,thm::both_tests_lower_upper_thresholds}}{Theorems 4.1 to 4.3}}

From the discussions at the start of this Appendix, we see that the limiting power of 1- and 2-sided tests can be found simply by looking at the Taylor expansion in \eqref{eq::taylor_expansion_of_difference_of_means}. Recall that the 1-sided test rejects when the standardized statistic lies below $z_\alpha.$ Hence, the limiting power of the 1-sided test can be described as follows.
\begin{equation*}\label{eq::1-sided_test_cases}
\lim_{N\to \infty} \sqrt{N}(\mu_N(\theta_1,\theta_2)-\mu_N(\theta_1,\theta_1)) =
\begin{cases*}
&$-\infty$, \text{ the limiting power is 1},\\
&$\gamma \in \R$, \text{ the limiting power is } $\Phi\left(z_\alpha - \gamma\frac{pq}{\sigma_0}\right)$,\\
& $\infty$, \text{ the limiting power is 0}.
\end{cases*}
\end{equation*}

where $\Phi$ is the standard normal CDF. The 2-sided test rejects the null hypothesis when the absolute value of the standardized statistic is at least $z_{1-\alpha/2}.$ Hence, the limiting power of the 2-sided test can be described as follows.
\begin{align}\label{2-sided_test_cases}
&\lim_{N\to \infty} \left|\sqrt{N}\left(\mu_N(\theta_1,\theta_2) - \mu_N(\theta_1,\theta_1)\right)\right|\nonumber \\
&=
\begin{cases*}
& $-\infty$, \text{ the limiting power is 1},\\
& $\gamma \in \R,$ \text{ the limiting power is $\Phi\left(z_{\alpha/2} + \gamma\frac{pq}{\sigma_0}\right) +   \Phi\left(z_{\alpha/2} - \gamma \frac{pq}{\sigma_0}\right)$},\\
& $\infty$, \text{ the limiting power is 1}.
\end{cases*}
\end{align}

With the above, we can now use \Cref{lemma::gradient_limit,lemma::hessian_limit,lemma::controlling_remainder_term} to prove the statements of \Cref{thm::power_below_criticality,thm::both_tests_above_criticality_broad_regimes,thm::both_tests_lower_upper_thresholds} by looking at various cases. For ease of notation, we will denote $\Delta_N := \sqrt{N}(\mu_N(\thetaone,\theta_2)-\mu_N(\thetaone,\thetaone))$

\begin{enumerate}
    \item Suppose $d$ is such that $N^{-\frac{1}{4}}\ll N^{-\frac{1}{2}}\left(\frac{N}{k_N}\right)^\frac{2}{d}$. There are three cases depending on the rate at which $\|\epsilon_N\|$ converges to 0.
    \begin{enumerate}
        \item Suppose $\|\epsilon_N\|\ll N^{-\frac{1}{4}}.$ Then by \Cref{lemma::gradient_limit,lemma::hessian_limit}
        \begin{align*}
            &\sqrt{N}\epsilon_N^T \nabla_\thetaone\mu_N(\theta_1,\theta_1) \to 0,\\
            &\sqrt{N}\epsilon_N^T\Hnormal \mu_N(\thetaone,\thetaone) \epsilon_N \to 0.
        \end{align*}
        Furthermore, from \Cref{lemma::controlling_remainder_term} we get $\Rcal_N\to 0.$ Hence, $\Delta_N \to 0$ and from \eqref{eq::1-sided_test_cases} and \eqref{2-sided_test_cases} we get that the limiting power of both tests is $\alpha$.
        
        \item Suppose $\epsilon_N = h N^{-\frac{1}{4}}$ for some $h\in \R^p\setminus\{0\}.$ Then using \Cref{lemma::gradient_limit,lemma::hessian_limit,lemma::controlling_remainder_term} we have
        \begin{align*}
            \frac{N_1N_2}{N^2}\frac{\Delta_N}{\sigma_0} \to -a(h,\theta_1),
        \end{align*}
        where $a(h,\theta_1)$ is as defined in \eqref{eq::Hessian_term}. Hence, the limiting power of the 1- and 2-sided test is $\Phi(z_\alpha + a(h,\thetaone))$ and $\Phi(z_{\alpha/2}+a(h,\thetaone)) + \Phi(z_{\alpha/2}-a(h,\thetaone))$ respectively.
        \item If $\|\epsilon_N\|\gg N^{-\frac{1}{4}},$ then using \Cref{lemma::gradient_limit,lemma::hessian_limit,lemma::controlling_remainder_term} along with the fact that $N^{-\frac{1}{4}}\ll N^{-\frac{1}{2}}\left(\frac{N}{k_N}\right)^\frac{2}{d}$ gives
        \begin{align*}
        &\sqrt{N}\epsilon_N^T\Hnormal \mu_N(\thetaone,\thetaone) \epsilon_N \to -\infty,\\
        &\left|\sqrt{N}\epsilon_N^T\Hnormal \mu_N(\thetaone,\thetaone) \epsilon_N\right| \gg \left|\sqrt{N}\epsilon_N^T \nabla_\thetaone\mu_N(\theta_1,\theta_1)\right|, |\Rcal_N|.
        \end{align*}
        Hence, in this case $\Delta_N \to \infty$ and from \eqref{eq::1-sided_test_cases} and \eqref{2-sided_test_cases} we get that the limiting power for both tests is 1.
    \end{enumerate}
    \item Now we consider the case $N^{-\frac{1}{4}}\left(\frac{N}{k_N}\right)^\frac{2}{d} \to \beta$ for some $\beta>0.$ This case is almost identical to the first with a few minor differences.
    \begin{enumerate}
        \item If $\epsilon_N \ll N^{-\frac{1}{4}}$ then as in the previous case, we have by \Cref{lemma::gradient_limit,lemma::hessian_limit}
        \begin{align*}
            &\sqrt{N}\epsilon_N^T \nabla_\thetaone\mu_N(\theta_1,\theta_1) \to 0,\\
            &\sqrt{N}\epsilon_N^T\Hnormal \mu_N(\thetaone,\thetaone) \epsilon_N \to 0.
        \end{align*}
        We also have $\Rcal_N \to 0.$ Hence, the limiting power of both tests is $\alpha.$
        \item If $\epsilon_N = hN^{-\frac{1}{4}}$ for some non-zero $h\in \R^p,$ then as before, we have
        $$\frac{N_1N_2}{N^2}\frac{\Delta_N}{\sigma_0} \to a(h,\theta_1) + \beta\cdot b(h,\thetaone)=: \nu,$$
        where $a(h,\theta_1)$ is as defined in \eqref{eq::Hessian_term}. This gives the limiting power of the 1- and 2-sided tests as $\Phi(z_\alpha + \nu)$ and $\Phi(z_{\alpha/2}+\nu) + \Phi(z_{\alpha/2}-\nu)$ respectively.
        \item Finally, if $\|\epsilon_N\|\gg N^{-\frac{1}{4}}$ then as before we can show that the limiting power of both tests is 1.
    \end{enumerate}
    \item We now consider the case where $N^{-\frac{1}{2}}\left(\frac{N}{k_N}\right)^\frac{2}{d}\ll N^{-\frac{1}{4}}.$ This is the most involved case and will require us to resort to quite a few cases.
    \begin{enumerate}
        \item If $\|\epsilon_N\|\ll N^{-\frac{1}{2}}\left(\frac{N}{k_N}\right)^\frac{2}{d}$ then \Cref{lemma::gradient_limit,lemma::hessian_limit,lemma::controlling_remainder_term} give that
        \begin{align*}
            &\sqrt{N}\epsilon_N^T \nabla_\thetaone\mu_N(\theta_1,\theta_1) \to 0,\\
            &\sqrt{N}\epsilon_N^T\Hnormal \mu_N(\thetaone,\thetaone) \epsilon_N \to 0,\\
            & \Rcal_N \to 0.
        \end{align*}
        Hence, $\Delta_N\to 0$ and the limiting power of both tests is equal to $\alpha.$
        \item Suppose $\epsilon_N = hN^{-\frac{1}{2}}\left(\frac{N}{k_N}\right)^\frac{2}{d}.$ Then we have
        \begin{align*}
            &\frac{N_1N_2}{N^2\sigma_0}\sqrt{N}\epsilon_N^T \nabla_\thetaone\mu_N(\theta_1,\theta_1) \to b(h,\theta_1),\\
            &\Rcal_N , \sqrt{N}\epsilon_N^T\Hnormal \mu_N(\thetaone,\thetaone) \epsilon_N \to 0.
        \end{align*}
        which gives
        $$\dfrac{N_1N_2}{N^2}\dfrac{\Delta_N}{\sigma_0}\to b(h,\thetaone).$$
        Hence, the limiting power of the 1- and 2-sided test is $\Phi(z_\alpha + b(h,\thetaone))$ and $\Phi(z_{\alpha/2}+b(h,\thetaone)) + \Phi(z_{\alpha/2}-b(h,\thetaone))$ respectively.
        \item If $N^{-\frac{1}{2}}\left(\frac{N}{k_N}\right)^\frac{2}{d}\ll \|\epsilon_N\|\ll \left(\frac{N}{k_N}\right)^{-\frac{2}{d}}$ then \Cref{lemma::gradient_limit,lemma::hessian_limit,lemma::controlling_remainder_term} give us
        \begin{align*}
        &\sqrt{N}\epsilon_N^T \nabla_\thetaone\mu_N(\theta_1,\theta_1) \to
        \begin{cases*}
         \infty &\text{ if $b(h,\thetaone)>0,$}\\
         -\infty &\text{ if $b(h,\thetaone)<0$},
        \end{cases*}\\
        & \left|\sqrt{N}\epsilon_N^T \nabla_\thetaone\mu_N(\theta_1,\theta_1)\right| \gg \left|\sqrt{N}\epsilon_N^T\Hnormal \mu_N(\thetaone,\thetaone) \epsilon_N\right|, |\Rcal_N|.
        \end{align*}
        As a result,
        \begin{equation*}
        \Delta_N \to
        \begin{cases*}
        \infty & \text{ if $b(h,\thetaone)>0$},\\
        -\infty & \text{ if $b(h,\thetaone)<0$}.
        \end{cases*}
        \end{equation*}
        From the above and \eqref{eq::1-sided_test_cases} we get that the limiting power of the 1-sided test is $0$ if $b(h,\thetaone)>0$ and $1$ if $b(h,\thetaone)<0.$
        \item If $\epsilon_N = h\left(\frac{N}{k_N}\right)^{-\frac{2}{d}}$ then \Cref{lemma::gradient_limit,lemma::hessian_limit,lemma::controlling_remainder_term} gives us
        $$N^{-\frac{1}{2}}\left(\frac{N}{k_N}\right)^\frac{4}{d}\frac{N_1N_2}{N^2}\frac{\Delta_N}{\sigma_0} \to -a(h,\thetaone)+b(h,\thetaone).$$
        In particular,
        \begin{equation*}
        \Delta_N \to
        \begin{cases*}
        \infty & \text{ if $a(h,\thetaone)-b(h,\thetaone)<0$},\\
        -\infty & \text{ if $a(h,\thetaone)-b(h,\thetaone)>0$}.
        \end{cases*}
        \end{equation*}
        Hence, \eqref{eq::1-sided_test_cases} gives that the limiting power of the 1-sided test is 0 and 1 if $a(h,\thetaone)-b(h,\thetaone)$ is negative or positive respectively. On the other hand, since $|\Delta_N|\to \infty,$ \eqref{2-sided_test_cases} gives us that the limiting power of the 2-sided test is 1.
        \item Finally, if $\|\epsilon_N\|\gg \left(\frac{N}{k_N}\right)^{-\frac{2}{d}}$ then we get
        \begin{align*}
         &\sqrt{N}\epsilon_N^T\Hnormal \mu_N(\thetaone,\thetaone) \epsilon_N \to -\infty,\\
        &\left|\sqrt{N}\epsilon_N^T\Hnormal \mu_N(\thetaone,\thetaone) \epsilon_N\right| \gg \left|\sqrt{N}\epsilon_N^T \nabla_\thetaone\mu_N(\theta_1,\theta_1)\right|, |\Rcal_N|.   
        \end{align*}
        Hence, $\Delta_N \to -\infty$ and hence, the limiting power of both tests is 1.
    \end{enumerate}
\end{enumerate}

This proves \Cref{thm::power_below_criticality,thm::both_tests_above_criticality_broad_regimes,thm::both_tests_lower_upper_thresholds}. Only the proofs of \Cref{lemma::gradient_limit,lemma::hessian_limit,lemma::controlling_remainder_term} remain. The remainder of this appendix is dedicated to their proofs.


\subsection{Technical results}

In order to prove the results on the limiting values of the gradients and hessians, we will need some technical results.

\begin{lemma}\label{lemma::gamma_function_identity}
For any $K\in \N,$
\begin{align*}
    \sum_{k=0}^{K-1}\frac{\Gamma(k+\frac{2}{d}+1)}{\Gamma(k+1)} = \frac{d}{d+2}\frac{\Gamma(K + \frac{2}{d}+1)}{\Gamma(K)}. 
\end{align*}
\end{lemma}
\begin{proof}
We can prove this by induction. For $K=1,$
\begin{align*}
\sum_{k=0}^{K-1}\frac{\Gamma(k+\frac{2}{d}+1)}{\Gamma(k+1)} &= \frac{\Gamma(1+\frac{2}{d})}{\Gamma(1)}\\
&= \frac{d}{d+2}\frac{\Gamma(2+\frac{2}{d})}{\Gamma(1)}.
\end{align*}

The identity holds for $K=1.$ Suppose it holds for some $K\in \N.$ To show that it holds for $K+1$ we consider the summation for $K+1.$ By the induction hypothesis and the recurrence of the Gamma function,
\begin{align*}
\sum_{k=0}^K \frac{\Gamma(k+\frac{2}{d}+1)}{\Gamma(k+1)} &= \sum_{k=0}^{K-1} \frac{\Gamma(k+\frac{2}{d}+1)}{\Gamma(k+1)} + \frac{\Gamma(K+1+\frac{2}{d})}{\Gamma(K+1)}\\
&= \left(\frac{d}{d+2}+\frac{1}{K}\right) \frac{\Gamma(K+1+\frac{2}{d})}{\Gamma(K)}\\
&= \frac{d}{d+2} \frac{K+1+\frac{2}{d}}{K}\frac{\Gamma(K+1+\frac{2}{d})}{\Gamma(K)}\\
&= \frac{d}{d+2}\frac{\Gamma(K+2+\frac{2}{d})}{\Gamma(K+1)}.
\end{align*}

This shows the identity for any $K\in \N.$
\end{proof}


Recall the definition of $\lambda_N^{\thetaone,\theta_2}(x,y)$ and $\rho_{k_N}^{\thetaone,\theta_2}(x,y)$ from \eqref{eq::lambda_N_defn} and \eqref{eq::expression_for_rho_kn}. Note that if we fix $x,$ then these are functions of $\|x-y\|.$ For the remainder of this appendix, we will deal with the case $\thetaone=
\theta_2.$ Hence, for ease of notation, we will refer to these functions as $\lambda_N(x,y)$ and $\rho_N(x,y).$ More formally, for $y$ such that $\|x-y\|=r,$ we define
\begin{align}
    \lambda_N(x,r) := \lambda_N^{\thetaone,\thetaone}(x,y), \label{eq::simplified_lambda_N}\\
    \rho_N(x,r) := \rho_{k_N}^{\thetaone,\thetaone}(x,y)\label{eq::simplified_rho_N}.
\end{align}

Going ahead, we will often take $x$ to be fixed. In such cases, we will denote the functions above as $\lambda_N(r), \rho_N(r)$, ignoring the dependence on $x.$\\

We begin with some technical results that give more details on the relationship between the distance $r$ from $x$ and change in the value of the function $\lambda_N(x,r).$\\

Recall from previous sections the definition of $r_N(K):=\left(K \cdot \dfrac{\max\left((\log N)^2,k_N\right)}{N}\right)^\frac{1}{d}.$\\

Recall also from \Cref{assume::assumptions_on_parametrized_family} the parametrized family $\{p_\theta\}_\theta$ to be supported over a compact set $S$ and uniformly bounded above and below.

\begin{lemma}\label{lemma::integral_of_density_over_small_balls}

Let $x$ be in the support of $p_\thetaone.$ Define $g_N(u) = \lambda_N\left(x,\left(\frac{u}{N}\right)^\frac{1}{d}\right).$ Then, there exists a non-negative sequence $\epsilon_N\to 0$ such that
$$(1-\epsilon_N)p_\thetaone(x)V_d\leq g_N'(u)\leq (1+\epsilon_N)p_\thetaone(x)V_d,$$

for all $0\leq u\leq Nr_N(K)^d.$ Consequently, for all $u,v>0$ such that $u+v\leq Nr_N(K)^d,$
$$v(1-\epsilon_N)p_\thetaone(x)V_d\leq g_N(u+v)-g_N(u)\leq v(1+\epsilon_N)p_\thetaone(x)V_d.$$

The sequence $\epsilon_N$ does not depend on $x.$

\end{lemma}

\begin{proof}
We will first prove the bounds on $g_N'.$ The other bound follows from using the Fundamental Theorem of Calculus for $g_N.$\\

WLOG we assume $x=0.$ Using the chain rule and the definition of $g_N,$ we have
\begin{align*}
    g_N'(u) &= \frac{1}{d}\frac{u^{\frac{1}{d}-1}}{N^\frac{1}{d}} \lambda_N'\left(\left(\frac{u}{N}\right)^\frac{1}{d}\right)\\
    &= \frac{1}{d}\frac{u^{\frac{1}{d}-1}}{N^\frac{1}{d}} N \int_{\partial B\left(0,\left(\frac{u}{N}\right)^\frac{1}{d}\right)} p_\thetaone(z)~dz\\
    &= p_\thetaone(0)V_d + \frac{V_d\tr(\Hnormal_x p_\thetaone(0))}{2} \left(\frac{u}{N}\right)^\frac{2}{d} + O\left(\left(\frac{u}{N}\right)^\frac{3}{d}\right) \quad \ldots \quad \textnormal{by \Cref{lemma::ball_and_surface_integrals}.}
\end{align*}

The proof of the bounds on $g_N'$ is completed by noticing that $\frac{u}{N}\leq r_N(K)^d \to 0$ and due to the fact $p_\thetaone$ is bounded below with uniformly bounded second and third derivatives.

\end{proof}

The next lemma gives an expansion of the inverse of $g_N$ as defined above.

\begin{lemma}\label{lemma::change_of_variables}

Let $x \in \interior(S)$ For $u\leq N r_N(K)^d,$ let $g_N(u):= \lambda_N\left(x,\frac{u^\frac{1}{d}}{N^\frac{1}{d}}\right).$ Then for $v\leq g_N((Nr_N(K))^d)$ the following hold.
\begin{align}
    &g_N^{-1}(v) = \frac{v}{p_\thetaone(x)V_d} - \frac{ \tr(\Hnormal_x p_\thetaone(x))}{2(d+2)V_d^{1+\frac{2}{d}}p_\thetaone(x)^{2+\frac{2}{d}}} \frac{v^{1+\frac{2}{d}}}{N^\frac{2}{d}} + v\delta_N^{(1)}(v) \label{eq::gN_inverse_expansion}\\
    \vspace{2mm}
    &\frac{g_N^{-1}(v)^\frac{2}{d}}{N^\frac{2}{d}} = \frac{1}{p_\thetaone(0)^\frac{2}{d}V_d^\frac{2}{d}}\frac{v^\frac{2}{d}}{N^\frac{2}{d}} + \delta_N^{(2)}(v)\label{eq::g_N_inverse_power_expansion}\\
    \vspace{2mm}
    &\frac{1}{g_N'(g_N^{-1}(v))} = \frac{1} {p_\thetaone(0)V_d} \left(1- \frac{\tr(\Hnormal_x p_\thetaone(0))}{2dp_\thetaone(0)^{1+\frac{2}{d}}V_d^\frac{2}{d}} \frac{v^\frac{2}{d}}{N^\frac{2}{d}} + \delta_N^{(3)}(v)\right) \label{eq::g_N_derivative_reciprocal}
\end{align}
where $\delta_N^{(i)}(v)\leq C r_N(K)^3$ for for $i=1,2,3,$ and some constant $C$ that depends on $p_\thetaone$ but does not depend on $x.$
\end{lemma}

\begin{proof}
WLOG, we assume $x=0.$ We prove \eqref{eq::gN_inverse_expansion} first. The other two results follow as easy consequences. From \Cref{lemma::ball_and_surface_integrals}, we know 
\begin{align*}
    g_N(u) &= N \lambda_N\left(\left(\frac{u}{N}\right)^\frac{1}{d}\right)\\
    &= p_\thetaone(0)V_d u + \frac{\tr(\Hnormal_x p_\thetaone(0))V_d}{2(d+2)} \frac{u^{1+\frac{2}{d}}}{N^\frac{2}{d}} + O\left(\frac{u^{1+\frac{3}{d}}}{N^\frac{3}{d}}\right)\\
    &=: \alpha u + \beta \frac{u^{1+\frac{2}{d}}}{N^\frac{2}{d}} + O\left(\frac{u^{1+\frac{3}{d}}}{N^\frac{3}{d}}\right).
\end{align*}

Denote 

$$\Tilde{v} := \frac{v}{\alpha} - \frac{\beta}{\alpha^{2+\frac{2}{d}}}\frac{v^{1+\frac{2}{d}}}{N^\frac{2}{d}}.$$

Then
\begin{align*}
g_N(\Tilde{v}) &= v - \frac{\beta}{\alpha^{1+\frac{2}{d}}}\frac{v^{1+\frac{2}{d}}}{N^\frac{2}{d}} + \frac{\beta}{\alpha^{1+\frac{2}{d}}} \frac{v^{1+\frac{2}{d}}}{N^\frac{2}{d}}\left(1-\frac{\beta}{\alpha^{1+\frac{2}{d}}}\left(\frac{v}{N}\right)^\frac{2}{d}\right)^{1+\frac{2}{d}} + O\left(\frac{\Tilde{v}^{1+\frac{3}{d}}}{N^\frac{3}{d}}\right). 
\end{align*}

For $v\leq g_N(Nr_N(K)^d),$ we know that $\left(\frac{v}{N}\right)^\frac{2}{d}\leq r_N(K)^2 \to 0.$ Hence, the big-O term above can be replaced with $O\left(\frac{v^{1+\frac{3}{d}}}{N^\frac{3}{d}}\right).$ Furthermore, we also get
$$\left(1- \frac{\beta}{\alpha^{2+\frac{2}{d}}}\left(\frac{v}{N}\right)^\frac{2}{d}\right)^\frac{2}{d} = 1+ O\left(\left(\frac{v}{N}\right)^\frac{2}{d}\right).$$

Putting these together, we get that
$$g_N(\Tilde{v}) = v + O\left(\frac{v^{1+\frac{3}{d}}}{N^\frac{3}{d}}\right).$$

Let $C>0$ be a constant such that for all $v,$
$$|g_N(\Tilde{v})-v| \leq C\frac{v^{1+\frac{3}{d}}}{N^\frac{3}{d}}.$$

Finally, by \Cref{lemma::integral_of_density_over_small_balls}, we get that for large all $N$, the inequalities
$$g_N\left(\Tilde{v} - 2C\frac{v^{1+\frac{3}{d}}}{N^\frac{3}{d}}\right)<v<g_N\left(\Tilde{v} + 2C\frac{v^{1+\frac{3}{d}
}}{N^\frac{3}{d}}\right).$$

Since $g_N$ is monotonically increasing, this shows that
$$g_N^{-1}(v) = \Tilde{v} + v\cdot O\left(\left(\frac{v}{N}\right)^\frac{3}{d}\right).$$

Noticing that $\frac{v}{N}\leq r_N(K)^d$ completes the proof of \eqref{eq::gN_inverse_expansion}. To prove \eqref{eq::g_N_inverse_power_expansion}, we can simply use the fact that for $y\approx 0,$
$$(1+y)^\frac{2}{d} = 1 + O(y)$$

combined with \eqref{eq::gN_inverse_expansion}.

From \Cref{lemma::ball_and_surface_integrals} and \eqref{eq::g_N_inverse_power_expansion}, we get
\begin{align*}
g_N'(g_N^{-1}(v)) &= V_dp_\thetaone(0) + \frac{V_d\tr(\Hnormal_x p_\thetaone(0))}{2d} \left(\frac{g_N^{-1}(v)}{N}\right)^\frac{2}{d} + O\left(\left(\frac{g_N^{-1}(v)}{N}\right)^\frac{3}{d}\right)\\
&= V_d p_\thetaone(0)\left( 1 + \frac{\tr(\Hnormal_x p_\thetaone(0))}{2dp_\thetaone(0)^{1+\frac{2}{d}}V_d^\frac{2}{d}}\frac{v^\frac{2}{d}}{N^\frac{2}{d}} + O\left(\frac{v^\frac{3}{d}}{N^\frac{3}{d}}\right)\right)
\end{align*}

Using the above expression along with that the fact that for $y\approx 0$
$$\frac{1}{(1+y)} = 1- y + O(y^2)$$
we get \eqref{eq::g_N_derivative_reciprocal}. The completes the proof.

\end{proof}

To find the limiting values of the gradient and hessian of $\mu_N(\thetaone,\thetaone),$ we will differentiate under the integral sign to write the derivatives as double integrals. Using the DCT, one can show that it is enough to find the point-wise limit of the single integrals. The next three lemmas enable us to find these point-wise limits.

\begin{lemma}\label{lemma::lemma_for_T1}

Let $f$ defined on $S$ be three times differentiable and bounded on $S$. Then, for any given $x\in \interior(S),$ we have
\begin{align*}
\frac{N}{k_N}\bigintsss f(y) \rho_{k_N}^{\thetaone,\thetaone}(x,y)~dy &= \frac{f(x)}{p_{\theta_1}(x)}\\
&+ \left(\frac{k_N}{N}\right)^{\frac{2}{d}}\left(\frac{\tr(H_x f(x))}{2dV_d^{\frac{2}{d}}(p_{\theta_1}(x))^{1+\frac{2}{d}}}\right)\left(\frac{1}{k_N^{1+\frac{2}{d}}}\sum_{k=0}^{k_N-1}\frac{\Gamma(k+\frac{2}{d}+1)}{\Gamma(k+1)}\right)\\
&-\left(\frac{k_N}{N}\right)^{\frac{2}{d}}\left(\frac{f(x)\tr(H_x p_{\theta_1}(x))}{2dV_d^\frac{2}{d}p_{\theta_1}(x)^{2+\frac{2}{d}}}\right)\left(\frac{1}{k_N^{1+\frac{2}{d}}}\sum_{k=0}^{k_N-1}\frac{\Gamma(k+1+\frac{2}{d})}{\Gamma(k+1)}\right)\\
&+ O\left(r_N(K)^3\right)
\end{align*}
\end{lemma}

\begin{proof}
WLOG we assume that $x=0.$ From \Cref{lemma::nearest_neighbor_distance_bound}, and the definition of $\rho_{K_N}^{\thetaone,\thetaone}(0,y),$ it is enough to show the convergence when the integral is over $y$ with $\|y\|\leq r_N(K)$ for a large enough $K.$\\

Recall the definition of $\lambda_N,\rho_N$ from \eqref{eq::simplified_lambda_N} and \eqref{eq::simplified_rho_N}. Since $x=0$ is fixed, for this proof we will denote them by $\lambda_N(r),\rho_N(r)$. Changing to spherical coordinates and denoting the radius as $r,$ we get that
$$\frac{N}{k_N}\bigintsss_{\|y\|\leq r_N(K)} f(y) \rho_{k_N}^{\thetaone,\thetaone}(x,y)~dy = \frac{N}{k_N} \int_0^{r_N(K)} \left(\int_{\partial B(0,r)} f(z)~dz\right) \rho_N(r) ~ dr.$$

We now make the change of variables $r = \left(\frac{u}{N}\right)^\frac{1}{d}.$ With this, we get that
\begin{align*}
&\frac{N}{k_N}\bigintsss_{\|x-y\|\leq r_N(K)} f(y) \rho_{k_N}^{\thetaone,\thetaone}(x,y)~dy\\
&= \frac{N^{1-\frac{1}{d}}}{k_N d}\int_0^{(Nr_N(K))^d} u^{\frac{1}{d}-1}\left(\int_{\partial B(0,\left(\frac{u}{N}\right)^\frac{1}{d})} f(z)~dz\right) \rho_N\left(\left(\frac{u}{N}\right)^\frac{1}{d}\right) ~ du.
\end{align*}

From \Cref{lemma::ball_and_surface_integrals}, we know that
$$N^{1-\frac{1}{d}}u^{\frac{1}{d}-1}\int_{\partial B\left(0,\left(\frac{u}{N}\right)^\frac{1}{d}\right)} f(z)~dz = d V_d f(0) + \frac{V_d \tr(\Hnormal f(0))}{2} \frac{u^\frac{2}{d}}{N^\frac{2}{d}} + \delta_N (u)$$ 
such that
$$ \delta_N(u) = O\left(\frac{u^\frac{3}{d}}{N^\frac{3}{d}}\right).$$

Since $\rho_N$ is a probability, it is non-negative and bounded above by 1 which gives
$$\frac{1}{k_N d} \int_0^{Nr_N(K)^d} \delta_N(u) \rho_N\left(\frac{u^\frac{1}{d}}{N^\frac{1}{d}}\right) ~ du \leq \frac{1}{k_Nd} \int_0^{Nr_N(K)^d} \delta_N(u) ~ du = O\left(r_N(K)^3\right).$$

We now deal with the other terms. Let
$$g_N(u) = \lambda_N\left(\frac{u^\frac{1}{d}}{N^\frac{1}{d}}\right)$$

for $u\leq N r_N(K)^d.$ Note that $g_N$ is a strictly increasing function for the given range of $u$ since it is the integral of a density. If we now make the change of variables $v = g_N(u),$ then we get that
\begin{align*}
    &\frac{1}{k_N d}\int_0^{(Nr_N(K))^d} \left(dV_d f(0) + \frac{V_d \tr(\Hnormal_xf(0))}{2}\frac{u^\frac{2}{d}}{N^\frac{2}{d}}\right)\rho_N\left(\frac{u^\frac{1}{d}}{N^\frac{1}{d}}\right)~du\\
    &= \frac{1}{k_N d}\int_0^{g_N((Nr_N(K))^d)} \frac{1}{g_N'(g_N^{-1}(v))} \left(dV_d f(0) + \frac{V_d \tr(\Hnormal_xf(0))}{2}\frac{(g_N^{-1}(v))^\frac{2}{d}}{N^\frac{2}{d}}\right) \sum_{k=0}^{k_N-1} \frac{v^k}{k!}e^{-v} ~ dv.
\end{align*}

Using \Cref{lemma::change_of_variables}, we know that for $0\leq v\leq g_N((Nr_N(K))^d),$
\begin{align*}
\frac{g_N^{-1}(v)^\frac{2}{d}}{N^\frac{2}{d}} &= \frac{1}{p_\thetaone(0)^\frac{2}{d}V_d^\frac{2}{d}}\frac{v^\frac{2}{d}}{N^\frac{2}{d}} + O\left(\frac{k_N^\frac{4}{d}}{N^\frac{4}{d}}\right),\\
\frac{1}{g_N'(g_N^{-1}(v))} &= \frac{1} {p_\thetaone(0)V_d} \left(1- \frac{\tr(\Hnormal_x p_\thetaone(0))}{2dp_\thetaone(0)^{1+\frac{2}{d}}V_d^\frac{2}{d}} \frac{v^\frac{2}{d}}{N^\frac{2}{d}} + O\left(\frac{k_N^\frac{3}{d}}{N^\frac{3}{d}}\right)\right)
\end{align*}

Finally note that by taking $K$ large enough, we can show that
\begin{equation*}
\prob(\Gamma(k,1)\geq g_N(Nr_N(K))^d),\prob(\Gamma(k+2/d,1)\geq g_N(Nr_N(K))^d)\leq N^{-M},
\end{equation*}

for any given $M>0$ and all $1\leq k\leq k_N.$ Hence, we can rewrite the integral as

\begin{align*}
    &\frac{1}{k_N d}\int_0^{g_N((Nr_N(K))^d)} \frac{1}{g_N'(g_N^{-1}(v))} \left(dV_d f(0) + \frac{V_d \tr(\Hnormal_xf(0))}{2}\frac{(g_N^{-1}(v))^\frac{2}{d}}{N^\frac{2}{d}}\right) \sum_{k=0}^{k_N-1} \frac{v^k}{k!}e^{-v} ~ dv\\
    &= \frac{1}{k_N} \int_0^\infty \left(\frac{f(0)}{p_\thetaone(0)} - \frac{f(0)\tr(\Hnormal_x p_\thetaone(0))}{2dV_d^\frac{2}{d}p_\thetaone(0)^{2+\frac{2}{d}}} \frac{v^\frac{2}{d}}{N^\frac{2}{d}} + \frac{\tr(\Hnormal_x f(0))}{2dV_d^\frac{2}{d}p_\thetaone(0)^{1+\frac{2}{d}}} \frac{v^\frac{2}{d}}{N^\frac{2}{d}}\right)\sum_{k=0}^{k_N-1}\frac{v^k}{k!}e^{-v} ~ dv + O(r_N(K)^3)\\
    &= \frac{f(0)}{p_\thetaone(0)} + \frac{k_N^\frac{2}{d}}{N^\frac{2}{d}} \left(\frac{\tr(\Hnormal_x f(0))}{2dV_d^\frac{2}{d}p_\thetaone(0)^{1+\frac{2}{d}}} - \frac{f(0)\tr(\Hnormal_x p_\thetaone(0))}{2dV_dp_\thetaone(0)^{2+\frac{2}{d}}}\right)\frac{1}{k_N^{1+\frac{2}{d}}}\sum_{k=0}^{k_N-1}\frac{\Gamma(k+1+\frac{2}{d})}{\Gamma(k+1)} + O(r_N(K)^3).
\end{align*}

This completes the proof.

\end{proof}


The next two lemmas are also proven in a very similar way.

\begin{lemma}\label{lemma::lemma_for_T2}

Let $f,g$ be bounded functions defined on $S,$ a bounded open set. Then for any $x\in \interior(S)$, 

\begin{align*}
    &\frac{N^2}{k_N}\bigintsss f(y)\left(\int_{B(x,\|x-y\|)} g(z) ~ dz\right) \frac{\left(\lambda_N^{\thetaone,\thetaone}(x,y)\right)}{(k_N-1)!}^{k_N-1}\exp(-\lambda_N^{\thetaone,\thetaone}(x,y)) ~dy\\
    &= \frac{f(x)g(x)}{p_\thetaone(x)^2}\\
    &+ \left(\frac{k_N}{N}\right)^\frac{2}{d}\frac{\Gamma(k_N+1+\frac{2}{d})}{k_N^\frac{2}{d}\Gamma(k_N+1)}\left(\frac{\tr(\Hnormal_xf(x))g(x)}{2dV_d^\frac{2}{d}p_\thetaone(x)^{2+\frac{2}{d}}} + \frac{f(x)(\tr(\Hnormal_xg(x)))}{2(d+2)V_d^\frac{2}{d}p_\thetaone(x)^{2+\frac{2}{d}}}\right)\\
    &-\left(\frac{k_N}{N}\right)^\frac{2}{d}\frac{\Gamma(k_N+1+\frac{2}{d})}{k_N^\frac{2}{d}\Gamma(k_N+1)}\frac{d+1}{d(d+2)}\frac{f(x)g(x)\tr(\Hnormal_x p_\thetaone(x))}{V_d^\frac{2}{d}p_\thetaone(x)^{3+\frac{2}{d}}}\\
    &+O\left(r_N(K)^3\right).
\end{align*}
\end{lemma}

\begin{proof}

Once again, WLOG we assume that $x=0.$ As before, we use \Cref{lemma::nearest_neighbor_distance_bound} to see that it is enough to show the convergence when the integral is over $y$ with $\|y\|\leq r_N(K)$ for a large enough $K.$\\

As before, $\lambda_N(r), \rho_N(r)$ denote the functions defined in \eqref{eq::simplified_lambda_N} and \eqref{eq::simplified_rho_N} with the dependence on $x$ dropped. We switch to polar coordinates with radius being denoted by $r$ to get
\begin{align*}
&\frac{N^2}{k_N}\bigintsss_{\|x-y\|\leq (r_N(K)N^{-1})^{\frac{1}{d}}} f(y)\left(\int_{B(x,\|x-y\|)} g(z) ~ dz\right) \frac{\left(\lambda_N^{\thetaone,\thetaone}(x,y)\right)}{(k_N-1)!}^{k_N-1}\exp(-\lambda_N^{\thetaone,\thetaone}(x,y))\\
&= \frac{N^2}{k_N}\int_0^{r_N(K)}\left(\int_{\partial B(0,r)}f(z)~dz\right)\left(\int_{B(0,r)}g(z)~dz\right)\frac{\lambda_N(r)^{k_N-1}}{(k_N-1)!}e^{-\lambda_N(r)}~dr.
\end{align*}

As before, we make the change of variables $r = \left(\frac{u}{N}\right)^\frac{1}{d}$ to get
\begin{align*}
&\frac{N^2}{k_N}\int_0^{r_N(K)}\left(\int_{\partial B(0,r)}f(z)~dz\right)\left(\int_{B(0,r)}g(z)~dz\right)\frac{\lambda_N(r)^{k_N-1}}{(k_N-1)!}e^{-\lambda_N(r)}~dr\\
&= \frac{N^2}{k_Nd}\int_0^{Nr_N(K)^d} \frac{u^{\frac{1}{d}-1}}{N^\frac{1}{d}} \left(\int_{\partial B\left(0,\frac{u^\frac{1}{d}}{N^\frac{1}{d}}\right)} f(z)~dz\right)\\
& \qquad \qquad \qquad \qquad \times \left(\int_{B\left(0,\frac{u^\frac{1}{d}}{N^\frac{1}{d}}\right)}g(z)~dz\right) \frac{\lambda_N\left(\frac{u^\frac{1}{d}}{N^\frac{1}{d}}\right)^{k_N-1}}{(k_N-1)!}e^{-\lambda_N\left(\frac{u^\frac{1}{d}}{N^\frac{1}{d}}\right)}~du.
\end{align*}

Using \Cref{lemma::ball_and_surface_integrals} we have
\begin{align*}
    \frac{N u^{\frac{1}{d}-1}}{N^\frac{1}{d}}\int_{\partial B\left(0,\frac{u^\frac{1}{d}}{N^\frac{1}{d}}\right)} f(z)~dz &= dV_df(0) + \frac{V_d \tr(\Hnormal_x f(0))}{2} \frac{u^\frac{2}{d}}{N^\frac{2}{d}} + O\left(\frac{u^\frac{3}{d}}{N^\frac{3}{d}}\right),\\
    \vspace{2mm}
    N \int_{B\left(0,\frac{u^\frac{1}{d}}{N^\frac{1}{d}}\right)}g(z)~dz &= g(0)V_d u + \frac{V_d \tr(\Hnormal_x g(0))}{2(d+2)} \frac{u^{1+\frac{2}{d}}}{N^\frac{2}{d}} + O\left(\frac{u^{1+\frac{3}{d}}}{N^\frac{3}{d}}\right).
\end{align*}

Let $g_N(u)$ be as defined in \Cref{lemma::change_of_variables}. Making the change of variables $v = g_N(u)$, or equivalently $u = g_N^{-1}(v),$ and using \Cref{lemma::change_of_variables} gives
\begin{align*}
&\frac{1}{g_N'(g_N^{-1}(v))} = \frac{1}{p_\thetaone(0)V_d}\left(1-\frac{\tr(\Hnormal_x p_\thetaone(0))}{2p_\thetaone(0)^{1+\frac{2}{d}} V_d^\frac{2}{d}}\frac{v^\frac{2}{d}}{N^\frac{2}{d}} + O(r_N(K)^3)\right),\\
\vspace{2mm}
&dV_df(0) + \frac{V_d \tr(\Hnormal_x f(0))}{2} \frac{u^\frac{2}{d}}{N^\frac{2}{d}} + O\left(\frac{u^\frac{3}{d}}{N^\frac{3}{d}}\right)\\
&= dV_df(0) + \frac{V_d \tr(\Hnormal_xf(0))}{2p_\thetaone(0)^\frac{2}{d}V_d^\frac{2}{d}}\frac{v^\frac{2}{d}}{N^\frac{2}{d}} + O\left(\frac{v^\frac{3}{d}}{N^\frac{3}{d}}\right),\\
\vspace{2mm}
&g(0)V_d u + \frac{V_d \tr(\Hnormal_x g(0))}{2(d+2)} \frac{u^{1+\frac{2}{d}}}{N^\frac{2}{d}} + O\left(\frac{u^{1+\frac{3}{d}}}{N^\frac{3}{d}}\right)\\
&= \frac{g(0)}{p_\thetaone(0)}v +\left(\frac{p_\thetaone(0)\tr(\Hnormal_x g(0)) - g(0)\tr(\Hnormal_x p_\thetaone(0))}{2(d+2)V_d^\frac{2}{d}p_\thetaone(0)^{2+\frac{2}{d}}}\right) \frac{v^{1+\frac{2}{d}}}{N^\frac{2}{d}} + O\left(\frac{v^{1+\frac{3}{d}}}{N^\frac{3}{d}}\right).
\end{align*}

Using the above expansion, and making the change of variables $v = g_N(u)$ in the integral, we now get
\begin{align*}
    &\frac{N^2}{k_Nd}\int_0^{Nr_N(K)^d} \frac{u^{\frac{1}{d}-1}}{N^\frac{1}{d}} \left(\int_{\partial B\left(0,\frac{u^\frac{1}{d}}{N^\frac{1}{d}}\right)} f(z)~dz\right)\left(\int_{B\left(0,\frac{u^\frac{1}{d}}{N^\frac{1}{d}}\right)}g(z)~dz\right) \frac{\lambda_N\left(\frac{u^\frac{1}{d}}{N^\frac{1}{d}}\right)^{k_N-1}}{(k_N-1)!}e^{-\lambda_N\left(\frac{u^\frac{1}{d}}{N^\frac{1}{d}}\right)}~du\\
    \vspace{2mm}
    &=\frac{f(0)g(0)}{p_\thetaone(0)^2}\int_0^\infty \frac{v^{k_N}}{k_N!}e^{-v}~dv\\
    \vspace{2mm}
    &- \frac{1}{N^\frac{2}{d}}\frac{f(0)g(0)\tr(\Hnormal_x p_\thetaone(0))}{2dp_\thetaone(0)^{3+\frac{2}{d}}V_d^\frac{2}{d}} \int_0^\infty \frac{v^{k_N + \frac{2}{d}}}{k_N!}e^{-v}~dv + \frac{1}{N^\frac{2}{d}}\frac{g(0)\tr(\Hnormal_x f(0))}{2dp_\thetaone(0)^{2+\frac{2}{d}}V_d^\frac{2}{d}}\int_0^\infty \frac{v^{k_N+\frac{2}{d}}}{k_N!}e^{-v}~dv\\
    \vspace{2mm}
    &+ \frac{1}{N^\frac{2}{d}} \frac{f(0)(p_\thetaone(0)\tr(\Hnormal_x g(0)) - g(0)\tr(\Hnormal_x p_\thetaone(0)))}{2(d+2)V_d^\frac{2}{d}p_\thetaone(0)^{3+\frac{2}{d}}}\int_0^\infty \frac{v^{k_N+\frac{2}{d}}}{k_N!}e^{-v}~ dv + O(r_N(K)^3)\\
    &= \frac{k_N^\frac{2}{d}}{N^\frac{2}{d}}\frac{\Gamma(k_N+1+\frac{2}{d})}{k_N^\frac{2}{d} \Gamma(k_N+1)}\left(\frac{p_\thetaone(0)g(0)\tr(\Hnormal_x f(0)) - f(0)g(0)\tr(\Hnormal_x p_\thetaone(0))}{2dp_\thetaone(0)^{3+\frac{2}{d}}V_d^\frac{2}{d}}\right)\\
    &+\frac{k_N^\frac{2}{d}}{N^\frac{2}{d}}\frac{\Gamma(k_N+1+\frac{2}{d})}{k_N^\frac{2}{d}\Gamma(k_N+1)}\left(\frac{f(0)(p_\thetaone(0)\tr(\Hnormal_x g(0)) - g(0)\tr(\Hnormal_x p_\thetaone(0)))}{2(d+2)V_d^\frac{2}{d}p_\thetaone(0)^{3+\frac{2}{d}}}\right) + O(r_N(K)^3).
\end{align*}

In the first equality, we have used the concentration inequality for Gamma random variables in \Cref{corollary::gamma_concentration} in order to change the integral from finite to infinite. Rewriting the last equality completes the proof.
\end{proof}


\begin{lemma}\label{lemma::lemma_for_final_hessian_term}
    Let $g$ be three times differentiable, non-negative and bounded above and below on $S.$ Then, for any fixed $x,$
    \begin{align*}
        &\frac{N^3}{k_N}\bigintsss p_\thetaone(y)\left(\bigintsss_{B(x,\|x-y\|)} g(z)~dz\right)^2\left(\frac{\lambda_N^{\thetaone,\thetaone}(x,y)}{(k_N-1)!}^{k_N-1}-\frac{\lambda_N^{\thetaone,\thetaone}(x,y)^{k_N-2}}{(k_N-2)!}\right)\exp(-\lambda_N^{\thetaone,\thetaone}(x,y))~dy\\
        &\to \frac{2g(x)^2}{p_\thetaone(x)^2}
    \end{align*}
\end{lemma}

\begin{proof}

WLOG we assume $x=0.$ As before, we can bound the integral to $\|y\|\leq r_N(K).$ The proof proceeds in multiple steps. \\

\textbf{Step 1}: In the first step, we will write the integral in a much simpler form involving certain expectation of Gamma and Exponential random variables.

Denote the integral in question as $\Ical_N.$ 
WLOG we assume $x=0.$ Changing to spherical coordinates with the radius being denoted by $r$ we can rewrite the integral as
\begin{align*}
\Ical_N &=\frac{N^3}{k_N}\bigintsss_{\|y\|\leq r_N(K)} p_\thetaone(y)\left(\bigintsss_{B(0,\|y\|)} g(z)~dz\right)^2\\
& \qquad \qquad \qquad \qquad \times \left(\frac{\lambda_N^{\thetaone,\thetaone}(0,y)}{(k_N-1)!}^{k_N-1}-\frac{\lambda_N^{\thetaone,\thetaone}(0,y)^{k_N-2}}{(k_N-2)!}\right)\exp(-\lambda_N^{\thetaone,\thetaone}(0,y))~dy\\
&= \frac{N^3}{k_N}\bigintsss_0^{r_N(K)} \left(\bigintsss_{\partial B(0,r)} p_\thetaone(z)~dz\right)\left(\bigintsss_{B(0,r)} g(z)~dz \right)^2\left(\frac{\lambda_N(0,r)^{k_N-1}}{(k_N-1)~} - \frac{\lambda_N(0,r)^{k_N-2}}{(k_N-2)!}\right)~dr
\end{align*}

where $\lambda_N$ is as defined in \eqref{eq::simplified_lambda_N}. Going forth in this proof, we will drop the first coordinate and simply denote it by $\lambda_N(r).$ We now make the change of variable $r = \frac{u^\frac{1}{d}}{N^\frac{1}{d}}$ to rewrite the integral $\Ical_N$ as
\begin{align*}
&\frac{N^2}{k_N}\bigintsss_0^{Nr_N(K)^d} \frac{Nu^{\frac{1}{d}-1}}{dN^\frac{1}{d}}\left(\bigintsss_{\partial B\left(0,\frac{u^\frac{1}{d}}{N^\frac{1}{d}}\right)} p_\thetaone(z)~dz\right)\left(\bigintsss_{B\left(0,\frac{u^\frac{1}{d}}{N^\frac{1}{d}}\right)} g(z)~dz\right)^2\\
& \qquad \qquad \qquad \qquad \times \left(\frac{\lambda_N\left(\frac{u^\frac{1}{d}}{N^\frac{1}{d}}\right)^{k_N-1}}{(k_N-1)!} - \frac{\lambda_N\left(\frac{u^\frac{1}{d}}{N^\frac{1}{d}}\right)^{k_N-2}}{(k_N-2)!}\right) ~ du.
\end{align*}

If we now define $g_N(u) = \lambda_N\left(\frac{u^\frac{1}{d}}{N^\frac{1}{d}}\right)$ as in \Cref{lemma::change_of_variables}, then
$$g_N'(u) = \frac{Nu^{\frac{1}{d}-1}}{dN^\frac{1}{d}}\left(\bigintsss_{\partial B\left(0,\frac{u^\frac{1}{d}}{N^\frac{1}{d}}\right)} p_\thetaone(z)~dz\right).$$

Hence, making the change of variables $v=g_N(u),$ we get that
\begin{equation*}
    \Ical_N = \frac{1}{k_N}\bigintsss_0^{g_{_N}(Nr_N(K)^d)} \left(N\bigintsss_{B\left(0,\left(\frac{(g_N^{-1}(v)}{N}\right)^\frac{1}{d}\right)} g(z)~dz\right)^2\left(\frac{v^{k_N-1}}{(k_N-1)!} - \frac{v^{k_N-2}}{(k_N-2)!}\right)e^{-v}~dv.
\end{equation*}

Let $G_N(v):= \left(N\bigintsss_{B\left(0,\left(\frac{(g_N^{-1}(v)}{N}\right)^\frac{1}{d}\right)} g(z)~dz\right).$ Then, the integral $\Ical_N$ can be rewritten as
\begin{equation}\label{eq::rewriting_I}
\Ical_N = \frac{1}{k_N}\E\left((G_N^2(X+Y)-G_N^2(X))\indicator\{X,X+Y\leq g_N(Nr_N(K)^d)\}\right),
\end{equation}

where $X,Y$ are independent $\Gamma(k_N-2,1)$ and $\Expo(1)$ random variables respectively. This concludes Step 1.

\textbf{Step 2}: In this step, we will show that the difference $G_N^2(v+y)-G_N^2(v)$ can be approximated by $2\frac{g(0)^2}{p_\thetaone(0)^2}vy$ when $v,v+y\leq g_N(Nr_N(K)^d)$.

We start by bounding the difference between $g_N^{-1}(v+y)$ and $g_N^{-1}(v)$ when $v,y$ are as above. Let $\epsilon>0$ and let $v,y$ be such that $0\leq v,v+y\leq g_N(Nr_N(K)^d).$ From \Cref{lemma::integral_of_density_over_small_balls} we have that
\begin{align*}
    g_N\left(g_N^{-1}(v) + \frac{(1-\epsilon)y}{p_\thetaone(0)V_d}\right) &= \int_0^{g_N^{-1}(v)} g_N'(t)~dt + \int_{g_N^{-1}(v)}^{g_N^{-1}(v)+\frac{(1-\epsilon)y}{p_\thetaone(0)V_d}} g_N'(t)~dt\\
    &\leq v + (1+\epsilon_N) \int_{g_N^{-1}(v)}^{g_N^{-1}(v) + \frac{(1-\epsilon)y}{p_\thetaone(0)V_d}} p_\thetaone(0)V_d dt\\
    &= v + y(1-\epsilon)(1+\epsilon_N)\\
    &\leq v + y\left(1-\frac{\epsilon}{2}\right).
\end{align*}
for all large $N$ and all $v,y$ as above where $\epsilon_N \to 0$ is as in \Cref{lemma::integral_of_density_over_small_balls}.

Hence, given any fixed $\epsilon>0,$ for all large $N$ we have
$$g_N\left(g_N^{-1}(v)+\frac{(1-\epsilon)y}{p_\thetaone(0)V_d}\right) \leq v+\left(1-\frac{\epsilon}{2}\right)y$$

for all $v,y$ as above. In a similar manner, we can show that for all large $N$ we have
$$v+\left(1+\frac{\epsilon}{2}\right)y \leq g_N\left(g_N^{-1}(v) + \frac{(1+\epsilon)y}{p_\thetaone(0)V_d}\right)$$

for all $v,y$ as above. Hence, since $g_N$ is monotonically increasing,
\begin{equation}\label{eq::gap_between_solutions_of_g_N}
\frac{(1-\epsilon)y}{p_\thetaone(0)V_d}\leq g_N^{-1}(v+y) - g_N^{-1}(v)\leq \frac{(1+\epsilon)y}{p_\thetaone(0)V_d}.
\end{equation}

We now use this to bound $G_N(v+y) - G_N(y).$ We can write $G_N$ as
$$G_N(v) = \Tilde{G}_N(g_N^{-1}(v))$$
where 
$$\Tilde{G}_N(u) = N\bigintsss_{B\left(0,\left(\frac{u}{N}\right)^\frac{1}{d}\right)} g(z)~dz.$$

Since $g$ is bounded below, in the same way that we proved \Cref{lemma::integral_of_density_over_small_balls}, we can show that there exists a sequence $\delta_N\to 0$ such that 
\begin{equation}\label{eq::bound_on_derivative_of_G_N}
(1-\delta_N)g(0)V_d\leq \Tilde{G}_N'(t)\leq (1+\delta_N)g(0)V_d
\end{equation}

for all $t\leq Nr_N(K)^d.$ Since 
$$\Tilde{G}_N(g_N^{-1}(v+y)) - \Tilde{G}_N(g_N^{-1}(v) = \int_{g_N^{-1}(v)}^{g_N^{-1}(v+y)} \Tilde{G}'_N(t)dt,$$
we have using \eqref{eq::gap_between_solutions_of_g_N} and \eqref{eq::bound_on_derivative_of_G_N} that
\begin{equation}\label{eq::bound_on_gap_of_G_N}
(1-2\epsilon) \frac{g(0)}{p_\thetaone(0)} y \leq G_N(v+y) - G_N(v)\leq (1+2\epsilon)\frac{g(0)}{p_\thetaone(0)}y.
\end{equation}

for all $v,y$ as above. Finally, using \Cref{lemma::change_of_variables} we see that there exists a sequence $\gamma_N\to 0$ such that for all $v,v+y\leq g_N(Nr_N(K)^d),$
\begin{align*}
    &(1-\gamma_N)\frac{v}{p_\thetaone(0)V_d} \leq g_N^{-1}(v) \leq (1+\gamma_N)\frac{v}{p_\thetaone(0)V_d},\\
    &(1-\gamma_N)\frac{v+y}{p_\thetaone(0)V_d} \leq g_N^{-1}(v+y) \leq (1+\gamma_N)\frac{v+y}{p_\thetaone(0)V_d}.
\end{align*}

Using the above along with \eqref{eq::bound_on_derivative_of_G_N} we get for all $v,y$ as above,
\begin{equation}\label{eq::bound_on_sum_of_G_N}
    (1-2\epsilon)(2v+y)\frac{g(0)}{p_\thetaone(0)}\leq G_N(v+y) + G_N(v)\leq (1+2\epsilon)(2v+y)\frac{g(0)}{p_\thetaone(0)}.
\end{equation}

The bounds in \eqref{eq::bound_on_gap_of_G_N} and \eqref{eq::bound_on_sum_of_G_N} gives
\begin{equation}\label{eq::bound_on_gap_of_GN_squared}
(1-2\epsilon)^2(2v+y)y\frac{g(0)^2}{p_\thetaone(0)^2}\leq G_N^2(v+y) -  G_N^2(v)\leq (1+2\epsilon)^2(2v+y)y\frac{g(0)^2}{p_\thetaone(0)^2}. 
\end{equation}

\textbf{Step 3}: The proof is almost complete. By taking $K$ large enough, we have $3k_N\leq g_N(Nr_N(K)^d)$ Recall that $X,Y$ are independent with $X\sim \Gamma(k_N-2,1)$ and $Y\sim \Expo(1).$ Using the Gamma concentration bound in \Cref{corollary::gamma_concentration} we have
$$\lim_{N\to \infty}\frac{1}{k_N}\E((2X+Y)Y\mid\{X,X+Y\leq g_N(Nr_N(K)^d)\})\to 2 = \lim_{N\to \infty} \frac{1}{k_N}\E((2X+Y)Y) = 2.$$

Using the above along with \eqref{eq::rewriting_I} and \eqref{eq::bound_on_gap_of_GN_squared} we have
$$2(1-2\epsilon)^2\frac{g(0)^2}{p_\thetaone(0)^2} \leq \liminf \Ical_N \leq \limsup \Ical_N \leq 2(1+2\epsilon)^2\frac{g(0)^2}{p_\thetaone(0)^2}.$$

Since $\epsilon>0$ was arbitrary, the result follows.

\end{proof}

\subsection{Limit of the gradient term}

This section will be dedicated to proving \Cref{lemma::gradient_limit}. Using the expression of $\mu_N(\theta_1,\theta_2)$ given in \eqref{eq::mu_N_defn}, we differentiate under the integral sign to get
$$\nabla_\thetaone \mu_N(\theta_1,\theta_1) = \int p_\thetaone(x) \nabla_\thetaone (p_\thetaone(y)\rho^{\thetaone,\thetaone}_{k_N}(x,y)) ~ dx~ dy. $$

Notice that $\rho_{k_N}^{\thetaone,\theta_2}(x,y)$ can be written as
\begin{align*}
\rho_{k_N}^{\thetaone,\theta_2}(x,y) &= \prob(\poi(\lambda_N^{\thetaone,\theta_2}(x,y))\leq k_N-1)\\ &= \sum_{k=0}^{k_N-1} \frac{(\lambda_N^{\thetaone,\theta_2}(x,y))^k}{k!}\exp\left(-\lambda_N^{\thetaone,\theta_2}(x,y)\right).
\end{align*}

Using the product rule and the particular form of $\rho_{k_N}^{\thetaone,\thetaone}$ above,
$$\epsilon_N^T \nabla_\thetaone \mu_N(\thetaone,\thetaone) = T_1 - N_2 T_2$$
where
\begin{align}
    T_1 &:= \frac{N}{k_N}\int p_\thetaone(x) \epsilon_N^T\nabla_{\theta_1}p_\thetaone(y) \rho_{k_N}^{\theta_1,\theta_1}(x,y) ~ dx ~ dy, \label{eq::T1_defn}\\
    T_2 &:= \frac{N}{k_N}\int p_\thetaone(x)p_\thetaone(y) \left(\int_{B(x,||x-y||)}\epsilon_N^T \nabla_{\theta_1}f(z\mid \theta_1) ~ dz\right)\left\{\rho_{k_N}^{\theta_1,\theta_1}(x,y)-\rho_{k_N-1}^{\theta_1,\theta_1}(x,y)\right\}~ dx ~ dy. \label{eq::T2_defn}
\end{align}

We can write $T_1, T_2$ as
\begin{align*}
    T_1 = \int p_\thetaone(x)~ T_1(x) ~ dx,\\
    T_2 = \int p_\thetaone(x)~ T_2(x) ~ dx.
\end{align*}

Taking $\epsilon_N = h N^{-\frac{1}{2}} \left(\frac{N}{k_N}\right)\frac{2}{d}$ and applying \Cref{lemma::lemma_for_T1} for $f = h^T \nabla_\thetaone p_\thetaone,$ we get the point wise limit of $T_1(x)$. Since $S$ is compact, the convergence is uniform and by the DCT we get the limit of $T_1.$ Specifically, we get
\begin{align*}
    \sqrt{N}T_1 &= \left(\frac{N}{k_N}\right)^{1+\frac{2}{d}} \int h^Tp_\thetaone(x)~dx\\
    &+ \bigintsss p_\thetaone(x) \left(\frac{\tr(H_x (h^T \nabla_\thetaone p_\thetaone)(x))}{2dV_d^{\frac{2}{d}}(p_{\theta_1}(x))^{1+\frac{2}{d}}}\right)\left(\frac{1}{k_N^{1+\frac{2}{d}}}\sum_{k=0}^{k_N-1}\frac{\Gamma(k+\frac{2}{d}+1)}{\Gamma(k+1)}\right) ~ dx\\
    &- \bigintsss p_\thetaone(x) \left(\frac{h^T p_\thetaone(x)\tr(H_x p_{\theta_1}(x))}{2dV_d^\frac{2}{d}p_{\theta_1}(x)^{2+\frac{2}{d}}}\right)\left(\frac{1}{k_N^{1+\frac{2}{d}}}\sum_{k=0}^{k_N-1}\frac{\Gamma(k+\frac{2}{d}+1)}{\Gamma(k+1)}\right)\\
    &+ \left(\frac{N}{k_N}\right)^\frac{2}{d} \cdot O\left(r_N(K)^3\right).
\end{align*}

Since $p_\thetaone$ is a density, the first term in the above expansion is equal to $0$ for any $N.$ Furthermore, due to the definition $r_N(K),$ the last term tends to $0.$ Hence, it suffices to find the limiting values of the second and third terms. By using the identity in \Cref{lemma::gamma_function_identity}, and using Stirling's approximation we get that as $k_N\to \infty,$
$$\frac{1}{k_N^{1+\frac{2}{d}}}\sum_{k=0}^{k_N-1}\frac{\Gamma(k+1+\frac{2}{d})}{\Gamma(k+1)} \to \frac{d}{d+2},$$

for $\epsilon_N = h N^{-\frac{1}{2}}\left(\frac{N}{k_N}\right)^\frac{2}{d}$ we get that
\begin{align}
\begin{split}\label{eq::limit_of_T1}
\sqrt{N}T_1 \to &  \bigintsss p_\thetaone(x)\left( \frac{ p_\thetaone(x)\tr(H_x (h^T \nabla_\thetaone p_\thetaone)(x))}{2(d+2)V_d^{\frac{2}{d}}(p_{\theta_1}(x))^{2+\frac{2}{d}}} - \frac{h^T p_\thetaone(x)\tr(H_x p_{\theta_1}(x))}{2(d+2)V_d^\frac{2}{d}p_{\theta_1}(x)^{2+\frac{2}{d}}}\right)~dx\\
&= \frac{1}{2(d+2)V_d^\frac{2}{d}} \bigintsss h^T \nabla_\thetaone \left(\frac{\tr(\Hnormal_x p_\thetaone(x))}{p_\thetaone(x)}\right)p_\thetaone^{1-\frac{2}{d}}(x) ~ dx. 
\end{split}
\end{align}

This gives the limiting value of $\sqrt{N}T_1$

Similarly, we can take $f=p_\thetaone$ and $g=h^T p_\thetaone$ in \Cref{lemma::lemma_for_T2} to get that for $\epsilon_N = h N^{-\frac{1}{2}}\left(\frac{N}{k_N}\right)^\frac{2}{d},$
\begin{align*}
&\sqrt{N}N_2 T_2\\
&= \frac{N_2}{N}\left(\frac{N}{k_N}\right)^\frac{2}{d} \int \nabla_\thetaone p_\thetaone(x)~dx\\  
&+\frac{N_2}{N} \frac{\Gamma(k_N+1+\frac{2}{d})}{k_N^\frac{2}{d}\Gamma(k_N+1)} \bigintsss p_\thetaone(x) \left(\frac{\tr(\Hnormal_x (p_\thetaone)(x)) h^T p_\thetaone(x)}{2dV_d^\frac{2}{d}p_\thetaone(x)^{2+\frac{2}{d}}} + \frac{p_\thetaone(x)(\tr(\Hnormal_x (h^T \nabla_\thetaone p_\thetaone)(x)))}{2(d+2)V_d^\frac{2}{d}p_\thetaone(x)^{2+\frac{2}{d}}}\right) ~ dx\\
&- \frac{N_2}{N} \frac{\Gamma(k_N+1+\frac{2}{d})}{k_N^\frac{2}{d}\Gamma(k_N+1)} \frac{d+1}{d(d+2)} \bigintsss p_\thetaone(x)~
\frac{p_\thetaone(x) h^T \nabla_\thetaone p_\thetaone(x)\tr(\Hnormal_x p_\thetaone(x))}{V_d^\frac{2}{d}p_\thetaone(x)^{3+\frac{2}{d}}} ~ dx\\
&+ \frac{N_2}{N} \left(\frac{N}{k_N}\right)^\frac{2}{d} O\left(r_N(K)^3\right).
\end{align*}

The integral in the first term equals 0 and last term tends to 0 by definition of $r_N(K).$ Hence, only the second and third terms feature in the limit. After some rewriting and applying Stirling's approximation to find the limiting value of the ratio of Gamma functions, we get that
\begin{align}\label{eq::limit_of_T2}
    \sqrt{N}N_2 T_2 \to \frac{q}{2(d+2)V_d^\frac{2}{d}} \bigintsss h^T \nabla_\thetaone \left(\frac{\tr(\Hnormal p_\thetaone(x))}{p_\thetaone(x)}\right) p_\thetaone^{1-\frac{2}{d}}(x) ~ dx.
\end{align}

The limit of gradient term in \Cref{lemma::gradient_limit} is gotten by combining \eqref{eq::limit_of_T1} and \eqref{eq::limit_of_T2}.

\subsection{Limit of the Hessian term}

This section will be dedicated to proving \Cref{lemma::hessian_limit}. If $\epsilon_N = N^{-\frac{1}{4}}h$ for some $h$, then to find the limit $\sqrt{N}\epsilon_N^T (H_\thetaone \mu_N(\thetaone,\thetaone)) \epsilon_N$ it suffices to find the limit $h^T(H_\thetaone \mu_N(\thetaone,\thetaone)) h$ for any given $h.$\\

We can differentiate under the integral sign in the expression for $\mu_N(\theta_1,\theta_2)$ given in \eqref{eq::mu_N_defn} to get
\begin{align}
\begin{split}\label{eq::Hessian_expression}
    h^T (\Hnormal_\thetaone \mu_N(\theta_1,\theta_1)) h &= \int p_\thetaone(x) (h^T \Hnormal_\thetaone p_\thetaone(y) h) \rho_{k_N}^{\thetaone,\thetaone}(x,y)~dx~dy\\
    &+ 2\int p_\thetaone(x) (h^T \nabla_\thetaone p_\thetaone(y)) (h^T \nabla_\thetaone \rho_{k_N}^{\thetaone,\thetaone}(x,y))~dx~dy\\
    &+ \int p_\thetaone(x) p_\thetaone(y) (h^T \Hnormal_\thetaone \rho_{k_N}^{\thetaone,\thetaone}  h)(x,y)~dx~dy.
\end{split}
\end{align}

We will call the expressions in \eqref{eq::Hessian_expression} as $T_{21}, T_{22}, T_{23}$ respectively.

We take $f(y)=h^T (\Hnormal_{\thetaone} p_\thetaone(y)) h$ and apply \Cref{lemma::lemma_for_T1} to get
\begin{align*}
    T_{21} = \int h^T \Hnormal_\thetaone p_\thetaone(x) h ~dx + O\left(\left(\frac{k_N}{N}\right)^\frac{2}{3}\right) + O(r_N(K)^3).
\end{align*}

Since the integral in the above expression is equal to $0,$ we get
\begin{equation}\label{eq::limit_of_T_21}
    T_{21} \to 0.
\end{equation}

We now come to $T_{22}.$ As we did in the previous section, we can use the exact form of $\rho_{k_N}^{\thetaone,\thetaone}(x,y)$ given in \eqref{eq::expression_for_rho_kn} to get that
$$\nabla_\thetaone \rho_{k_N}^{\thetaone,\theta_2}(x,y) = -N_2 \left(\int_{B(x,\|x-y\|)} \nabla_\thetaone p_\thetaone(z)~dz\right) \frac{(\lambda_N^{\thetaone,\thetaone}(x,y))}{(k_N-1)!}^{k_N-1}\exp(-\lambda_N^{\thetaone,\thetaone}(x,y)).$$

Using \Cref{lemma::lemma_for_T2} with $f = g = h^T \nabla_\thetaone p_\thetaone$ to get
$$T_{22} = -\frac{2N_2}{N}\bigintsss p_\thetaone(x) \left(\frac{h^T \nabla_\thetaone p_\thetaone(x)}{p_\thetaone(x)}\right)^2~dx + \frac{N_2}{N} O\left(\left(\frac{k_N}{N}\right)^\frac{3}{d}\right) + \frac
{N_2}{N}O(r_N(K)^3).$$

Hence, as $N\to \infty$
\begin{equation}\label{eq::limit_of_T_22}
    T_{22} \to -2q\E \left(\frac{h^T \nabla_\thetaone p_\thetaone(X)}{p_\thetaone(X)}\right)^2.
\end{equation}

To rewrite $T_{23}$ in a form which allows us to use previous results, we use the explicit summation form of $\rho_{k_N}^{\thetaone,\thetaone}(x,y)$ given in \eqref{eq::expression_for_rho_kn}. Using this, we can write $T_{23}$ as
\begin{align*}
T_{23} &= -N_2~\frac{N}{k_N}\bigintsss p_\thetaone(x) p_\thetaone(y)\left(\bigintsss_{B(x,\|x-y\|)} h^T ~ \Hnormal_\thetaone p_\thetaone(z) ~ h~dz\right)\\
& \qquad \qquad \qquad \qquad \times \frac{(\lambda_N^{\thetaone,\thetaone}(x,y))}{(k_N-1)!}^{k_N-1}\exp\left(-\lambda_N^{\thetaone,\thetaone}(x,y)\right)~dx~dy\\
&+ N_2^2 \frac{N}{k_N}\bigintsss p_\thetaone(x)p_\thetaone(y) \left(\bigintsss _{B(x,\|x-y\|)} h^T\nabla_\thetaone p_\thetaone(z)~dz\right)^2\\
& \qquad \qquad \qquad \qquad \times \frac{(\lambda_N^{\thetaone,\thetaone}(x,y))}{(k_N-1)!}^{k_N-1}\exp(-\lambda_N^{\thetaone,\thetaone}(x,y))~dy~dx\\
&- N_2^2 \frac{N}{k_N}\bigintsss p_\thetaone(x)p_\thetaone(y) \left(\bigintsss _{B(x,\|x-y\|)} h^T\nabla_\thetaone p_\thetaone(z)~dz\right)^2\\
& \qquad \qquad \qquad \qquad \times \frac{(\lambda_N^{\thetaone,\thetaone}(x,y))}{(k_N-2)!}^{k_N-2}\exp(-\lambda_N^{\thetaone,\thetaone}(x,y))~dy~dx.
\end{align*}

Call the three terms $T_{231},T_{232}$ and $T_{233}$ respectively. By applying \Cref{lemma::lemma_for_T2} with $f=p_\thetaone$ and $g = h^T (\Hnormal_\thetaone p_\thetaone) h$, we get that
$$T_{231} = - \frac{N_2}{N} \int h^T (\Hnormal_\thetaone p_\thetaone (x))h ~ dx + O\left(\left(\frac{k_N}{N}\right)^\frac{2}{d}\right).$$

Since the integral above is $0,$ we get that
\begin{equation*}
T_{231}\to 0. 
\end{equation*}

To evaluate $T_{232}+T_{233}$, we can use \Cref{lemma::lemma_for_final_hessian_term} with $g=h^T \nabla_\thetaone p_\thetaone$ to get
$$T_{232}+T_{233} \to -2q^2 \int p_\thetaone(x) \frac{(h^T\nabla_\thetaone p_\thetaone(x))}{p_\thetaone(x)^2} = 2q^2\E\left(\frac{h^T\nabla_\thetaone p_\thetaone(X)}{p_\thetaone(X)}\right)^2.$$

Since $T_{231}\to 0,$ we get that
$$T_{23}\to 2q\E\left(\frac{h^T\nabla_\thetaone p_\thetaone(X)}{p_\thetaone(X)}\right)^2.$$

The three limits of $T_{21},T_{22}$ and $T_{23}$ together give us
\begin{equation}\label{eq::limit_T_23}
    h^T \Hnormal_\thetaone \mu_N(\thetaone,\thetaone)~h \to -2pq \cdot \E\left(\frac{h^T \nabla_\thetaone p_\thetaone(X)}{p_\thetaone(X)}\right)^2.
\end{equation}

This proves \Cref{lemma::hessian_limit}.

\subsection{Controlling the remainder term}\label{appendix::controlling_remainder_term}

This section will prove \Cref{lemma::controlling_remainder_term}. Just as with the Hessian, the third derivative can also be written as a sum of multiple terms. \Cref{lemma::lemma_for_T1,lemma::lemma_for_T2} are sufficient to control most of the terms arising from this. The trickiest one to control is the term that comes from taking the third derivative of $\rho_N(\thetaone,\theta)$ with respect to $\theta.$ This term $T$ can be written as
\begin{align*}
T &= N_2^3\frac{N}{k_N}\bigintsss p_\thetaone(x)p_{\theta}(y)\left(\bigintsss_{B(x,\|x-y\|)}h^T\nabla_{\theta}p_{\theta}(z)~dz\right)^3\\
& \qquad \qquad \qquad \qquad \left(\frac{2(\lambda_N^{\theta_1,\theta}(x,y))^{k_N-2}}{(k_N-2)!}\right) \exp\left(-\lambda_N^{\theta_1,\theta}(x,y)\right)~dy~dx\\
& - N_2^3\frac{N}{k_N}\bigintsss p_\thetaone(x)p_{\theta}(y)\left(\bigintsss_{B(x,\|x-y\|)}h^T\nabla_{\theta}p_{\theta}(z)~dz\right)^3\\ & \qquad \qquad \qquad \qquad \frac{(\lambda_N^{\theta_1,\theta}(x,y))^{k_N-1}}{(k_N-1)!} \exp\left(-\lambda_N(\theta_1,\theta)(x,y)\right)~dy~dx\\
& - N_2^3\frac{N}{k_N}\bigintsss p_\thetaone(x)p_{\theta}(y)\left(\bigintsss_{B(x,\|x-y\|)}h^T\nabla_{\theta}p_{\theta}(z)~dz\right)^3\\
& \qquad \qquad \qquad \qquad \frac{(\lambda_N^{\theta_1,\theta}(x,y))^{k_N-3}}{(k_N-3)!} \exp\left(-\lambda_N(\theta_1,\theta)(x,y)\right)~dy~dx,
\end{align*}

for some $\theta$ on the segment joining $\theta_1$ and $\theta_2.$ We need to show that the above term is bounded. For convenience, we will show that it  is bounded for $\theta=\theta_1.$ The general case is similar but the proof is more tedious.

The idea is similar to the one used in the proof of \Cref{lemma::lemma_for_final_hessian_term}. The similarity between the two terms is evident. The only difference is that we have a cubic term above rather than a square term as in \Cref{lemma::lemma_for_final_hessian_term}.

Just as in \Cref{lemma::lemma_for_final_hessian_term}, we will show point-wise boundedness of the inner integral over $y$ for every $x.$ As before, compactness of the support will give a uniform bound over $x\in S$ and hence we will have shown that $T$ is bounded.

For a given $x$ and a function $g$ satisfying the assumptions in \Cref{lemma::lemma_for_final_hessian_term}, define
\begin{align*}
\Ical_N(x) &= N_2^3\frac{N}{k_N}\bigintsss p_\theta(y)\left(\bigintsss_{B(x,\|x-y\|)}h^T\nabla_\theta p_\theta(z)~dz\right)^3\\
& \qquad \qquad \qquad \qquad \qquad \times \left(\frac{2(\lambda_N^{\theta,\theta}(x,y))^{k_N-2}}{(k_N-2)!}\right) \exp\left(-\lambda_N^{\theta,\theta}(x,y)\right)~dy~dx\\
& - N_2^3\frac{N}{k_N}\bigintsss p_\theta(x)p_{\theta}(y)\left(\bigintsss_{B(x,\|x-y\|)}h^T\nabla_{\theta}p_{\theta}(z)~dz\right)^3\\
& \qquad \qquad \qquad \qquad \times \frac{(\lambda_N^{\theta,\theta}(x,y))^{k_N-1}}{(k_N-1)!} \exp\left(-\lambda_N^{\theta,\theta}(x,y)\right)~dy~dx\\
& - N_2^3\frac{N}{k_N}\bigintsss p_\theta (x) p_\theta (y)\left(\bigintsss_{B(x,\|x-y\|)}h^T\nabla_\theta p_\theta (z)~dz\right)^3\\
& \qquad \qquad \qquad \qquad \times \frac{(\lambda_N^{\theta,\theta}(x,y))^{k_N-3}}{(k_N-3)!} \exp\left(-\lambda_N^{\theta,\theta}(x,y)\right)~dy~dx.
\end{align*}

WLOG we assume $x = 0.$ As before, it is enough to restrict the integral to $\|y\|\leq r_N(K).$ For simplicity, we will simply denote the integral by $\Ical_N.$ The proof now proceeds in multiple steps as in the proof of \Cref{lemma::lemma_for_final_hessian_term}. At numerous times in the following steps, we will use the same notation to refer to different constants that do not depend on $N$ or $x\in S$.

\textbf{Step 1} : As in the first step of the proof of \Cref{lemma::lemma_for_final_hessian_term}, we will change the integral to an expectation over Gamma random variables. We achieve this by making the same sequence of variable changes.

We first make the change to spherical coordinates and denote the radius by $r.$ We then make the change of variables $r = \left(\frac{u}{N}\right)^\frac{1}{d}.$ Finally, taking $v = g_N(u)$ where $g_N(u) = \lambda_N^{\theta,\theta}\left(\left(\frac{u}{N}\right)^\frac{1}{d}\right)$ to obtain
\begin{align*}
\Ical_N &= \frac{1}{k_N}\bigintsss_0^{g_N(Nr_N(K)^d)}G_N^3(v)\left(\frac{v^{k_N-1}}{(k_N-1)!} + \frac{v^{k_N-3}}{(k_N-3)!} - 2\frac{v^{k_N-2}}{(k_N-2)!}\right)e^{-v} ~ dv,
\end{align*}

where
$$G_N(v) = \left(N\bigintsss_{B\left(0,\left(\frac{g_N^{-1}(v)}{N}\right)^\frac{1}{d}\right)} g(z)~dz \right).$$

Hence, $\Ical_N$ can be written as
$$\Ical_N = \frac{1}{k_N}\E\left((G_N^3(X+Y_1+Y_2) + G_N^3(X) - 2 G_N^3(X+Y_1))\cdot \indicator_{A_N}\right)$$

where $X,Y_1,Y_2$ are independent random variables with $X\sim \Gamma(k_N-3,1)$ and $Y_1,Y_2 \sim \Expo(1)$ and
$$A_N := \{X,X + Y_1, X+Y_1+Y_2\leq g_N(Nr_N(K)^d)\}.$$
This concludes Step 1.

\textbf{Step 2}: In the second step, we reduce the problem further from one of controlling cubic terms to one of controlling some linear terms. We will also restrict the expectation from the event $A_N$ to a smaller event.

Let
\begin{align*}
\delta_1 &= G_N(X + Y_1) - G_N(X),\\
\delta_2 &= G_N(X+Y_1+Y_2) - G_N(X+Y_1).
\end{align*}

Then,
\begin{equation}\label{eq::difference_of_cubes}
\begin{split}
    (G_N^3(X+Y_1+Y_2) + G_N^3(X) - 2G_N^3(X+Y_1) &= 3G_N^2(X)(\delta_2-\delta_1)\\
    &+ 3G_N(X)((\delta_1+\delta_2)^2 - 2\delta_1^2)\\
    &+ (\delta_1 + \delta_2)^3 - 2\delta_1^3.
\end{split}
\end{equation}

Since $g$ is bounded, using \Cref{lemma::integral_of_density_over_small_balls}, we get
\begin{align*}
g_N(Nr_N(K)^d) &\leq C\cdot K \cdot \max\{(\log N)^2, k_N\},\\
|G_N(v)| &\leq Cv,\\
|\delta_1| &\leq C Y_1,\\
|\delta_2| &\leq C Y_2,
\end{align*}

for some constant $C$ that depends only on $\|g\|_\infty$ and $\|p_\theta\|_\infty.$ Since $Y_1,Y_2$ have finite first and second moments and $\E(X)=k_N-3,$ we see that the second and third expressions in \eqref{eq::difference_of_cubes} is bounded in mean. Showing that $\Ical_N$ is bounded now comes down to showing
$$\left|\frac{1}{k_N}\E\left(G_N^2(X)(\delta_2-\delta_1)\cdot \indicator_{A_N}\right)\right|$$

is bounded. For convenience, going forward we will take 
$$\Ical_N = \frac{1}{k_N}\E\left(G_N^2(X)(\delta_2-\delta_1)\cdot \indicator_{A_N}\right).$$

Let $B,B_1,B_2$ denote the events
\begin{align*}
    B &= \{X \in [0.5 k_N,2k_N]\},\\
    B_1 &= \{Y_1 \leq 5\log k_N\},\\
    B_2 &= \{Y_2\leq 5\log k_N\}.
\end{align*}

By using \Cref{corollary::gamma_concentration} and tail bounds for Exponential random variables, we get that 
$$\prob(B^c),\prob(B_1^c),\prob(B_2^c)\leq k_N^{-5}.$$
Furthermore, using the bounds on $\delta_1,\delta_2$ and $G_N$ established above, we also have that
$$|G_N^2(X)(\delta_2-\delta_1)|\leq C k_N^2(Y_1+Y_2).$$

Hence,
\begin{align*}
    \left|\frac{1}{k_N}\E\left(G_N^2(X)(\delta_2-\delta_1)\cdot \indicator_{A_N\cap S^c}\right)\right| \to 0
\end{align*}

for $S = B,B_1,B_2.$ By taking $K$ large enough, we have that $A_N\cap B \cap B_1\cap B_2 = B \cap B_1\cap B_2.$ Hence, going forward we will denote by $A_N$ the event $B\cap B_1\cap B_2.$ Specifically, $A_N$ is now defined as
$$A_N := \{X\in [0.5k_N, 2k_N]~;~ Y_1,Y_2\leq 5\log k_N\}.$$

Define
$$\Tilde{\delta}_2 = G_N(X+Y_2) - G_N(X).$$

Note that $\tilde{\delta}_2$ and $\delta_1$ have the same distribution and conditioned on $X,$ are functions of $Y_1$ and $Y_2$ respectively.Using the independence of $Y_1,Y_2$ and the new definition of $A_N,$ we get that
$$\E(G_N^2(X)(\deltatilde_2-\delta_1)\cdot \indicator_{A_N}) = 0.$$

Hence, we can now take
$$\Ical_N = \frac{1}{k_N}\E\left(G_N^2(X)(\delta_2-\tilde{\delta}_2)\cdot \indicator_{A_N}\right).$$

This concludes Step 2.

\textbf{Step 3}: This is the longest and final step. In this step, we will control $\delta_2-\Tilde{\delta}_2$ and complete the proof of boundedness of $\Ical_N.$\\

If we show
$$|\delta_2-\Tilde{\delta}_2|\indicator_{A_N}\leq u(Y_1,Y_2) k_N^{-1}$$

for some polynomial $u,$ then we will have that
$$|\Ical_N|\leq \frac{C}{k_N^2}\E(G_N^2(X)\cdot \indicator_{A_N}).$$

since $Y_1,Y_2$ have moments of all orders. From \Cref{lemma::ball_and_surface_integrals,lemma::integral_of_density_over_small_balls} we get that
$$G_N^2(X)\indicator_{A_N}\leq C k_N^2,$$

for some constant $C.$ Putting these together, we will have shown that $\Ical_N$ is bounded. Hence, we only need to prove
$$|\delta_2-\Tilde{\delta}_2|\indicator_{A_N}\leq u(Y_1,Y_2) k_N^{-1}$$

for some polynomial $u.$ For this, we now define the following quantities.
\begin{align*}
    l &:= \left(\frac{g_N^{-1}(X+Y_1+Y_2)}{N}\right)^\frac{1}{d} - \left(\frac{g_N^{-1}(X+Y_1)}{N}\right)^\frac{1}{d},\\
    \Tilde{l} &:= \left(\frac{g_N^{-1}(X+Y_2)}{N}\right)^\frac{1}{d} - \left(\frac{g_N^{-1}(X)}{N}\right)^\frac{1}{d},\\
    p &:= \left(\frac{g_N^{-1}(X+Y_1)}{N}\right)^\frac{1}{d},\\
    \Tilde{p} &:= \left(\frac{g_N^{-1}(X)}{N}\right)^\frac{1}{d}.
\end{align*}

Define $\delta_3$ to be
$$\delta_3 := N\bigintsss_{B(0,\Tilde{p}+l)} g(z)~dz - N \bigintsss_{B(0,\Tilde{p})} g(z)~dz.$$

We know by their definitions that
\begin{align*}
    \delta_2 &= N\bigintsss_{B(0,p+l)} g(z)~dz - N \bigintsss_{B(0,p)} g(z)~dz,\\
    \Tilde{\delta}_2 &= N\bigintsss_{B(0,\Tilde{p}+\Tilde{l})} g(z)~dz - N \bigintsss_{B(0,\Tilde{p})} g(z)~dz.
\end{align*}

We will show that
\begin{equation}\label{eq::bound_needed_on_delta_differences}
\begin{split}
|\deltatilde_2-\delta_2|\cdot \indicator_{A_N}\leq \frac{u_1(Y_1,Y_2)}{k_N} ,\\
|\deltatilde_2-\delta_3|\cdot \indicator_{A_N} \leq \frac{u_2(Y_1,Y_2)}{k_N} ,
\end{split}
\end{equation}

for some polynomials $u_1,u_2.$ These together will give the required bound on $|\delta_2-\deltatilde_2|.$\\

We start with $\Tilde{\delta}_2 - \delta_3.$
\begin{align*}
    |\Tilde{\delta_2} - \delta_3| &= N \left|\bigintsss_{B(0,\Tilde{p}+\Tilde{l})} g(z)~dz - \bigintsss_{B(0,\Tilde{p}+l)} g(z)~dz\right|\\
    &= N\left|\int_{\Tilde{p}+l}^{\Tilde{p}+\Tilde{l}} r^{d-1}\int_{\partial B(0,1)} g(r\cdot u)~du~dr\right|\\
    &\leq N dV_d \|g\|_\infty  \int_{\Tilde{p}+l}^{\Tilde{p}+\Tilde{l}} r^{d-1}~dr\\
    &= N dV_d \|g\|_\infty \Tilde{p}^d\left|\left(1+\frac{l}{\Tilde{p}}\right)^d - \left(1+\frac{\Tilde{l}}{\Tilde{p}}\right)^d\right|.
\end{align*}

Over the event $A_N,$ $l/\Tilde{p}$ and $\Tilde{l}/\Tilde{p}$ are non-negative and bounded above. Hence, we have the bound
\begin{equation}\label{eq::delta_2_tilde-delta_3}
|\Tilde{\delta}_2 - \delta_3| \leq C N \Tilde{p}^{d-1}|l-\Tilde{l}|
\end{equation}

over the event $A_n.$ We now bound $\delta_2-\delta_3.$
\begin{equation}\label{eq::delta_2-delta_3}
\begin{split}
|\delta_2-\delta_3| &= N \left |\int_{\Tilde{p}}^{\Tilde{p}+l} r^{d-1}\int_{\partial B(0,1)}g(r\cdot u)du - \int_{p}^{p+l} r^{d-1}\int_{\partial B(0,1)}g(r\cdot u)du\right |\\
&\leq N\left |\int_{\Tilde{p}}^{\Tilde{p}+l} (r^{d-1} - (r+p-\Tilde{p})^{d-1})\int_{\partial B(0,1)}g(r\cdot u) ~ du ~dr \right |\\
&+ N \left | \int_{\Tilde{p}}^{\Tilde{p}+l} r^{d-1}\int_{\partial B(0,1)}(g(r\cdot u) - g((r+p-\Tilde{p})\cdot u))~du~dr \right |\\
&\leq N \cdot d V_d \left(\|g\|_\infty \left|1-\left(1+\frac{p-\Tilde{p}}{\Tilde{p}}\right)^{d-1}\right| + C \cdot |p-\Tilde{p}|\right)\int_{\Tilde{p}}^{\Tilde{p}+l}r^{d-1}~dr\\
&= N \cdot V_d \left(\|g\|_\infty \left|1-\left(1+\frac{p-\Tilde{p}}{\Tilde{p}}\right)^{d-1}\right| + C \cdot |p-\Tilde{p}|\right)\left((\Tilde{p}+l)^d - \Tilde{p}^d\right),
\end{split}
\end{equation}

where $C$ is some constant that depends on $\|\nabla g\|_\infty.$ From \eqref{eq::delta_2_tilde-delta_3} and \eqref{eq::delta_2-delta_3} we see that the problem has been reduced to showing suitable bounds on $l-\Tilde{l}$ and $p-\Tilde{p}$ over the event $A_N.$ We divide this into two sub-steps.

\textbf{Step 3a:} In this step we will prove the required bound on $|\delta_2-\delta_3|.$\\

Define $\Delta$ and $\Tilde{\Delta}$ as
\begin{align*}
    \Delta_1 &= g_N^{-1}(X+Y_1) - g_N^{-1}(X),\\
    \Delta_2 &= g_N^{-1}(X+Y_1+Y_2) - g_N^{-1}(X+Y_1),\\
    \Tilde{\Delta}_2 &= g_N^{-1}(X+Y_2) - g_N^{-1}(X).
\end{align*}

From \Cref{lemma::integral_of_density_over_small_balls} we see that there exists a constant $C$ depending only on $\|p_\theta\|_\infty$ such that
\begin{align*}
|\Delta_1|\leq C Y_1,\\
|\Delta_2| \leq C Y_2.
\end{align*}

Hence, using the Taylor expansion of the function $x^{1/d}$ near $1,$ we get that over the event $A_N,$
\begin{align*}
|p-\Tilde{p}| &= \Tilde{p}\left(\left(1+\frac{\Delta_1}{N \Tilde{p}^d}\right)^\frac{1}{d} - 1\right)\\
&\leq C \ptilde \frac{Y_1}{N\ptilde^d}.
\end{align*}

for some global constant $C.$ Similarly, we get on the event $A_N,$
\begin{align*}
|l| &= p\left|\left(1+\frac{\Delta_2}{Np^d}\right)^\frac{1}{d} -1 \right|\\
&\leq C p \frac{Y_2}{Np^d},
\end{align*}

The above expressions give us bounds on $|p-\Tilde{p}| ,\quad |p-\Tilde{p}|/\Tilde{p}$ and $|l|/p.$ By \Cref{lemma::integral_of_density_over_small_balls} we see that $p/\ptilde$ is non-negative and bounded above on the event $A_N.$ Hence, substituting the above bounds in \eqref{eq::delta_2-delta_3} gives that over the event $A_N,$
\begin{align*}
|\delta_2-\delta_3| &\leq C N p^d \frac{|l|}{p}\frac{1}{N\ptilde^d}\left(Y_1 + pY_2\right)\\
&\leq C \frac{Y_2}{Np^d}(Y_1 + pY_2).
\end{align*}

From \Cref{lemma::integral_of_density_over_small_balls} we get that $Np_d \geq C k_N$ for some constant $C$ depending only on $\|p_\theta\|_\infty.$ Also, $p\to 0$ uniformly on the event $A_N.$ Hence, we get 
\begin{equation}\label{eq::bound_on_delta_2-delta_3}
|\delta_2-\delta_3|\leq \frac{Y_2(1+2Y_2)}{k_N}
\end{equation}

which provides the bound on $\delta_2-\delta_3$ required in \eqref{eq::bound_needed_on_delta_differences}.

\textbf{Step 3b:} We now show the bound on $|\deltatilde_2-\delta_3|.$\\

From \eqref{eq::delta_2_tilde-delta_3} we see that it is enough to show for some polynomial $u_2,$
\begin{equation}\label{eq::bound_needed_on_l-l_tilde}
|l-\ltilde|\leq \frac{u_2(Y_1,Y_2)}{N^2 p^{2d-1}}.
\end{equation}

Note that
\begin{align*}
l = p \left(\left(1+\frac{\Delta_2}{Np^d}\right)^\frac{1}{d} -1\right).
\end{align*}

As noted before, $p^d\geq Ck_N$ for some global constant $C.$ Hence, by the second order Taylor expansion of the function $x^{1/d}$ around $1,$ we get
$$l = p\left(\frac{\Delta_2}{dNp^d} + O\left(\frac{Y_2^2}{k_N^2}\right)\right).$$

Similarly,
$$\ltilde = \ptilde\left(\frac{\Deltatilde_2}{dN\ptilde^d} + O\left(\frac{Y_2^2}{k_N^2}\right)\right).$$

Note that the second order terms are smaller than the bound desired in \eqref{eq::bound_needed_on_l-l_tilde}. Hence, we only need to bound the difference of the first order terms of $l$ and $\ltilde.$ We will bound the following two terms
\begin{align*}
    \frac{p^{1-d}}{dN}(\Delta_2 - \Deltatilde_2),\\
    \frac{\Deltatilde_2}{dN}(p^{1-d} - \ptilde^{1-d}).
\end{align*}

By triangle inequality, we will have the required bound on $l-\ltilde.$\\

The second term is easier to handle and using the same method as we used to bound $p-\ptilde,$ we can show that
\begin{equation}\label{eq::bound_on_second_term_of_l-ltilde}
\left|\frac{\Deltatilde_2}{dN}(p^{1-d} - \ptilde^{1-d})\right|\leq C\frac{Y_2Y_1}{dN^2}p^{1-2d}
\end{equation}

We now bound the first term. Let $u,v\in [g_N^{-1}(X), g_N^{-1}(X+Y_1+Y_2)].$ We first bound $g_N'(u)-g_N^{-1}(v).$ Using the definition of $g_N,$ we can write out the expression for $g_N'$ and get the difference as
\begin{align*}
    |g_N'(u)-g_N'(v)| &\leq \frac{1}{d}\int_{\partial B(0,1)} \left|p_\theta\left(\left(\frac{u}{N}\right)^\frac{1}{d}z\right) - p_\theta\left(\left(\frac{v}{N}\right)^\frac{1}{d}z\right)\right|~dz\\
    &\leq \frac{C}{d}\left(\frac{u}{N}\right)^\frac{1}{d}\left|1-\left(1+\frac{v-u}{u}\right)^\frac{1}{d}\right|
\end{align*}

On the event $A_N,$ we know that $C k_N\leq u\leq C'k_N,$ $|v-u|\leq C(Y_1+Y_2)\leq 10C\log k_N$ for some global constants $C,C'.$ Hence, Taylor expanding the function $x^{1/d}$ around $1$ in the above bound gives
$$|g_N'(u)-g_N'(v)|\leq C\left(\frac{k_N}{N}\right)^\frac{1}{d}\frac{(Y_1+Y_2)}{N\ptilde^d}$$

In particular, this shows that
\begin{align*}
&|g_N(g_N^{-1}(X+Y_1)+\Deltatilde_2) - (X+Y_1+Y_2)|\\
&\leq \left|\bigintsss_{g_N^{-1}(X+Y_1)}^{g_N^{-1}(X+Y_1)+\Deltatilde_2} g_N'(z)~dz - Y_2\right|\\
&= \left|\bigintsss_{g_N^{-1}(X+Y_1)}^{g_N^{-1}(X+Y_1)+\Deltatilde_2} g_N'(z)~dz - \bigintsss_{g_N^{-1}(X)}^{g_N^{-1}(X)+\Deltatilde_2} g_N'(z)~dz \right|\\
&\leq C \left(\frac{k_N}{N}\right)^\frac{1}{d}\frac{(Y_1+Y_2)}{N\ptilde^d}\Deltatilde_2\\
&\leq C \left(\frac{k_N}{N}\right)^\frac{1}{d}\frac{(Y_1+Y_2)Y_2}{N\ptilde^d},
\end{align*}

where the last step follows from the fact that $\Deltatilde_2\leq CY_2$ for some global constant $C.$ Note that since $g_N$ is increasing and $g_N'$ is positive and bounded below, the above bound shows that
$$|g_N^{-1}(X+Y_1)+\Deltatilde_2 - g_N^{-1}(X+Y_1+Y_2)|\leq C' \left(\frac{k_N}{N}\right)^\frac{1}{d}\frac{(Y_1+Y_2)Y_2}{N\ptilde^d}$$

for some other global constant $C'.$ In particular this shows that
\begin{equation}\label{eq::bound_on_first_term_of_l-ltilde}
|\Delta_2-\Deltatilde_2|\leq  C' \left(\frac{k_N}{N}\right)^\frac{1}{d}\frac{(Y_1+Y_2)Y_2}{N\ptilde^d}.
\end{equation}

Using \eqref{eq::bound_on_first_term_of_l-ltilde} and \eqref{eq::bound_on_second_term_of_l-ltilde}, we get the required bound on $|l-\ltilde|$ and hence the bound needed on $|\deltatilde_2-\delta_3|$ in \eqref{eq::bound_needed_on_delta_differences}. As noted before, this completes the proof of showing $\Ical_N$ is bounded.


\section{Additional material on simulations}

\label{appendix::additional_simulations}

\subsection{Calculations for normal scale perturbations}

In this section, we complete the full calculation of $b(h,\theta_1)$ for the normal scale perturbations in \Cref{subsection::power_comparison_simulation}.

Recall the untruncated density is

$$p(x\mid \theta) = \frac{1}{(2\pi)^{\frac{d}{2}} (\theta_1...\theta_m)^\frac{1}{2}} \exp\left(-\sum_{i=1}^m\frac{x_i^2}{2\theta_i} + \sum_{j>m}\frac{x_i^2}{2}\right),$$

while the truncated density is
$$p_L(x\mid \theta) = \frac{1}{Z_L(\theta)} \exp\left(-\sum_{i=1}^m\frac{x_i^2}{2\theta_i} + \sum_{j>m}\frac{x_i^2}{2}\right)\indicator\{x\in [-L,L]^d\}.$$
for $x\in [-L,L]^d$, where $Z_L(\theta)$ is some normalizing constant. We want to find the sign of the integral
$$\bigintssss_{[-L,L]^d} h^T\nabla_{\theta_1}\left(\frac{\tr(\Hnormal_x p_L(x|\theta_1))}{p_L(x|\theta_1)}\right)p_L^{\frac{d-2}{d}}(x|\theta_1) ~ dx,$$

To calculate the trace of the Hessian, we first note that
$$\nabla_x p_L(x\mid \theta) = -p_L(x\mid \theta)\left(\sum_{i=1}^m \frac{x_i}{\theta_i}e_i + \sum_{i>m}x_ie_i\right).$$

Hence,
\begin{equation*}
\nabla^2_{x_i}p_L(x\mid \theta) = \begin{cases*}
-p_L(x\mid \theta) \left(\frac{1}{\theta_i} - \frac{x_i^2}{\theta_i^2}\right) \text{ if $i\leq m$,}\\
-p_L(x\mid \theta)\left(1-x_i^2\right) \text{ otherwise.}
\end{cases*}
\end{equation*}

This gives
$$\frac{\tr(\Hnormal_x p(x_L\mid\theta))}{p_L(x\mid \theta)} \implies h^T \nabla_x \left(\frac{\tr(\Hnormal_x p_L(x\mid \theta))}{p_L(x\mid \theta)}\right) = \sum_{i=1}^m h_i(1-2x_i^2).$$

As a result we have for $\theta_1 = \indicator_m,$
\begin{align*}
\bigintssss_S h^T\nabla_{\theta_1}\left(\frac{\text{tr}(\text{H}_x p_L(x|\theta_1))}{p(x|\theta_1)}\right)p_L^{\frac{d-2}{d}}(x|\theta_1) ~ dx &= \bigintssss_{[-L,L]^d} 
\left(\sum_{i=1}^m h_i(1-2x_i^2)\right) p_L^{\frac{d-2}{d}}(x|\theta_1)\\
&\propto \sum_{i=1}^m h_i(1-2\E(X_i^2)),
\end{align*}
where $X_i \sim N\left(0,\frac{d}{d-2}\right)\mid_{[-L,L]}$. Thus, the sign of $b(h,\theta_1)$ is determined by the sign of $\sum h_i$ for $\theta_1 = \indicator_m.$

In calculating the gradients and Hessians, we never used the boundedness of $L$. The truncation only entered at the final step when actually evaluating the integral. Hence, in the untruncated case we have
\begin{align*}
\bigintssss h^T\nabla_{\theta_1}\left(\frac{\text{tr}(\text{H}_x p_L(x|\theta_1))}{p(x|\theta_1)}\right)p_L^{\frac{d-2}{d}}(x|\theta_1) ~ dx &= \bigintssss 
\left(\sum_{i=1}^m h_i(1-2x_i^2)\right) p_L^{\frac{d-2}{d}}(x|\theta_1)\\
&\propto \sum_{i=1}^m h_i(1-2\E(X_i^2)),
\end{align*}
where $X_i \sim N\left(0,\frac{d}{d-2}\right)$.

\subsection{Power in log-normal family}

We will consider the standard log-normal location family, where the density is given by
$$p(x\mid \theta) = \frac{1}{(2\pi)^\frac{d}{2}x_1\dots x_d}\exp\left(-\frac{\|\log x - \theta\|^2}{2}\right).$$

As in the previous power simulations, we will take $\theta_1$ to be the null and $\theta_N = \theta_1 + hN^{b}$ to be the local alterative, for some negative exponent $b$ and some direction $h\neq 0.$ We will consider the standard Lognormal to be the null, giving $\theta_1 = 0.$ To determine the effect of the direction, we require $b(h,\theta_1)$, defined in \Cref{eq::Gradient_term}. For this, we do the following calculation. We see that
$$\log p(x\mid \theta) = -\frac{\|\log x - \theta\|^2}{2} - \sum_{i=1}^d \log x_i + C,$$

where $C$ is some constant that does not depend on $x$ or $\theta.$ Hence,
$$\nabla_{x_i} p(x\mid \theta) = -p(x\mid \theta)\left(\frac{(\log x_i - \theta_i)}{x_i} + \frac{1}{x_i}\right).$$

This in turn gives
\begin{align*}
\nabla^2_{x_i}p(x\mid \theta) &= p(x\mid \theta)\left(\left(\frac{\log x_i - \theta_i}{x_i} + \frac{1}{x_i}\right)^2 + \frac{(\log x_i - \theta_1)}{x_i^2}\right).
\end{align*}

To find the sign of $b(h,\theta_1)$ it is enough to find the sign of the following integral
\begin{align*}
\bigintssss_S h^T\nabla_{\theta_1}\left(\frac{\text{tr}(\text{H}_x p(x|\theta_1))}{p(x|\theta_1)}\right)p^{\frac{d-2}{d}}(x|\theta_1) ~ dx.
\end{align*}
\begin{figure}
    \centering
    \captionsetup{width = 0.9\linewidth}
    \includegraphics[width=0.9\linewidth]{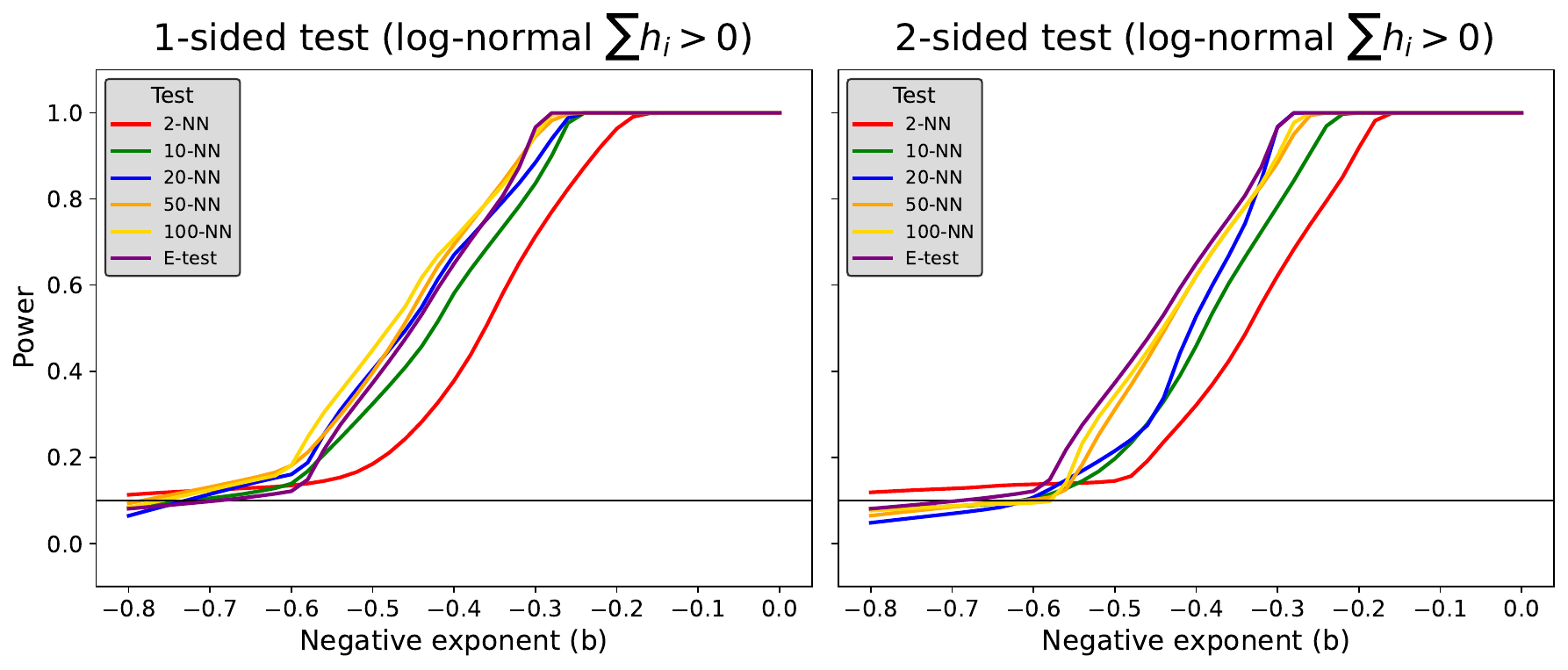}
    \caption{Power comparison in log normal family with $\sum h_i > 0$.}
    \label{fig::log_normal_family_pos_sum_lineplots}
\end{figure}

Hence, for $\theta_1=0,$ using the above expression for the second derivatives we have
\begin{align*}
&h^T\nabla_{\theta_1}\left(\frac{\text{tr}(\text{H}_x p(x|\theta_1))}{p(x|\theta_1)}\right)p^{\frac{d-2}{d}}(x|\theta_1) ~ dx\\
&= - \left(\sum_{i=1}^d h_i\right) \bigintssss \frac{1}{x^\frac{d-2}{d}(2\pi)^\frac{(d-2)}{2d}}\left(\frac{3}{x^2} + \frac{\log x}{x^2}\right)\exp\left(-\frac{(d-2)(\log x)^2}{2d}\right)~dx\\
&= - C \left(\sum_i h_i\right)\bigintssss \frac{1}{(2\pi)^\frac{1}{2}}\left(\frac{d-2}{d}\right)^\frac{1}{2} e^{\frac{2u}{d}}\left(3e^{-2u} + ue^{-2u}\right)\exp\left(-\frac{(d-2)u^2}{2d}\right) ~ du\\
&= -C \left(\sum_i h_i\right)\E\left(e^{\frac{2U}{d}}\left(3e^{-2U} + Ue^{-2U}\right)\right),
\end{align*}

where $C$ is some positive constant and $U\sim N\left(0,\frac{d}{d-2}\right)$. Note that the last expectation can be computed quite easily, and we get that it is negative. Hence, for the log
$$b(h,\theta_1)\propto -\sum h_i.$$

This shows that $b(h,\theta_1)<0$ if $\sum h_i>0$ and $b(h,\theta_1)>0$ if $\sum h_i<0$. Using this, we can describe the expected behavior of the 1- and 2-sided tests. From\Cref{thm::both_tests_above_criticality_broad_regimes}, the power of the 1- and 2-sided tests can be described as follows:

\textbf{The case of $\sum h_i>0$:} In this case, $b(h,\theta_1)<0.$ Hence, we can use \Cref{thm::both_tests_above_criticality_broad_regimes} to derive the limiting power. Recall that we take $k_N = N^\gamma$ for some $\gamma<1$.
\begin{itemize}
    \item When $\delta < -\frac{1}{2} + \frac{2(1-\gamma)}{d}$, both tests have limiting power $\alpha.$
    \item When $-\frac{1}{2} + \frac{2(1-\gamma)}{d}<\delta < -\frac{2(1-\gamma)}{d},$ the limiting power of both tests is $1.$
    \item When $\delta>-\frac{2(1-\gamma)}{d},$ the limiting power of both tests is $1$
\end{itemize}
From the above, we see that both tests should have high power for $\delta>-\frac{1}{2}+\frac{2(1-\gamma)}{d}.$ For $d=25,$ this lower bound is extremely close to $-\frac{1}{2}$ for any value of $\gamma$, and we should see both tests transition from low to high power close to $\delta = -\frac{1}{2}.$ This is given in \Cref{fig::log_normal_family_pos_sum_lineplots}.

\textbf{The case of $h>0$:} In this case, $b(h,\theta_1)<0.$ Hence, we can use \Cref{thm::both_tests_above_criticality_broad_regimes} to derive the limiting power. Recall that we take $k_N = N^\gamma$ for some $\gamma<1$.
\begin{itemize}
    \item When $\delta < -\frac{1}{2} + \frac{2(1-\gamma)}{d}$, both tests have limiting power $\alpha.$
    \item When $-\frac{1}{2} + \frac{2(1-\gamma)}{d}<\delta < -\frac{2(1-\gamma)}{d},$ the limiting power of both tests is $1.$
    \item When $\delta>-\frac{2(1-\gamma)}{d},$ the limiting power of both tests is $1$
\end{itemize}
From the above, we see that both tests should have high power for $\delta>-\frac{1}{2}+\frac{2(1-\gamma)}{d}.$ For $d=25,$ this lower bound is extremely close to $-\frac{1}{2}$ for any value of $\gamma$, and we should see both tests transition from low to high power close to $\delta = -\frac{1}{2}.$ The power comparison is given in \Cref{fig::log_normal_family_neg_sum_lineplots}.

\begin{figure}
    \centering
    \captionsetup{width = 0.9\linewidth}
    \includegraphics[width=0.9\linewidth]{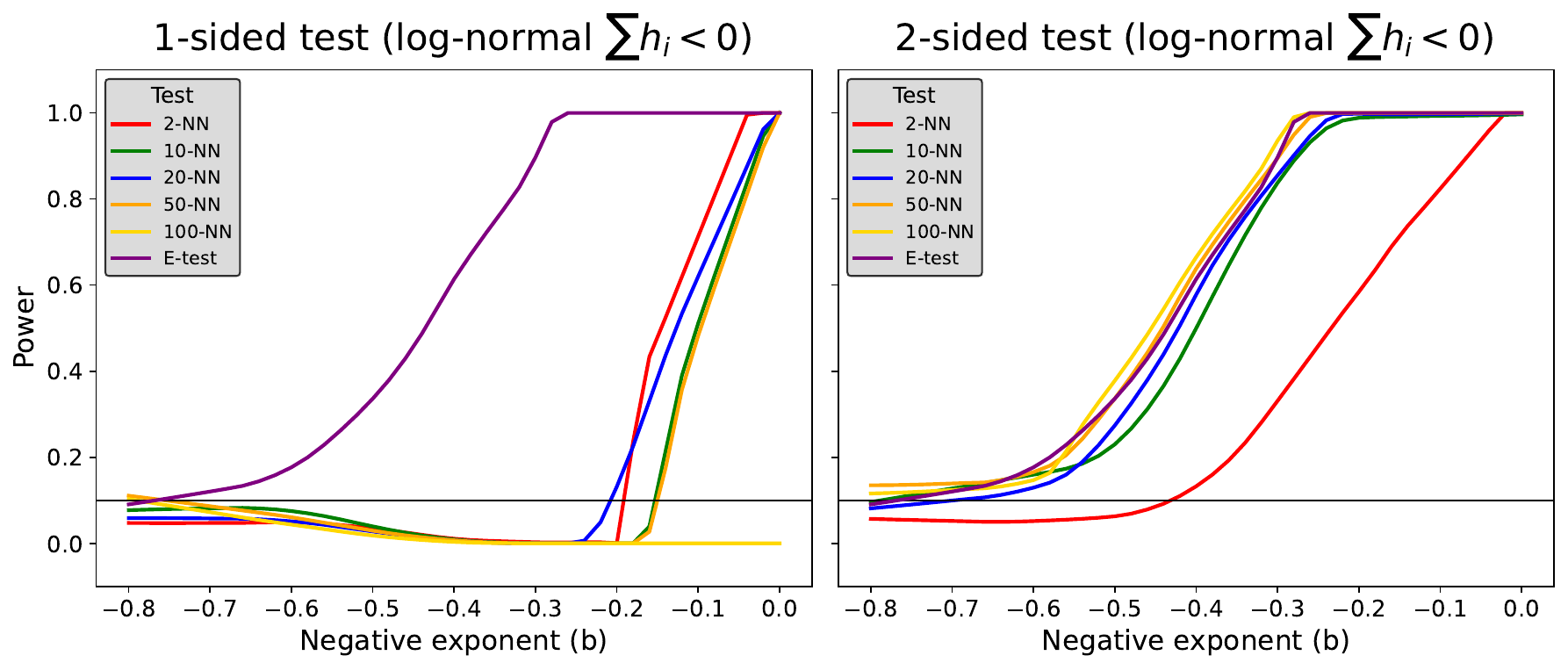}
    \caption{Power comparison in log normal family with $\sum h_i < 0$.}
    \label{fig::log_normal_family_neg_sum_lineplots}
\end{figure}

\subsection{Effect of dimension}

In this section, we will present some simulations to demonstrate the difference between the behavior of the tests when the ambient dimension $d$ is below criticality versus when it is above criticality. For this purpose, we consider the following 1-dimensional parametric family in $d$ dimensions. For $\theta>0,$
$$p(x\mid \theta) := \frac{1}{(2\pi \theta)^{\frac{d}{2}}} \exp\left(-\frac{1}{2\theta}\sum_{i=1}^d x_i^2 \right).$$

Note that the above family is a subset of the normal scale family described in \Cref{section::simulations}. Following the same example, we take the null to correspond to some fixed $\theta_1>0$ and the local alternative is taken to be $\theta_N = \theta_1 + hN^{b}$ for sample size $N$ and some negative exponent $b$. We will study the differences between the 1- and 2-sided tests for the cases of $h>0$ and $h<0$ for $d=2,4,6,10$.

The case of $h>0$ is shown in \Cref{fig::spherical_normal_across_dimension_positive_deviation}. In this case, for any given dimension the 1- and 2-sided tests do not display any significant difference in behavior for any value of $k$. This aligns with the results of \Cref{thm::power_below_criticality,thm::both_tests_above_criticality_broad_regimes}, and the specific calculations for the given normal scale family given in \Cref{section::simulations} which show that the detection thresholds of both tests are equal in all dimensions.

The case of $h<0$ is shown in \Cref{fig::spherical_normal_across_dimension_negative_deviation}. In this case, the differences between the 1- and 2-sided tests are more visible. As the dimension increases, the 1-sided tests for the highest values of $k$ get consistently worse, while the 2-sided tests for the same values of $k$ consistently improve. This is due to the exponent gap that arises when the ambient dimension crosses the critical dimension $d_c$.

The critical dimension $d_c(\gamma)$ for $k_N = N^\gamma$ is a function of the growth rate of the number of neighbors $k_N$ given \Cref{eq::critical_dimension_recalled}. The same equation shows that $d_c(\gamma)$ is non-increasing in $\gamma$. In particular, for a greater value of $\gamma$ the phase transition in the detection threshold occurs at lower dimensions, and the 1- and 2-sided tests start to differ at lower dimensions. This is demonstrated in \Cref{fig::spherical_normal_across_dimension_negative_deviation} where the 1-sided tests for higher values of $k$ lose power while the 2-sided tests increase in power, particularly for higher values of $k$.  

\begin{figure}[p]
\centering
\captionsetup{width = 0.8\linewidth}
\includegraphics[width = 0.7\linewidth]{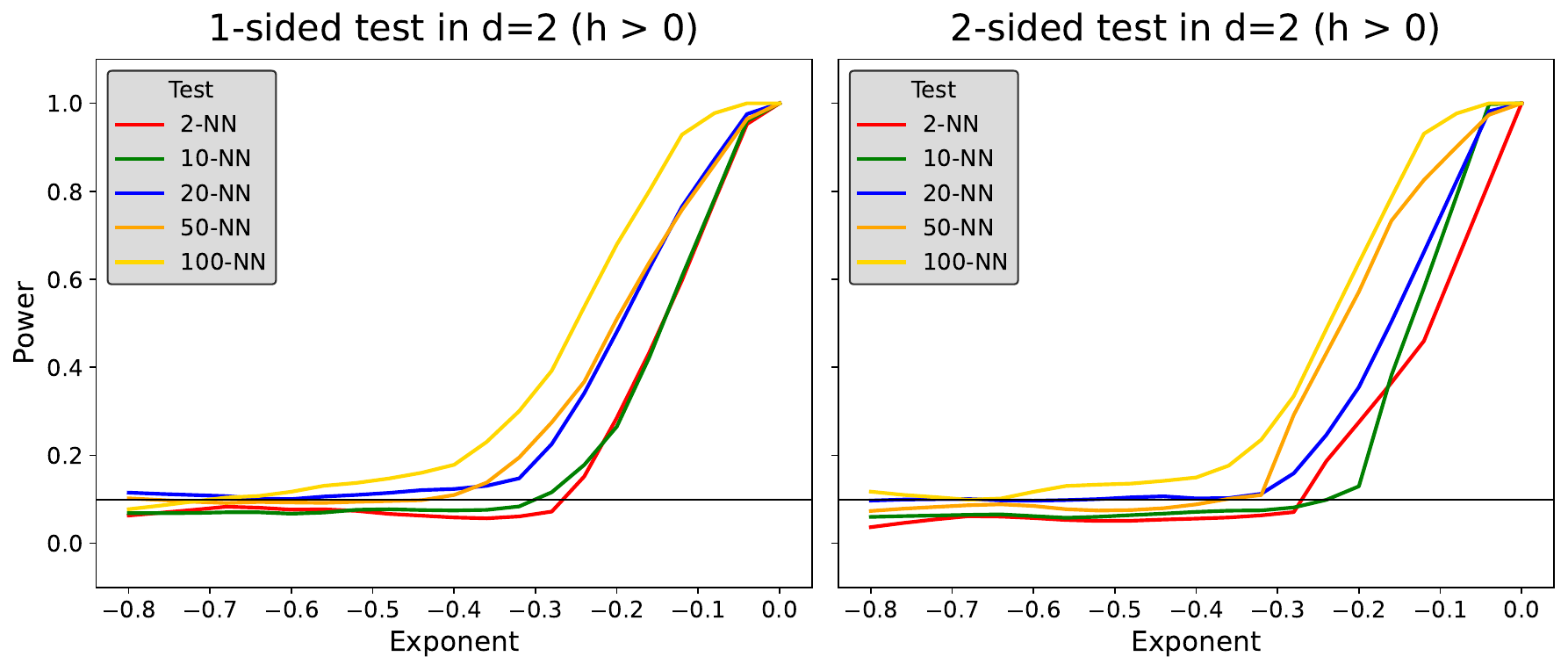}
\vfill
\includegraphics[width = 0.7\linewidth]{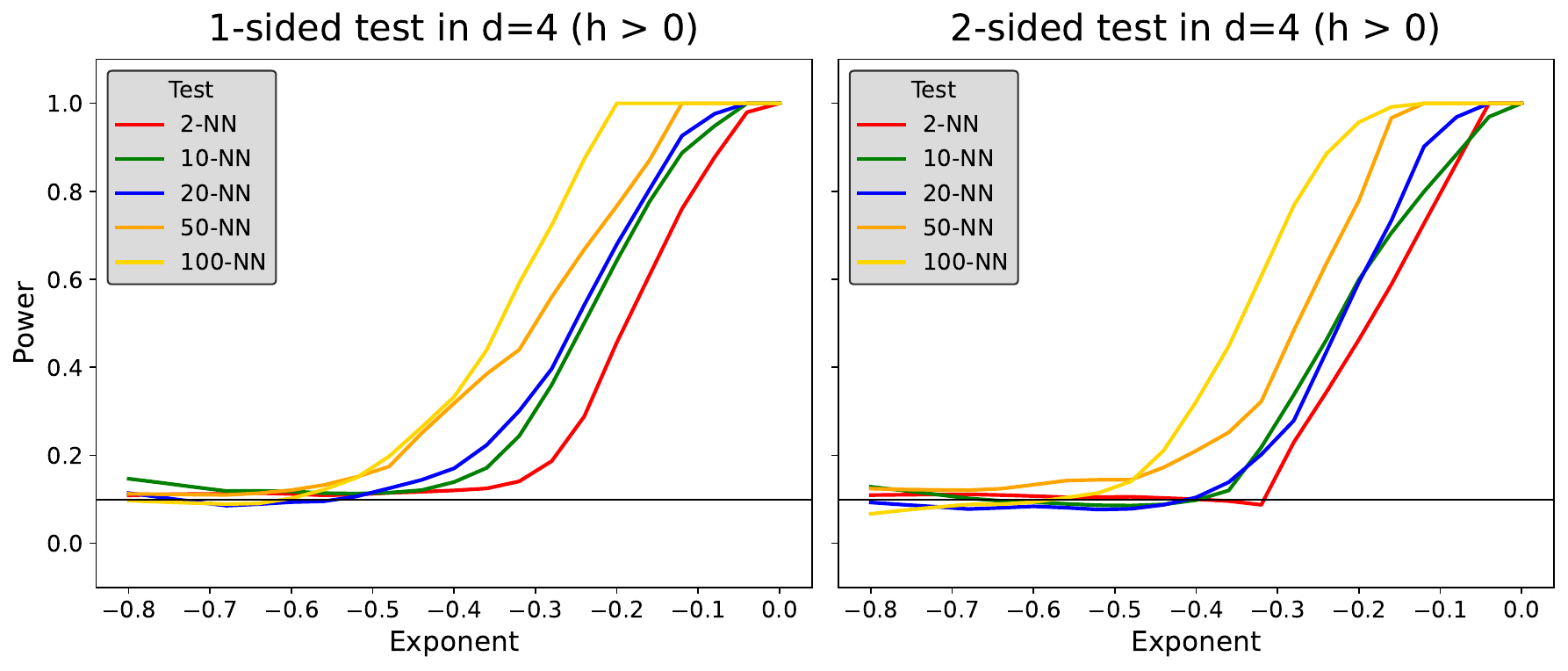}
\vfill
\includegraphics[width = 0.7\linewidth]{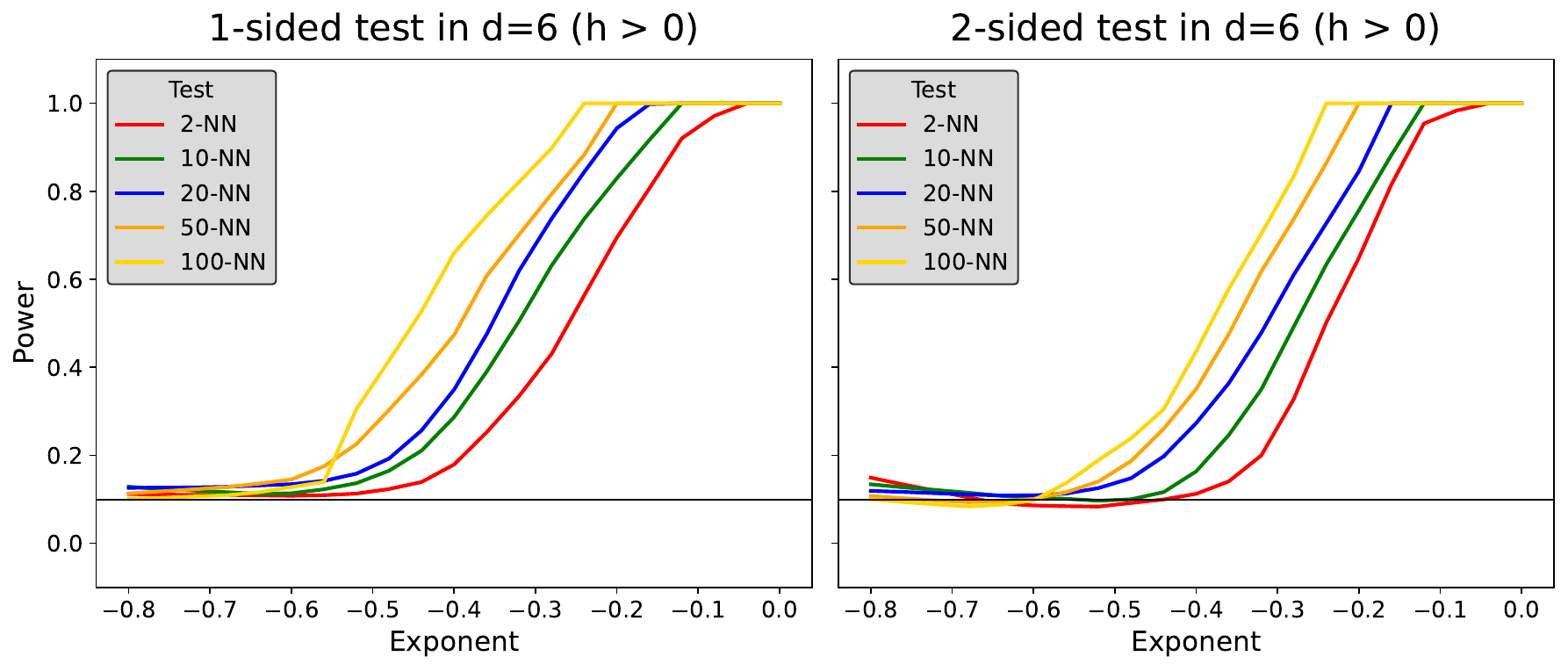}
\vfill
\includegraphics[width = 0.7\linewidth]{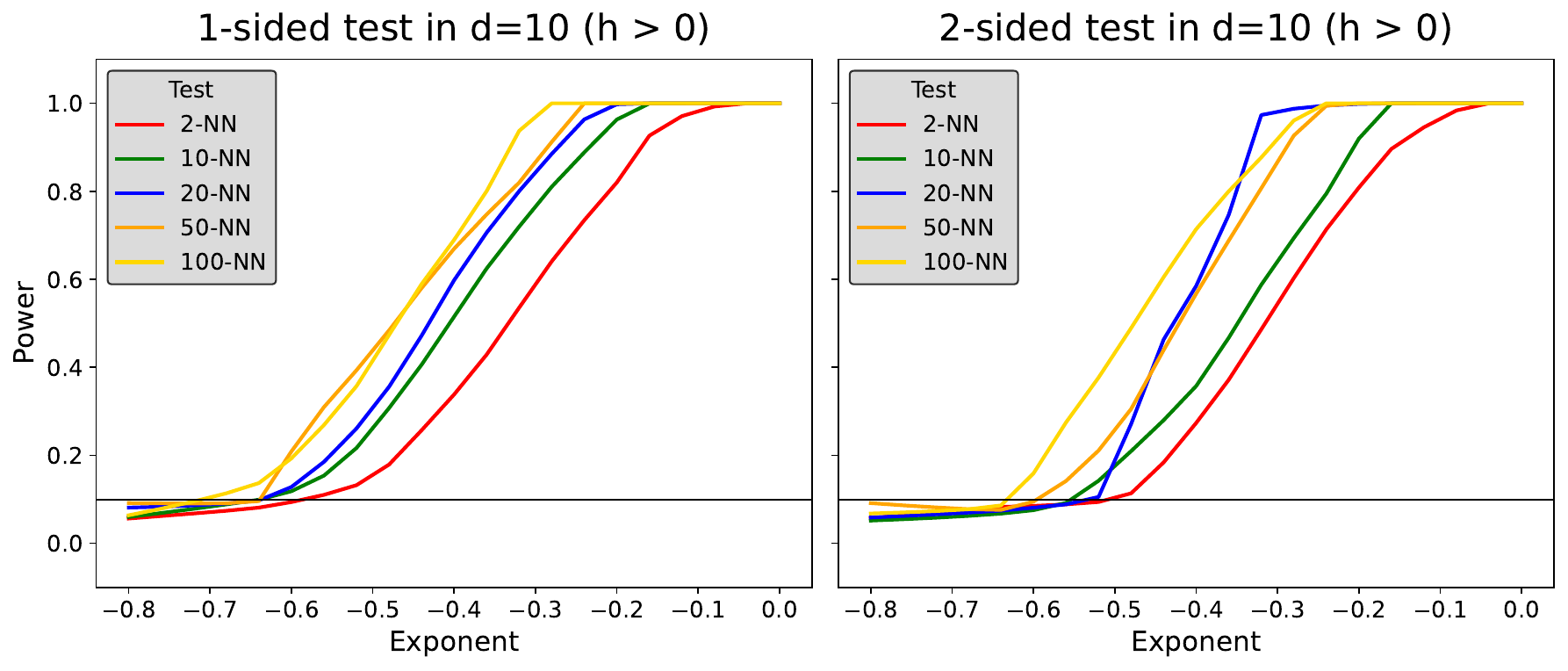}
\caption{}
\label{fig::spherical_normal_across_dimension_positive_deviation}
\end{figure}

\FloatBarrier

\begin{figure}[p]
\centering
\captionsetup{width = 0.7\linewidth}
\includegraphics[width = 0.7\linewidth]{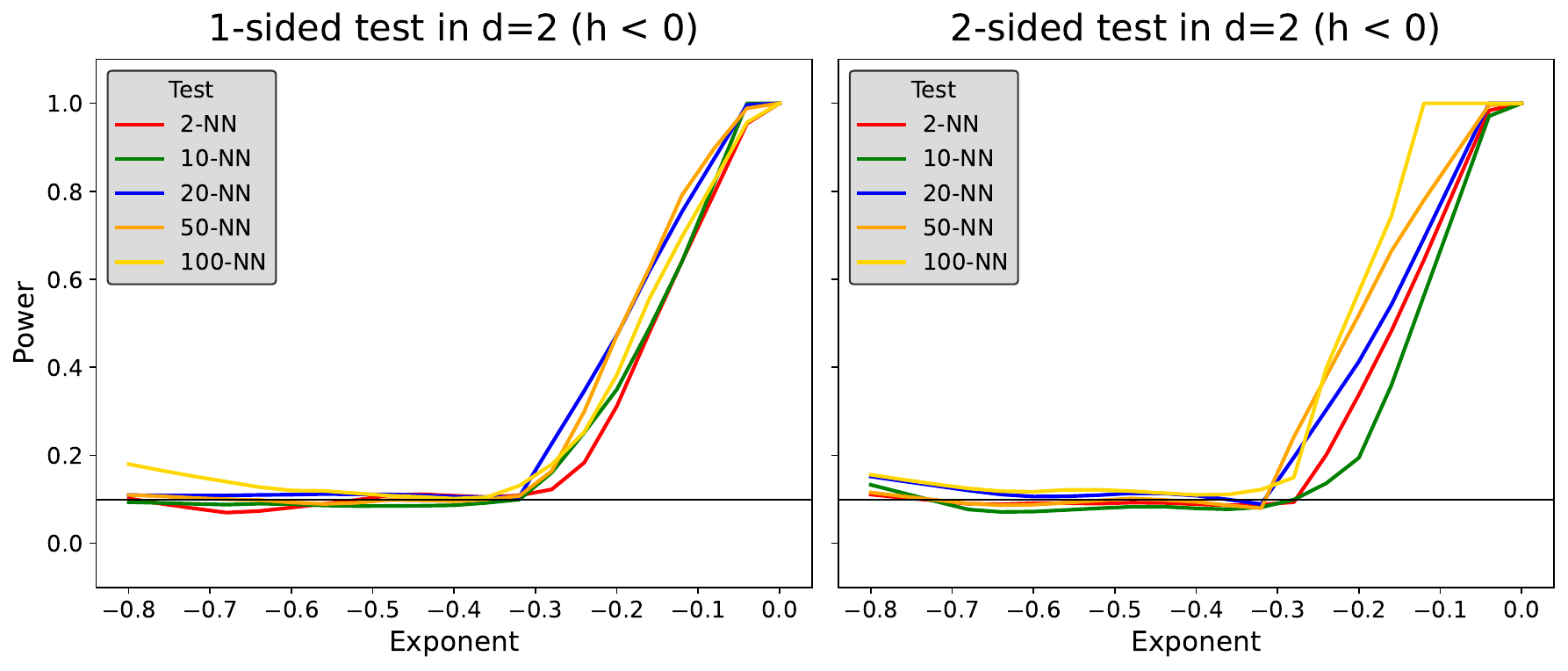}
\vfill
\includegraphics[width = 0.7\linewidth]{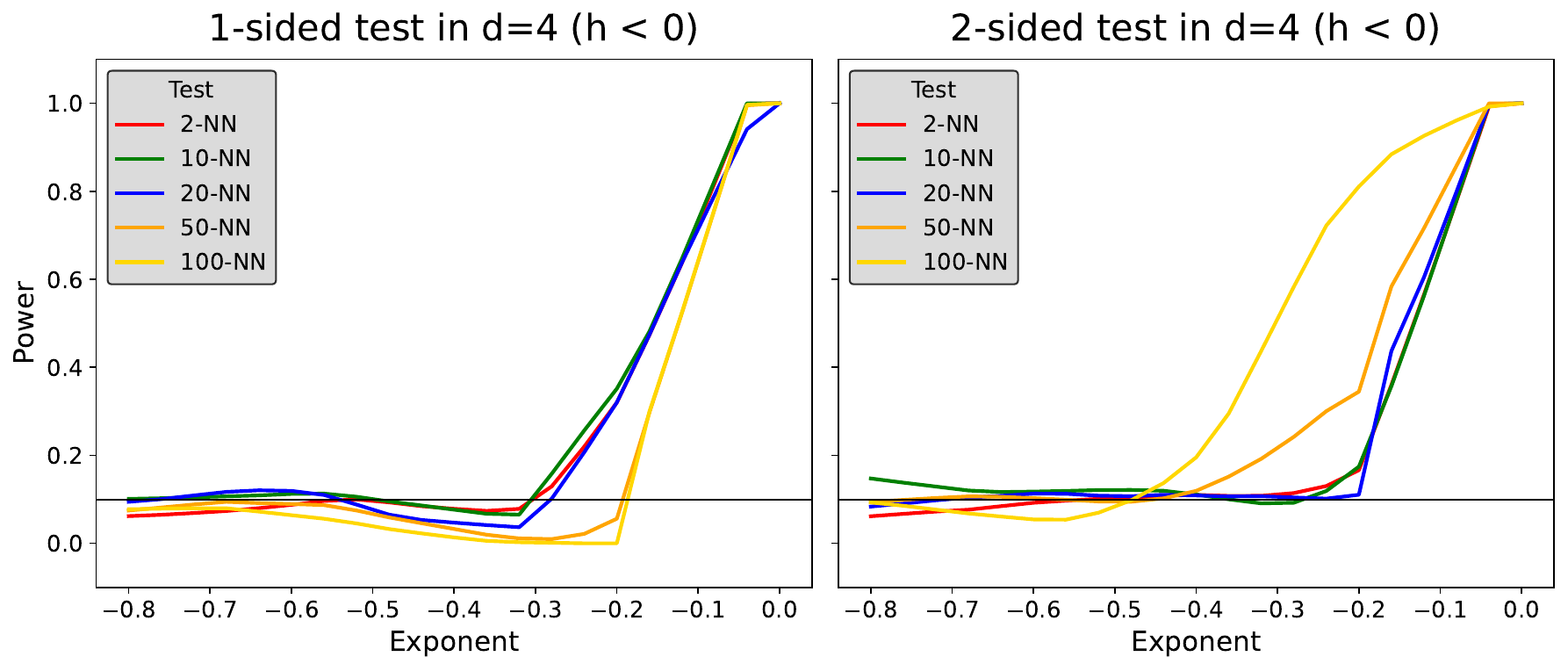}
\vfill
\includegraphics[width = 0.7\linewidth]{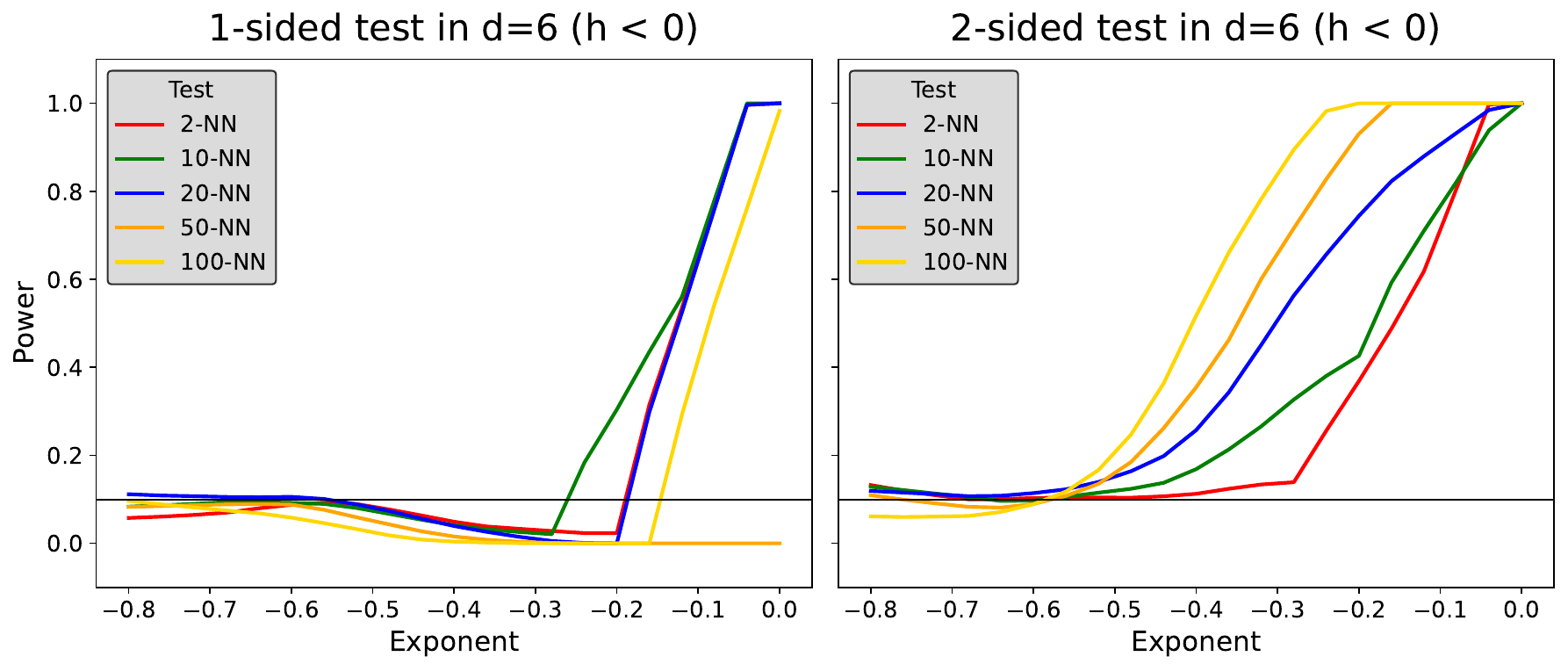}
\vfill
\includegraphics[width = 0.7\linewidth]{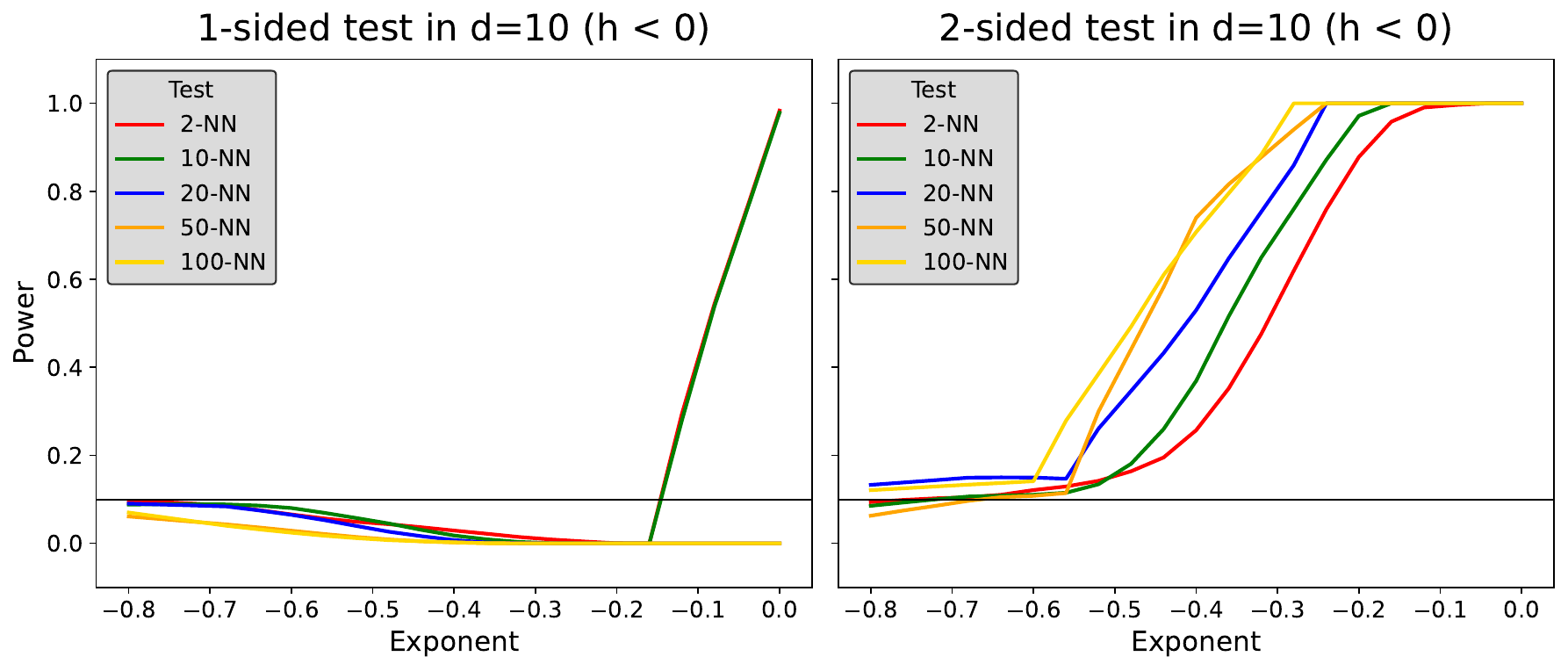}
\caption{}
\label{fig::spherical_normal_across_dimension_negative_deviation}

\end{figure}

\FloatBarrier

\subsection{Finite sample normal approximation}

In this section, we wish to assess the quality of the normal approximation for $T(\Gscr_{k_N})$ under the null. Specifically, we will be checking if the conditional statistic given in \Cref{thm::clt_conditional_stat} is close to the standard normal under the null. For calculating the conditional variance, we will use the approach outlined in \Cref{subsection::runtime_analysis}. In \Cref{fig::QQ_plots}, we have considered an i.i.d sample of size 2000 from the 10-dimensional standard Gaussian distribution. For a range of values of $K,$ we consider 500 trials which gives 500 i.i.d copies of standardized test statistic. \Cref{fig::QQ_plots} gives the QQ-plots of the 500 samples obtained for each value of $K$. We see that the quantile comparison is quite close to the reference line for $K$ up to 50. For $K=75, 100$ it is a little worse. However, overall the quantiles match very closely with the reference line, showing that the conditional statistic is indeed close a standard normal for a large enough sample size.

\begin{figure}[h]
\centering
\includegraphics[width=0.32\linewidth]{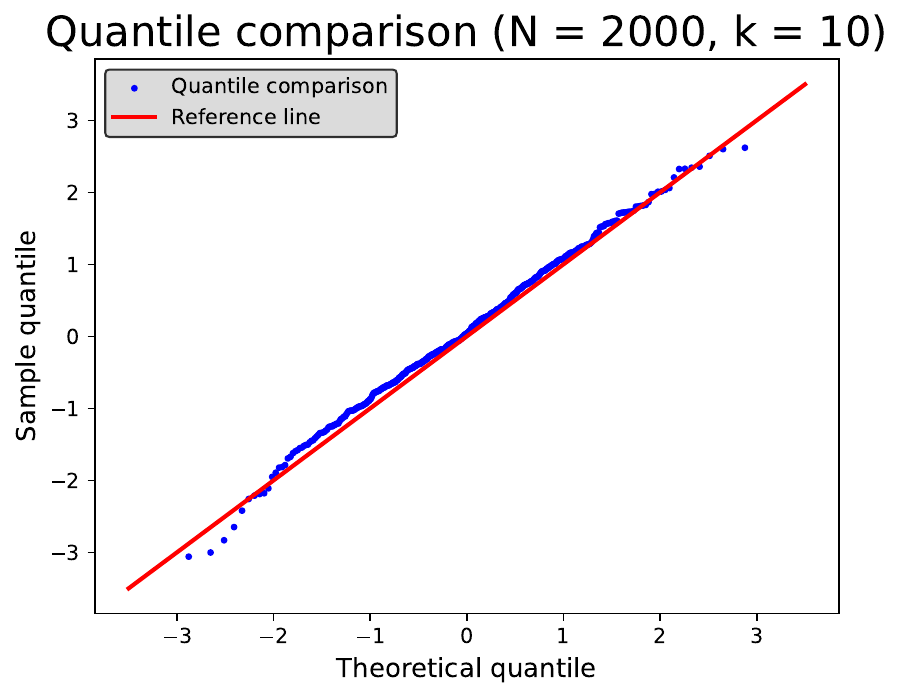}
\hfill
\includegraphics[width=0.32\linewidth]{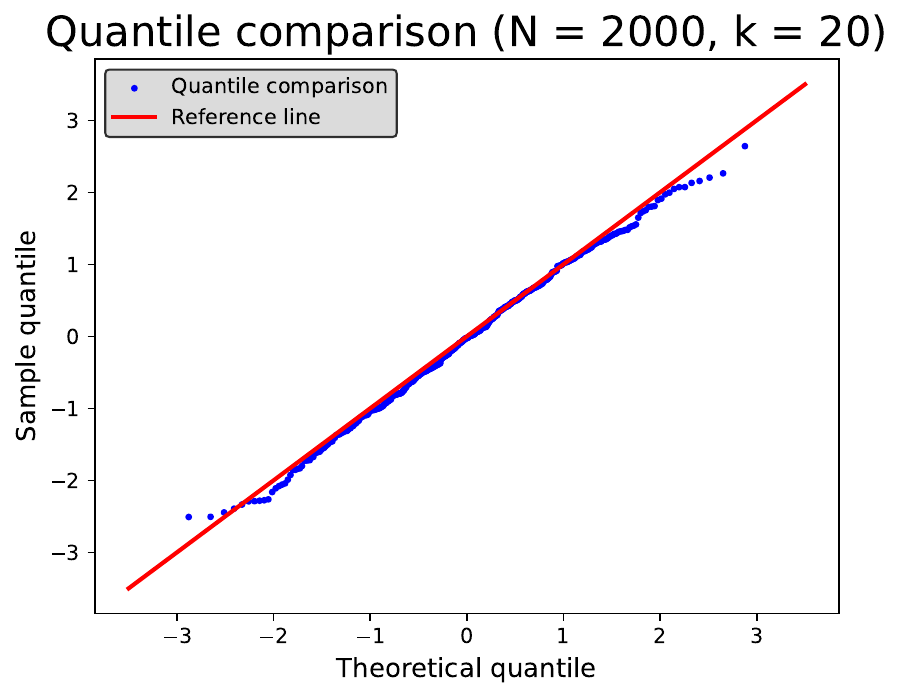}
\hfill
\includegraphics[width=0.32\linewidth]{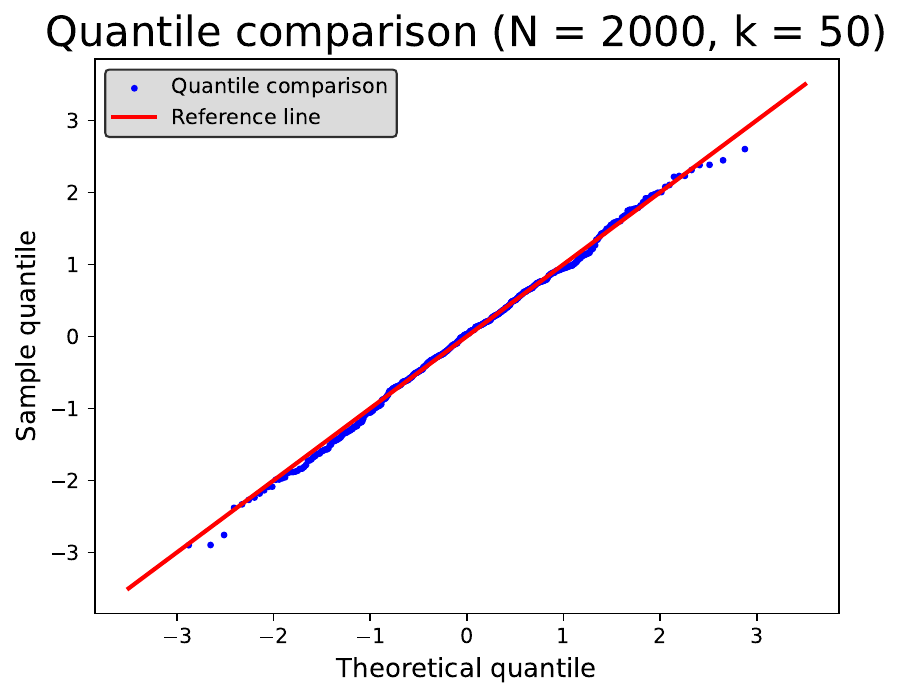}
\vspace{1cm}
\includegraphics[width=0.32\linewidth]{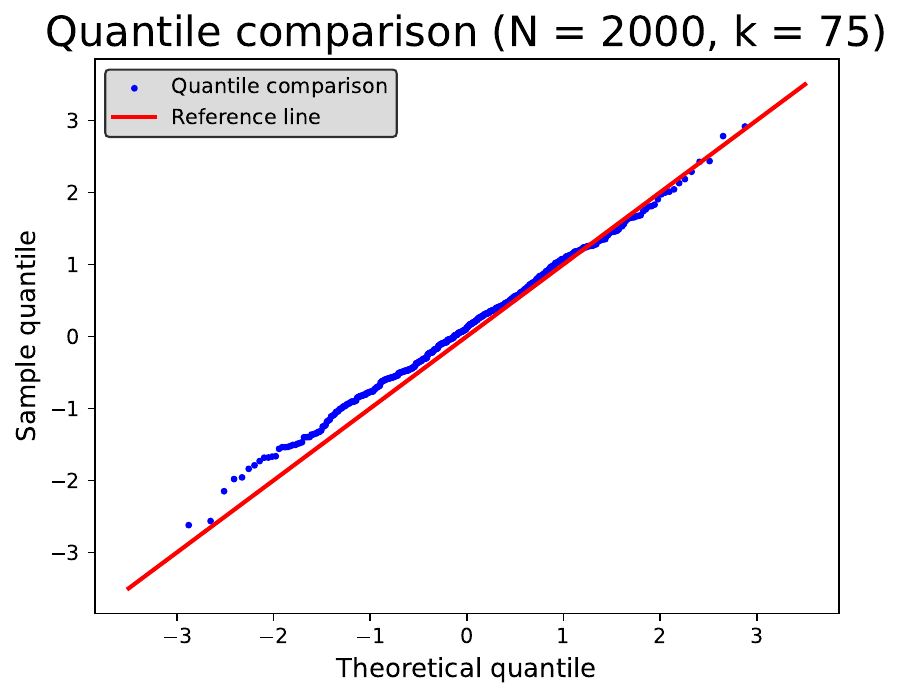}
\hfill
\includegraphics[width=0.32\linewidth]{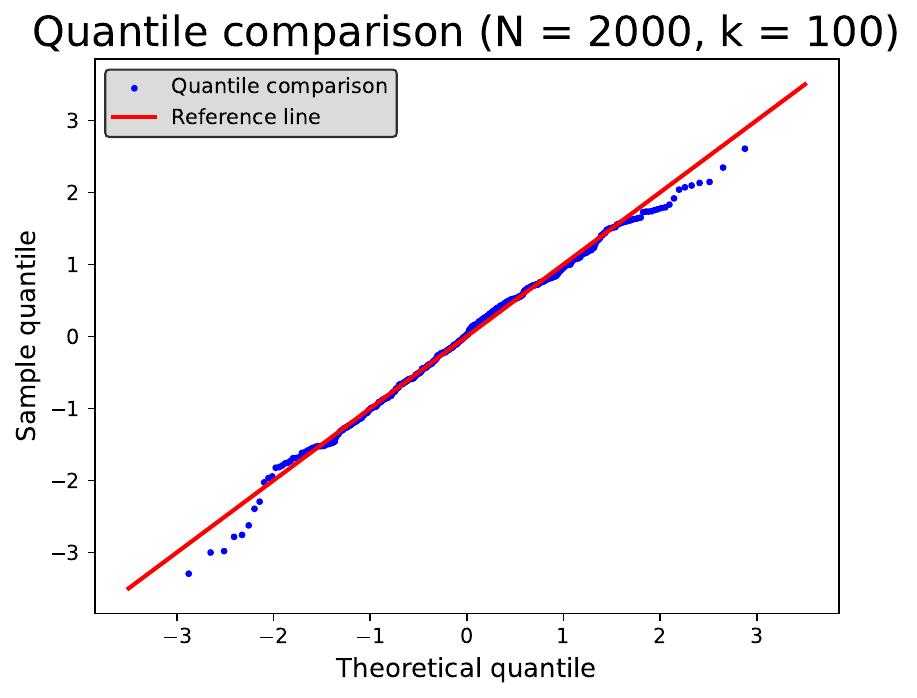}
\hfill
\includegraphics[width=0.32\linewidth]{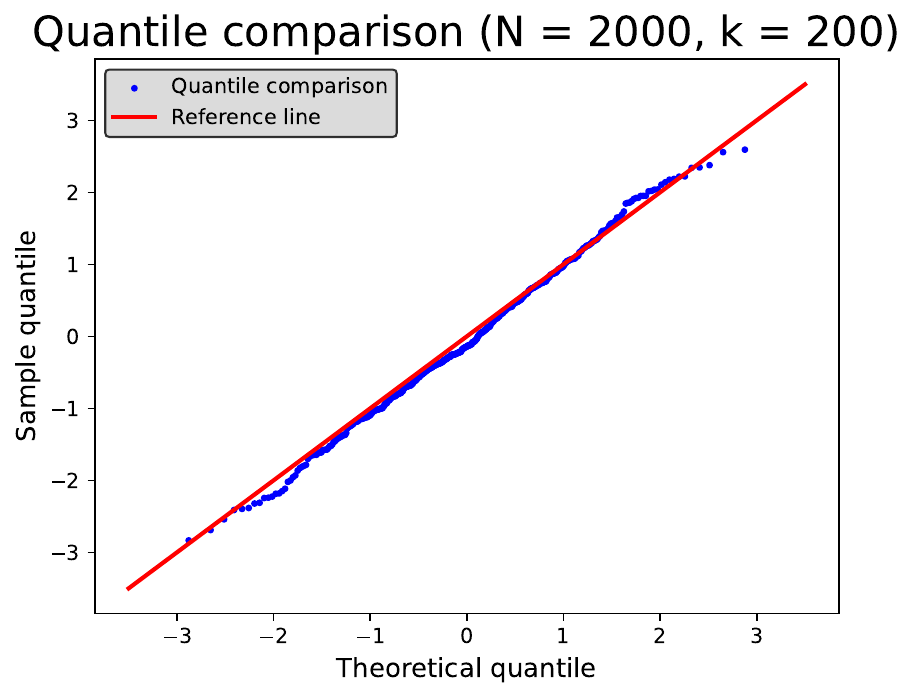}
\label{fig::QQ_plots}
\end{figure}

\subsection{Performance of test based on MST}

In this section, we consider the power of another graph based test in the same settings as before in order to compare with the $k$-NN test. A popular graph-based test is the one based on the minimal spanning tree (MST), \citet{friedmanrafskytest}. Just as increasing the number of neighbors provides denser versions of the nearest neighbors graphs, one can also obtain denser versions of the MST, known as $k$-MST's for $k\in \N$ which can be defined inductively.

The 1-MST is the usual MST. Given a finite set $S\subset \R^d,$ a spanning tree $\Tcal$ of $S$ is a connected, undirected graph with vertex set $S$ and no cycles. The length of a spanning tree is the sum of the lengths of all edges in the tree. A tree $\Tcal$ is called the MST of $S$ if its length is at most the length of any other spanning tree $\Tcal'$ of $S.$

Let $\Tcal_1(S)$ denote the MST of $S$. To construct the 2-MST, we consider all spanning trees of $S$ which use no edges from the 1-MST. The spanning tree $\Tcal_2(S)$ with minimum length among all such trees is known as the 2-MST. Similarly, the $k$-MST $\Tcal_k(S)$ is the minimum length spanning tree among all spanning trees of $S$ which use no edge belonging to $\Tcal_1(S) \cup \dots \cup \Tcal_{k-1}(S)$.

Using a $k$-MST as the underlying graph gives the test statistic to be the number of edges going across groups. As before, one can use either the 1-sided version which rejects when the statistic is too small, or the 2-sided version which rejects when the observed value of the statistic deviates significantly from the bulk of the distribution. To calibrate and implement the test, we use the permutation test based on resampling the labels. This is the same procedure as described for the $k$-NN test in \Cref{subsection::runtime_analysis} with the only change being that the underlying graph is now the $k$-MST.

\begin{figure}[ht]
\centering
\captionsetup{width = 0.8\linewidth}
\includegraphics[width=0.8\linewidth]{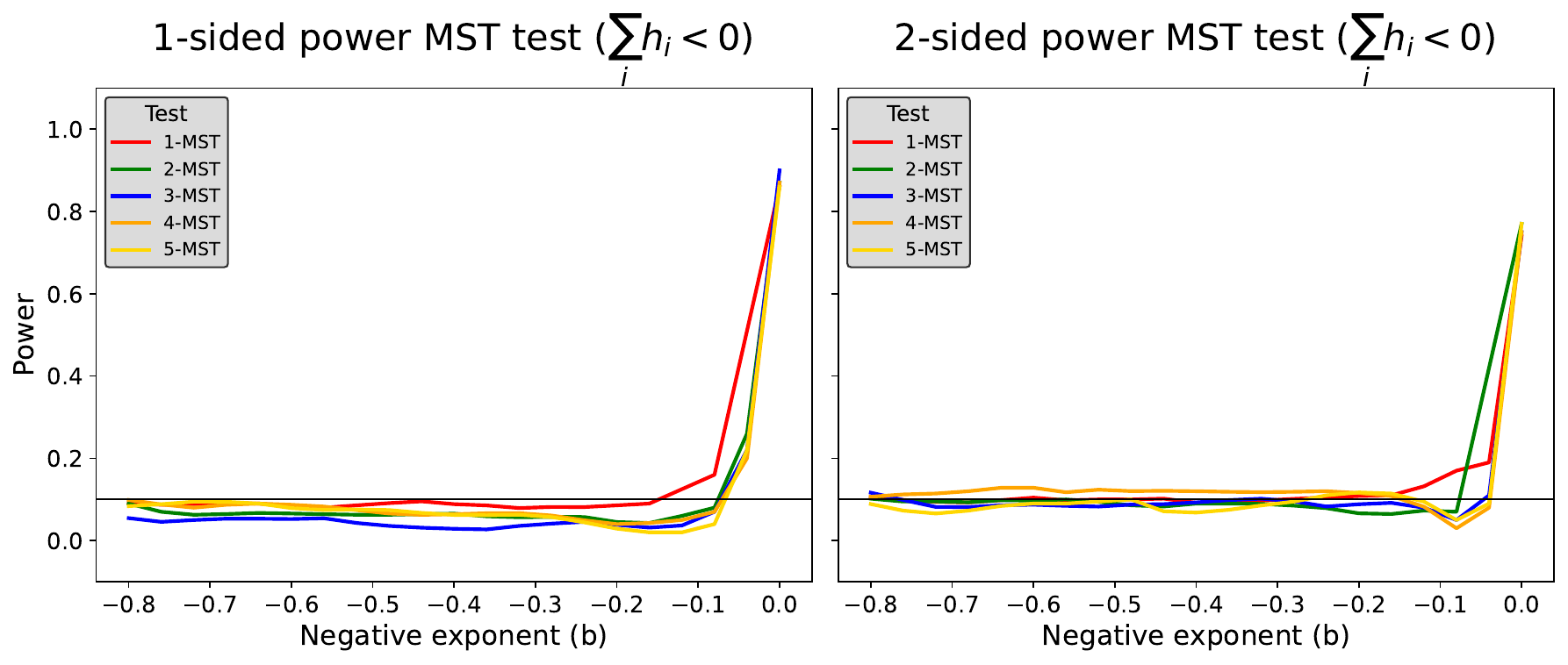}
\caption{Power comparison of MST-based test in normal scale family with $\sum h_i < 0$.}
\label{fig::MST_test_spiked_covariance_neg_sum_lineplots}
\end{figure}
\begin{figure}[ht]
\centering
\captionsetup{width=0.8\linewidth}
\includegraphics[width = 0.8\linewidth]{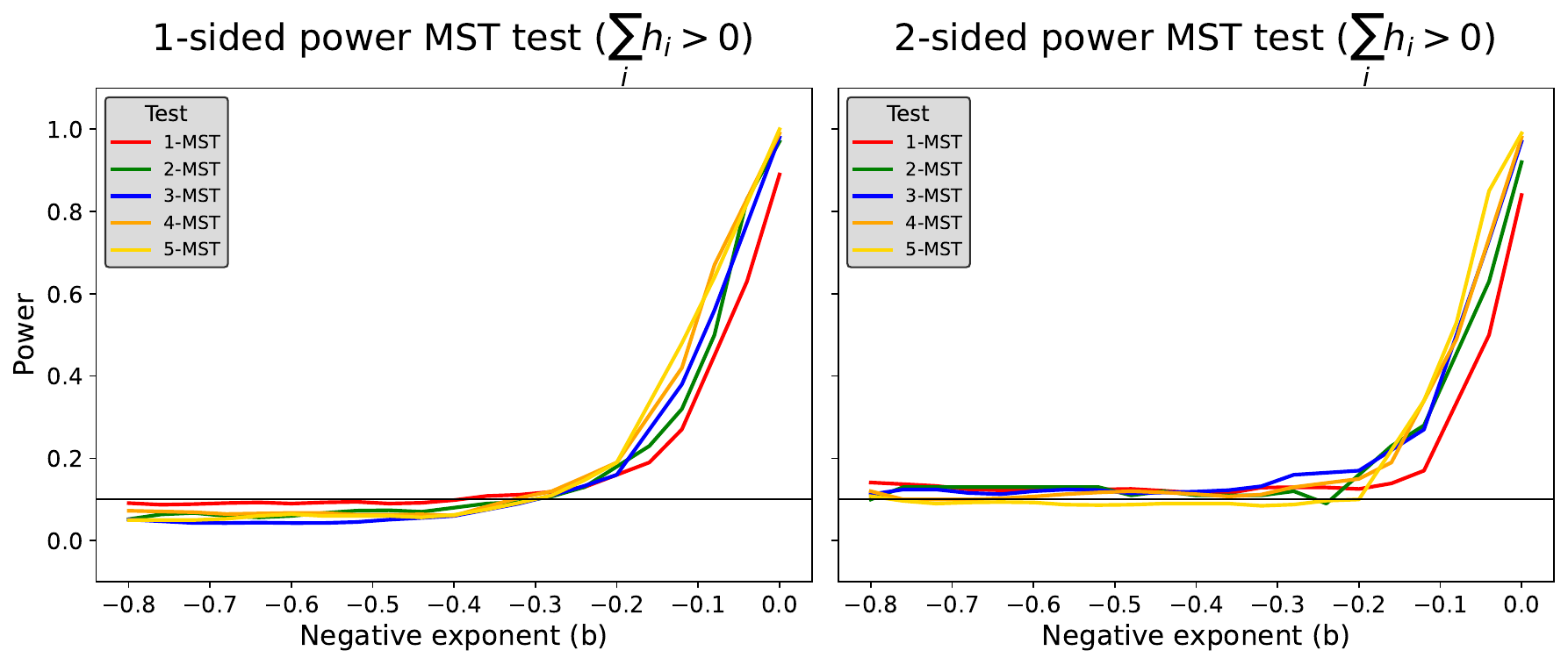}
\caption{Power comparison of 1- and 2-sided MST-based tests in normal scale family with $\sum h_i > 0$.}
\label{fig::MST_test_spiked_covariance_pos_sum_lineplots}
\end{figure}

We first consider the power of the 1- and 2-sided test in the normal scale family described in \Cref{section::simulations}. The simulation set up is exactly the same, with the null being $\theta_1 = \indicator_{10}$ and the local alternative given by $\theta_N = \theta_1 + hN^b$ for some $h\in \R^{10}$, some negative exponent $b$ and where $N$ denotes the sample size. The ambient dimension is once again $d=25$. However, for computational reasons we have taken the group sample sizes to be $N_1 = 900, N_2 = 600$ with a combined sample size of $N=1500.$

\Cref{fig::MST_test_spiked_covariance_neg_sum_lineplots,fig::MST_test_spiked_covariance_pos_sum_lineplots} plot the power of the 1- and 2-sided tests in the settings where $\sum h_i<0$ and $\sum h_i >0$ respectively. We see that there is no significant difference between the behavior of the 1- and 2-sided tests in this case. However, both tests seem to do marginally better in the case where $\sum h_i > 0.$ 

A similar sort of trend seems to appear in the case of the lognormal family as shown in \Cref{fig::MST_test_log_normal_neg_sum_lineplots,fig::MST_test_log_normal_pos_sum_lineplots}. Both test seems to be similarly in any given setting. However, in the setting where the direction of the alternate is such that $\sum h_i < 0$, the tests seem to increase in power when the negative exponent $b\approx 0.3,$ while for $\sum h_i>0$ the increase in power seems to occur around $b\approx 0.5$.

Thus, there is some indication that the effect of the direction on the power might be a phenomenon that affects other graph based tests apart from the one based on nearest neighbors. As a final point of interest, the loss of power experienced by the traditional, 1-sided implementation of the MST-based test was noted in \citet{chen2017new}.

\begin{figure}[ht]
\centering
\captionsetup{width = 0.8\linewidth}
\includegraphics[width=0.8\linewidth]{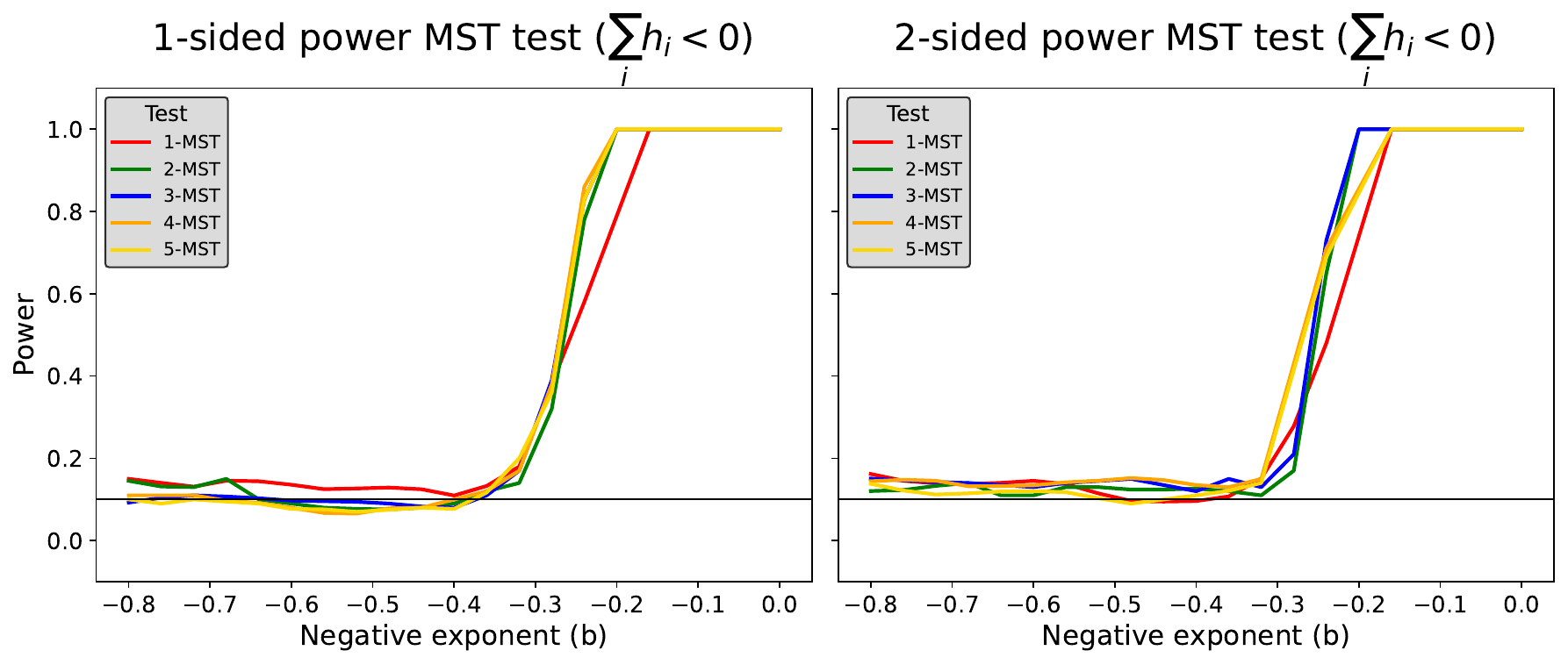}
\caption{Power comparison of MST-based test in lognormal location family with $\sum h_i < 0$.}
\label{fig::MST_test_log_normal_neg_sum_lineplots}
\end{figure}
\begin{figure}[ht]
\centering
\captionsetup{width=0.8\linewidth}
\includegraphics[width = 0.8\linewidth]{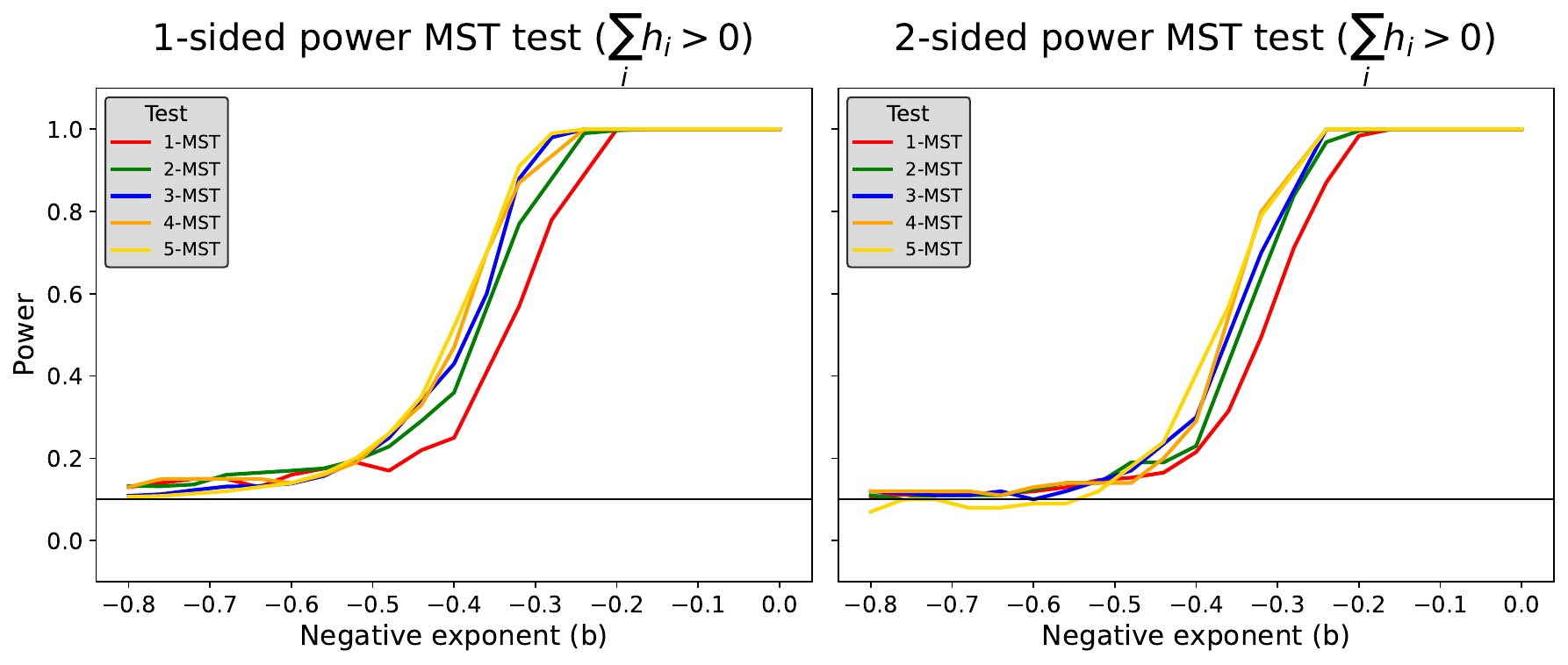}
\caption{Power comparison of 1- and 2-sided MST-based tests in lognormal location family with $\sum h_i > 0$.}
\label{fig::MST_test_log_normal_pos_sum_lineplots}
\end{figure}

\end{document}